\documentclass[reqno,11pt]{amsart}
\usepackage{amssymb, amsmath,latexsym,amsfonts,amsbsy, amsthm}
\usepackage{color}

\allowdisplaybreaks[4]

\let\e=\varepsilon

\let\p=\psi

\let\de=\delta

\def\p{\partial}

\newcommand{\beq}{\begin{equation}}
\newcommand{\eeq}{\end{equation}}
\newcommand{\ben}{\begin{eqnarray}}
\newcommand{\een}{\end{eqnarray}}
\newcommand{\beno}{\begin{eqnarray*}}
\newcommand{\eeno}{\end{eqnarray*}}

\renewcommand{\theequation}{\thesection.\arabic{equation}}

\newtheorem{theorem}{Theorem}[section]

\newtheorem{lemma}[theorem]{Lemma}
\newtheorem{proposition}[theorem]{Proposition}

\newtheorem{Theorem}{Theorem}[section]

\newtheorem{Proposition}[Theorem]{Proposition}
\newtheorem{Lemma}[Theorem]{Lemma}
\newtheorem{Corollary}[Theorem]{Corollary}
\newtheorem{Remark}[Theorem]{Remark}

\newcommand{\ude}{u^\delta}

\begin{document}
\title[inviscid limit in unit]
{Prandtl-Batchelor flows on a disk}

%
\author{Mingwen Fei}
\address{School of  Mathematics and Statistics, Anhui Normal University, Wuhu 241002, China}
\email{mwfei@ahnu.edu.cn}

\author{Chen Gao}
\address{Beijing International Center for Mathematical Research, Peking University, Beijing,100871, China}
\email{gaochen@amss.ac.cn}

\author{Zhiwu Lin}
\address{School of Mathematics, Georgia Institute of Technology, 30332, Atlanta, GA, USA}
\email{zlin@math.gatech.edu}

\author{Tao Tao}
\address{School of  Mathematics, Shandong University, jinan, Shandong, 250100, China}
\email{taotao@sdu.edu.cn}

\date{\today}
\maketitle

\renewcommand{\theequation}{\thesection.\arabic{equation}}
\setcounter{equation}{0}
\begin{abstract}
For steady two-dimensional flows with a single eddy (i.e. nested closed
streamlines), Prandtl (1905) and Batchelor (1956) proposed that in the limit
of vanishing viscosity the vorticity is constant in an inner region separated
from the boundary layer. In this paper, by constructing higher order approximate solutions of the
Navier-Stokes equations and establishing the validity of Prandtl boundary layer expansion, we give a rigorous proof of the
existence of Prandtl-Batchelor flows on a disk with the wall velocity slightly
different from the rigid-rotation. The leading order term of the flow is the
constant vorticity solution (i.e. rigid rotation) satisfying Batchelor-Wood
formula.
\end{abstract}

\numberwithin{equation}{section}

\indent

\section{Introduction}

\indent

In his famous paper \cite{prandtl} on the birth of boundary layer theory,
Prandtl (1904) noted that in the limit of infinite Reynolds number, the
vorticity of steady two-dimensional laminar flows becomes constant within a
region of nested closed streamlines (i.e. a single eddy). The same property
was later rediscovered by Batchelor (1956) in \cite{B}. See also the
conference abstract \cite{feymann-lagerstrom} of Feynman and Lagerstrom
(1956). This class of result is now usually referred to as Prandtl-Batchelor
theory and such laminar flows are called Prandtl-Batchelor (PB) flows in the
literature. For PB flows on a circular disk, the formula of the limiting
vorticity constant was given in (\cite{B,feymann-lagerstrom,W})
and is usually referred to as the Batchelor-Wood formula. For PB flows on more
general domains, it is more difficult to determine the limiting vorticity
constant and partial results were given in (\cite{edwards,feymann-lagerstrom, kim-formula,riley81,van wijngarden,W}). The
Prandtl-Batchelor theory plays an important role in many studies involving
laminar flows with small viscosity, for example, the nonlinear critical
layer theory near shear flows in \cite{maslowe}. It also implies a selection
mechanism of the Navier-Stokes equations in the inviscid limit. Although the
Euler equations have infinitely many steady solutions, only those Euler
solutions satisfying the PB theory can appear in the inviscid limit. That is,
the vorticity of the compatible Euler solution must be a constant within each eddy.

Moreover, the Prandtl-Batchelor theory can be applied to general 2D advection
diffusion equation of a passive scalar field $\theta\left(  x,y\right)  $
\[
u\cdot\nabla\theta-R^{-1}\Delta\theta=0,
\]
where $u\left(  x,y\right)  $ is a given steady incompressible flow, and
$R^{-1}$ is the diffusion coefficient. Then Prandtl-Batchelor theory implies
that when $R^{-1}$ tends to zero, $\theta$ should tend to a constant within
any single eddy associated with the flow $u$. The first example is the
homogenization of potential vorticity in the ocean circulation theory
(\cite{ocean-pedlosky,rhines-young83,rhines-young82}), where
$\theta$ is the potential vorticity in the 2D quasi-geostrophic model and $R$
is the Peclet number. Another example is the flux expulsion in the
self-excited dynamo theory (\cite{dormy-moffatt-19,moffatt78,weiss66}), where $\theta$ is the magnetic potential and $R$ is
the magnetic Reynolds number. Then PB theory implies that within each eddy of
the flow $u$, the magnetic field (i.e. $\nabla\theta$) should become zero when
$R^{-1}$ tends to zero.

Despite the importance of PB theory and its wide applications to fluids, there
is relatively little mathematical work on this problem. In the existing
``proof" of PB theory stemming from \cite{B,prandtl}, the eddy structure
of the laminar flow is a priori assumed. In the Appendix, we give such a proof
for the case of a single eddy. It is based on the assumptions that: for
sufficiently small viscosity i) the steady flow of Navier-Stokes has a single
eddy (i.e. no hyperbolic stagnation point of the stream function); ii) any
interior domain is separated from the boundary layer uniformly for vanishing viscosity; iii) inside
the interior domain, steady NS solutions tend to a steady Euler solution in
$C^{2}$. However, given a domain and the boundary condition, it is difficult
to understand the eddy structures (i.e. streamline structures) of the steady
solutions of NS equations and control the boundary layer. The above assumptions
were never verified in any case rigorously. In a series of works
(\cite{K1998,K2000,K2003,kim-childress,kim-free boundary,kim-thesis,kim-formula}), Kim initiated a mathematical study of PB flows on a
disk. In particular, when the boundary velocity is slightly different from a
constant, the well-posedness of the Prandtl equation under the Batchelor-Wood
condition was shown in \cite{K2000} and some asymptotic study of the boundary
layer expansion was given in \cite{K1998}. However, it remains difficult to
show the convergence of the boundary layer expansion to the steady
Navier-Stokes solution.

In this paper, we give the first proof of the existence of PB flow on a disk.
As in \cite{K2000}, we assume the boundary condition to be slightly different
from a constant rotation. We first construct steady solutions of Euler equations
(leading order term of the Navier-Stokes equations) to be the constant vorticity solution (i.e. rigid rotation)
satisfying the Batchelor-Wood formula. Then, by constructing higher order approximate solutions and establishing the convergence of Prandtl boundary layer expansion, we construct a class of Prandtl-Batchelor flows to the steady Navier-Stokes equations on a disk.
More precisely, we consider the steady Navier-Stokes equations in two dimensional disk $B_1(0)$
\begin{eqnarray}
\left \{
\begin {array}{ll}
\mathbf{u}^\varepsilon\cdot\nabla\mathbf{u}^\varepsilon+\nabla p^\varepsilon-\varepsilon^2\Delta\mathbf{u}^\varepsilon=0,\\[5pt]
\nabla\cdot \mathbf{u}^\varepsilon=0,
\end{array}
\right.\label{ns}
\end{eqnarray}
with rotating boundary
\begin{align}\label{rotating boundary condition}
\mathbf{u}^\varepsilon\big|_{\partial B_1}=(\alpha+\eta f(\theta))\mathbf{t},
\end{align}
%
where  $\varepsilon>0$ is reciprocal to Reynolds number,  $\mathbf{u}^\varepsilon$ is the velocity, $p^\varepsilon$ is the pressure,  $\alpha>0$, $\eta$ is a  small number, $\mathbf{t}$ is the unit tangential vector to $\partial B_1$, and $f(\theta)$ is a $2\pi$-periodic smooth function. Let $u^\varepsilon(\theta,r)$ be the tangential component and $v^\varepsilon(\theta,r)$ be the normal component of $\mathbf{u}^\varepsilon$ in polar coordinates,
then  (\ref{ns}) reads
\begin{eqnarray}
\left \{
\begin {array}{lll}
u^\varepsilon u^\varepsilon_\theta+rv^\varepsilon u^\varepsilon_r+u^\varepsilon v^\varepsilon+p^\varepsilon_\theta-\varepsilon^2\big(ru^\varepsilon_{rr}
 + u^\varepsilon_{r}+\frac{u^\varepsilon_{\theta\theta}}{r}+\frac{2}{r} v^\varepsilon_\theta-\frac{ u^\varepsilon}{r}\big)=0,\\[5pt]
u^\varepsilon v^\varepsilon_\theta+rv^\varepsilon v^\varepsilon_r-(u^\varepsilon)^2+rp^\varepsilon_r-\varepsilon^2\big(rv^\varepsilon_{rr}
 +v^\varepsilon_{r}+\frac{ v^\varepsilon_{\theta\theta}}{r}-\frac{2}{r} u^\varepsilon_\theta-\frac{ v^\varepsilon}{r}\big)=0,\\[5pt]
 u^\varepsilon_\theta+rv^\varepsilon_r+ v^\varepsilon=0,\\[5pt]
 u^\varepsilon(\theta,1)=\alpha+\eta f(\theta),~~ v^\varepsilon(\theta,1)=0,\label{NS-curvilnear}
\end{array}
\right.
\end{eqnarray}
where $(\theta,r)\in\Omega:=[0,2\pi]\times (0,1]$.
Formally, as $\varepsilon\rightarrow 0$,
we obtain the steady Euler equations
\begin{align}\label{equ:Euler}
\left\{
\begin{aligned}
& u_e\partial_\theta u_e+v_er\partial_ru_e+u_ev_e+\partial_\theta p_e=0,\\
&u_e\partial_\theta v_e+v_er\partial_rv_e-(u_e)^2+r\partial_rp_e =0,\\
&\partial_\theta u_e+\partial_r(rv_e)=0.
\end{aligned}
\right.
\end{align}

We will show the existence of solution $(u^\varepsilon, v^\varepsilon)$ to (\ref{NS-curvilnear}) which converges to  a solution of steady Euler equations (\ref{equ:Euler}) with constant vorticity as $\varepsilon\rightarrow 0$. We first choose a steady Euler flow with constant vorticity (rigid rotation) which satisfies the Batchelor-Wood formula, then construct a solution $(u^\varepsilon, v^\varepsilon)$ to (\ref{NS-curvilnear}) by perturbing this steady Euler flow.
However, in general, there is a mismatch between the tangential velocities of the Euler flow $u_e$  and the prescribed Navier-Stokes flows $u^\varepsilon$. Due to the mismatch on the boundary, Prandtl in 1904 formally introduced the boundary layer theory to correct this mismatch, but the justification of this formal boundary expansion is a challenging problem. A first step for the justification of this formal boundary expansion  is to understand the approximate terms in the formal expansion, that is the steady Prandtl equations and linearized steady Prandtl equations, see \cite{DM, GI3,OS, renardy, SWZ, WZ}. For the validity of Prandtl boundary layer expansion, there are some important results in recent years. For the moving boundary, Guo and Nguyen made the first important progress in \cite{GN} where they considered the Navier-Stokes equations over a moving plate. Then Iyer in \cite{Iyer1} extended Guo-Nguyen's result to the case of  a rotating disk by considering the curvature effect. The leading order of Euler flows in \cite{GN} and \cite{Iyer1} are both shear flows, and the width of the region in \cite{GN} and the angle of sector in \cite{Iyer1} are small. Later, Iyer in \cite{Iyer2} justified the global steady Prandtl expansions over a moving plane under the assumption of the smallness of the mismatch, and considered the situation that Euler flow is a perturbation of shear flow in \cite{Iyer3}. For the no-slip boundary, there are also some important works. In \cite{GI1}, Guo and Iyer justified the validity of 2d steady Prandtl layer expansion in a narrow region, and later they extended their result in  \cite{GI2} where they considered the 2d steady Navier-Stokes equations with external force. Gao and Zhang in \cite{GZ} justified the validity of 2d steady Prandtl layer expansion by introducing the stream function in the error estimates and removed the smallness of the width for shear flow. Then, Iyer and Masmoudi in \cite{IM1,IM2} justified the global-in-$x$ steady Prandtl boundary layer expansion. Very recently, Gao and Zhang in \cite{GZ2} justified the Prandtl expansion under the situation of non-shear Euler flow when the width of the region is small. Moreover,  the stability in Sobolev space for some class of shear flows of Prandtl type has been studied by Gerard-Varet and Maekawa in \cite{GMM}, see also \cite{CWZ}.

In the current work, we assume that $\int_0^{2\pi}f(\theta)d\theta=0$ and the leading order of Euler flow $(u_e(\theta,r),v_e(\theta,r))$ is the Couette flow
$$(u_e,v_e)=(ar, 0).$$
There is an important basis for this choice---Prandtl-Batchelor theory in \cite{B}(also see Appendix C). This theory shows that if the Euler flow $(u_e(\theta,r),v_e(\theta,r))$ in the disk $B_1(0)$ is the vanishing viscosity limit of Navier-Stokes flow whose streamlines are closed, then it must be the Couette flow
$$(u_e(\theta,r),v_e(\theta,r))=(ar, 0),$$
where $a$ is a constant. While the Batchelor-Wood formula in \cite{W} shows
\begin{align*}
u_{e}^2(1)=\alpha^2+\frac{\eta^2}{2\pi}\int_0^{2\pi}f^2(\theta)d\theta.
\end{align*}
 So we take the leading order of Euler flow $(u_e(\theta,r),v_e(\theta,r))$ as follows
\begin{align}\label{Taylor-Couette flow}
u_e(\theta,r)=u_{e}(r):=ar, \ v_e(\theta,r)=0,
\end{align}
where
\begin{align*}
a=\Big(\alpha^2+\frac{\eta^2}{2\pi}\int_0^{2\pi}f^2(\theta)d\theta\Big)^{\frac12}.
\end{align*}
Then, we introduce the steady Prandtl equations near $r=1$
\begin{align}
\left \{
\begin {array}{ll}
\big(u_e(1)+u_p^{(0)}\big)\partial_\theta u_p^{(0)}+\big(v_p^{(1)}-v_p^{(1)}(\theta,0)\big)\partial_Yu_p^{(0)}-\partial_{YY}u_p^{(0)}=0,\\[5pt]
\partial_\theta u_p^{(0)}+\partial_Yv_p^{(1)}=0,\\[5pt]
u_p^{(0)}(\theta,Y)=u_p^{(0)}(\theta+2\pi,Y),\quad v_p^{(1)}(\theta,Y)=v_p^{(1)}(\theta+2\pi,Y),\\[5pt]
u_p^{(0)}\big|_{Y=0}=\alpha+\eta f(\theta)-u_e(1),\quad \lim_{Y\rightarrow -\infty}(u_p^{(0)},v_p^{(1)})=(0,0).
\label{prandtl problem near 1}
\end{array}
\right.
\end{align}

The above steady Prandtl equations will be derived by matched asymptotic expansion and the solvability will be studied in next section.

Now our main theorem is stated as follows.

\begin{Theorem}\label{main theorem}
Assume that $f(\theta)$ is a $2\pi$-periodic smooth function which satisfies
$\int_0^{2\pi}f(\theta)d\theta=0$, then there exist $\varepsilon_0>0, \eta_0>0$ such that for any $\varepsilon\in (0,\varepsilon_0), \eta\in (0,\eta_0)$, the Navier-Stokes equations (\ref{NS-curvilnear}) have a solution $(u^\varepsilon(\theta,r), v^\varepsilon(\theta,r))$ which satisfies
\beno
\Big\|u^\varepsilon(\theta,r)-u_e(r)-u_p^{(0)}\Big(\theta,\frac{r-1}{\varepsilon}\Big)
\Big\|_{L^\infty(\Omega)}\leq C\varepsilon, \\
\|v^\varepsilon\|_{L^\infty(\Omega)}\leq C\varepsilon,
\eeno
where $(u_e(r),0)$ is the Couette flow in (\ref{Taylor-Couette flow}), and $u_p^{(0)}$ is the solution of steady Prandtl equations in  (\ref{prandtl problem near 1}).

Moreover, for any $r<1$, there holds
\begin{align}
\lim_{\varepsilon\rightarrow 0}\|w^\varepsilon-2a\|_{L^\infty(B_r(0))}=0, \nonumber
\end{align}
where $w^\varepsilon(\theta,r)$ is the vorticity of $(u^\varepsilon(\theta,r), v^\varepsilon(\theta,r))$.
\end{Theorem}

\begin{Remark}
The condition $\int_0^{2\pi}f(\theta)d\theta=0$ can be dropped due to the fact
$$\alpha+\eta f(\theta)=\alpha+ \frac{\eta}{2\pi}\int_0^{2\pi}f(\theta)d\theta +\eta \tilde{f}(\theta) $$
where $\int_0^{2\pi} \tilde{f}(\theta) d\theta=0.$ Moreover, the smoothness of $f(\theta)$ can be relaxed, but we don't pursue this issue here.
\end{Remark}

Now we present a sketch of the proof and some key ideas.

{\bf Step 1: Construction of the approximate solution.} We construct an approximate solution $(u^a,v^a)$ by matched asymptotic expansion. The approximate solution consists of Euler part $(u_e^a,v^a_e)$ and Prandtl part $(u_p^a,v^a_p)$, and satisfies the following estimates
\begin{align*}
|\partial_\theta u_e^a(\theta,r)+v_e^a(\theta,r)|\leq C\varepsilon \eta r, \ \ |\partial_\theta v_e^a(\theta,r)-u_e^a(\theta,r)+ar|\leq C\varepsilon \eta r
\end{align*}
which will be used frequently in the error estimate.
The details of constructing the approximate solution will be given in Section 2. After the construction of approximate solution, we derive the  equations (\ref{e:error equation}) for the error $(u,v):=(u^\varepsilon-u^a,v^\varepsilon-v^a)$, then establish the well-posedness of (\ref{e:error equation}).
Notice that the nonlinear term can be easily handled by higher order approximation, hence we only need to consider the linearized error equations:
 \begin{align}\label{linearized error equation}
\left\{
\begin{array}{lll}
u^au_\theta+v^aru_r+uu^a_\theta+vru^a_r+v^au+vu^a+p_\theta-\varepsilon^2\big( ru_{rr}
+\frac{u_{\theta\theta}}{r}+2\frac{v_{\theta}}{r}+u_r-\frac{u}{r}\big)=F_u,\\[5pt]
u^av_\theta+v^arv_r+uv^a_\theta+vrv^a_r-2u u^a+rp_r-\varepsilon^2 \big( rv_{rr}+\frac{v_{\theta\theta}}{r}-2\frac{u_{\theta}}{r}+v_r-\frac{v}{r}\big)=F_v,\\[5pt]
u_\theta+(rv)_r=0,  \\[5pt]
u(\theta+2\pi,r)=u(\theta,r), \ v(\theta+2\pi,r)=v(\theta,r), \\[5pt]
u(\theta,1)=0,\ v(\theta,1)=0.
\end{array}
\right.
\end{align}

{\bf Step 2: Linear stability estimate for (\ref{linearized error equation}).}
Equations (\ref{linearized error equation}) are the linearized Navier-Stokes equations around the approximate solution $(u^a,v^a)$. The leading order of $(u^a,v^a)$ is $\Big(ar+\chi(r)u_p^{(0)}\big(\theta,\frac{r-1}{\varepsilon}\big),0\Big)$, where $\chi(r)$ is a cut-off function, see (\ref{cut-off function}). Since $|u^a_\theta|=|u^{(0)}_{p\theta}|\lesssim \eta$, the leading order of the system (\ref{linearized error equation}) can be simplified as
\begin{align}\label{simplied equation}
\left\{
\begin{array}{lll}
u^au_\theta+vru^a_r+vu^a+p_\theta-\varepsilon^2\big( ru_{rr}
+\frac{u_{\theta\theta}}{r}+2\frac{v_{\theta}}{r}+u_r-\frac{u}{r}\big)=F_u,\\[5pt]
u^av_\theta-2u u^a+rp_r-\varepsilon^2 \big( rv_{rr}+\frac{v_{\theta\theta}}{r}-2\frac{u_{\theta}}{r}+v_r-\frac{v}{r}\big)=F_v,\\[5pt]
u_\theta+(rv)_r=0,  \\[5pt]
u(\theta+2\pi,r)=u(\theta,r), \ v(\theta+2\pi,r)=v(\theta,r), \\[5pt]
u(\theta,1)=0,\ v(\theta,1)=0,
 \end{array}
 \right.
\end{align}
where $u^a$ can be regarded as $ar+\chi(r)u_p^{(0)}\big(\theta,\frac{r-1}{\varepsilon}\big)$.

The linear stability estimate consists of a basic energy estimate and a positivity estimate. In fact, it's easy to know that the basic energy estimate is not good enough for obtaining aprior estimate of (\ref{simplied equation}) because $\e$ is small. The key point is the following observation which gives the important positivity estimate: $u^a$ is strictly positive, we should make use of the terms $u^au_\theta$ and $u^av_\theta$ to obtain a positive quantity from the convective term. To do this, we choose $(u_\theta,v_\theta\big)$ as the multiplier, the pressure is eliminated due to the divergence-free condition and the diffusion terms also vanish. It is easy to obtain
{\small \begin{align*}
&\int_{0}^{1}\int_0^{2\pi}\big(u^au_\theta+vru^a_r+vu^a\big)u_\theta d\theta dr
+\int_{0}^{1}\int_0^{2\pi}\big(u^av_\theta-2u u^a\big)v_\theta d\theta dr\\
=&\underbrace{\int_{0}^{1}\int_0^{2\pi}\big(aru_\theta+2arv\big)u_\theta d\theta dr
+\int_{0}^{1}\int_0^{2\pi}\big(arv_\theta-2aru\big)v_\theta d\theta dr}_{I}\\
&+\underbrace{\int_{0}^{1}\int_0^{2\pi}\big(\chi(r)u_p^{(0)}u_\theta+vr\partial_r(\chi(r)u_p^{(0)})+v\chi(r)u_p^{(0)}\big)u_\theta d\theta dr
+\int_{0}^{1}\int_0^{2\pi}\big(\chi(r)u_p^{(0)}v_\theta-2u \chi(r)u_p^{(0)}\big)v_\theta d\theta dr}_{II}.
\end{align*}}
It is easy to get
\begin{align*}
I_1=a\int_{0}^{1}\int_0^{2\pi}r(u^2_\theta+v^2_\theta)d\theta dr.
\end{align*}
Moreover, since $\chi(r)=0$ for $r<\frac12$, by the Hardy inequality, we deduce that
\begin{align*}
&\Big|\int_{0}^{1}\int_0^{2\pi}vr\partial_r(\chi(r)u_p^{(0)})u_\theta d\theta dr\Big|\\
\leq&\Big|\int_{0}^{1}\int_0^{2\pi}vr\chi'(r)u_p^{(0)}u_\theta d\theta dr\Big|+\Big|\int_{0}^{1}\int_0^{2\pi}\frac{vr}{r-1}\chi(r)Y\partial_Yu_p^{(0)}u_\theta d\theta dr\Big|\\
\leq& C\eta \int_{0}^{1}\int_0^{2\pi}r(u^2_\theta+v^2_\theta)d\theta dr
\end{align*}
and
\begin{align*}
&\Big|\int_{0}^{1}\int_0^{2\pi}u\chi(r)u_p^{(0)}v_\theta d\theta dr\Big|\\
\leq&\varepsilon\Big|\int_{0}^{1}\int_0^{2\pi}\frac{u}{r-1}\chi(r)Yu_p^{(0)}v_\theta d\theta dr\Big|
\leq C\varepsilon\eta \Big(\int_{0}^{1}\int_0^{2\pi}ru^2_rd\theta dr\Big)^{\frac12}\Big(\int_{0}^{1}\int_0^{2\pi}rv^2_\theta d\theta dr\Big)^{\frac12}.
\end{align*}
The other terms in $II$ can be handled by the same argument. Thus, we obtain
\begin{align*}
&\int_{0}^{1}\int_0^{2\pi}\big(u^au_\theta+vru^a_r+vu^a\big)u_\theta d\theta dr
+\int_{0}^{1}\int_0^{2\pi}\big(u^av_\theta-2u u^a\big)v_\theta d\theta dr\\
\geq& (a-C\eta) \int_{0}^{1}\int_0^{2\pi}r(u^2_\theta+v^2_\theta)d\theta dr- C\varepsilon^2\eta\int_{0}^{1}\int_0^{2\pi}ru^2_rd\theta dr.
\end{align*}
When $\eta$ is sufficiently small, $(a-C\eta) \int_{0}^{1}\int_0^{2\pi}r(u^2_\theta+v^2_\theta)d\theta dr$ is a positive quantity. The details of the positivity estimate can be found in Lemma \ref{positivity estimate}.

 However, the positive quantity $\|\sqrt{r}(u_\theta,v_\theta)\|_2$ don't contain zero-frequency of $(u,v)$ which is needed to obtain the $L^\infty$ estimate for the error $(u,v)$. Notice that $\int_0^{2\pi}v(\theta,r) d\theta=0$ because of the divergence-free condition and the boundary condition, we can dominate $\|v\|_2$ by $\|v_\theta\|_2$ using the Poincar\'{e} inequality. However $u_0(r):=\frac{1}{2\pi}\int_0^{2\pi}u(\theta,r) d\theta\neq 0$, we need to obtain the estimate of $|u_0(r)|$ from the basic energy estimate.

 Now we give a sketch of the basic energy estimate.  Choose $u_0$ as a multiplier to the first equation in (\ref{simplied equation}). The diffusion term is
\begin{align*}
&\int_{0}^{1}\int_0^{2\pi}-\e^2\big( ru_{rr}
+\frac{u_{\theta\theta}}{r}+2\frac{v_{\theta}}{r}+u_r-\frac{u}{r}\big)u_0(r)d\theta dr=\e^2\int_{0}^{1}\int_0^{2\pi}\big(r|u'_0|^2+\frac{u_0^2}{r}\big)d\theta dr.
\end{align*}
Recall that $u^a= ar+\chi(r)u^{(0)}_p(\theta,\frac{r-1}{\e})$.  It's direct to deduce that the following quantity vanishes:
\begin{align*}
&\int_{0}^{1}\int_0^{2\pi}\big( aru_\theta+2arv+p_\theta\big)u_0d\theta dr=0,
\end{align*}
where we used $\int_0^{2\pi}v(\theta,r) d\theta=0$. Moreover, the Prandtl part can be handled by the Hardy inequality as above
\begin{align*}
&\int_{0}^{1}\int_0^{2\pi}\big( u_p^{(0)}u_\theta+vr\p_r u_{p}^{(0)}+vu_p^{(0)}\big)u_0d\theta dr\\
\approx &\int_{0}^{1}\int_0^{2\pi}\e\chi(r)\Big( Yu_p^{(0)}u_\theta+\frac{vr}{r-1} Y^2\p_Y u_{p}^{(0)}+v Yu_p^{(0)}\Big)\frac{u_0}{r-1}d\theta dr
\lesssim\eta\e\|\sqrt{r}(u_\theta,v_\theta)\|_2\|\sqrt{r}u'_0\|_2.
\end{align*}
Thus, we obtain that $\varepsilon^2\|\sqrt{r}u'_0\|_2\lesssim \|\sqrt{r}(u_\theta,v_\theta)\|_2$. One can see the details of the
basic energy estimate in Lemma \ref{basic Energy estimate}.

 Combining the positivity estimate and basic energy estimate, we obtain the linear stability of equations (\ref{simplied equation}).

{\bf Step 3:  $H^2$ estimate for error.} To close the nonlinearity, we need the $L^\infty$ estimate. However, the quantity $\|\sqrt{r}(u_\theta,v_\theta)\|_2$ can't be used to control $\|(u,v)\|_0$ due to the weight $r$, and we can only use $\int_{0}^{1}\int_0^{2\pi}r(u^2_r+v_r^2)d\theta dr$ which comes from  the diffusion term.  To do this, we first rewrite the error equations and the associated linear stability estimate in Euler coordinates, then get the $H^2$ estimate by Stokes estimates in a smooth bounded domain, finally $\|(u,v)\|_0$ can be obtained by Sobolev embedding.

Finally, combining the linear stability estimate,  $H^2$ estimate and Sobolev embedding, we can establish the well-posedness of error equations by contraction mapping theorem.

The paper is organized as follows. In Section 2, we construct an approximate solution by matched asymptotic expansion method and study the property of this approximate solution. In Section 3, we derive the error equations and establish their linear stability estimate. The linear stability estimate consists of the basic energy estimate and the positivity estimate. In Section 4, we firstly rewrite the error equations and the associated linear stability estimate in Euler coordinates, then obtain $H^2$ estimate by the Stokes estimate in a smooth bounded domain. In Section 5, we complete the proof of Theorem \ref{main theorem} by combining the linear stability estimate and  $H^2$ estimate.

\smallskip


\section{Construction of approximate solutions}
\indent

In this section, we construct an approximate solution of the Navier-Stokes equations (\ref{NS-curvilnear}) by matched asymptotic expansion.

\subsection{Euler expansions}

\indent

Away from the boundary, we make the following formal expansions
\begin{align*}
&u^{\varepsilon}(\theta,r)=u_e^{(0)}(\theta,r)+\varepsilon u_e^{(1)}(\theta,r)+\cdots,\\[5pt]
&v^{\varepsilon}(\theta,r)=v_e^{(0)}(\theta,r)+\varepsilon v_e^{(1)}(\theta,r)+\cdots,\\[5pt]
&p^{\varepsilon}(\theta,r)=p_e^{(0)}(\theta,r)+\varepsilon p_e^{(1)}(\theta,r)+\cdots.
\end{align*}

\subsubsection{Equations for $(u_e^{(0)},v_e^{(0)},p_e^{(0)})$}

\indent

By substituting the above expansions into (\ref{NS-curvilnear}) and collecting the $\varepsilon$-zeroth order terms, we deduce that $(u_e^{(0)},v_e^{(0)},p_e^{(0)})$ satisfies the following steady nonlinear Euler equations
\begin{eqnarray}
\left \{
\begin {array}{ll}
u_e^{(0)} \partial_\theta u_e^{(0)}+rv_e^{(0)} \partial_ru_e^{(0)}+u_e^{(0)} v_e^{(0)}+\partial_\theta p_e^{(0)}=0,\\[7pt]
 u_e^{(0)} \partial_\theta v_e^{(0)}+rv_e^{(0)} \partial_rv_e^{(0)}-(u_e^{(0)})^2+r\partial_rp_e^{(0)}=0,\\[7pt]
 \partial_\theta u_e^{(0)}+r\partial_rv_e^{(0)}+v_e^{(0)}=0.\label{outer-leading order equation}
\end{array}
\right.
\end{eqnarray}

Due to the Prandtl-Bathelor theory in \cite{B}, the leading order Euler flows $(u_e^{(0)},v_e^{(0)})$ with closed streamlines must be the Couette flows in $\Omega$
\begin{align}
u_e^{(0)}(\theta,r)=u_e(r)=:ar,\ \ v_e^{(0)}(\theta,r)=0,\label{outer-leading order velocity}
\end{align}
where the value of $a$ is a constant determined by the Wood formula (\ref{BW formula}).

Then equations (\ref{outer-leading order equation}) reduce to
 \begin{align}
\partial_\theta p_e^{(0)}(\theta,r)=0,\quad  \partial_rp_e^{(0)}(\theta,r)=\frac{1}{r}u_e^2.\nonumber
 \end{align}
 Thus, we deduce that
 \begin{align}
 p_e^{(0)}(\theta,r)=p_e(r),\quad  p'_e(r)=a^2r.\label{outer-leading order pressure}
 \end{align}

\subsubsection{Equations for $(u_e^{(1)},v_e^{(1)},p_e^{(1)})$}

\indent

By collecting the $\varepsilon$-order terms, we deduce that $(u_e^{(1)},v_e^{(1)},p_e^{(1)})$ satisfies the following linearized Euler equations in $\Omega$
\begin{eqnarray}
\left \{
\begin {array}{ll}
ar \partial_\theta u_e^{(1)}+2arv_e^{(1)} +\partial_\theta p_e^{(1)}=0,\\[5pt]
ar \partial_\theta v_e^{(1)}-2aru_e^{(1)}+r\partial_rp_e^{(1)}=0,\\[5pt]
\partial_\theta u_e^{(1)}+r\partial_rv_e^{(1)}+ v_e^{(1)}=0,\label{outer-1 order equation}
\end{array}
\right.
\end{eqnarray}
which are equipped with the boundary condition
\begin{align}
 v_e^{(1)}|_{r=1}=-v_p^{(1)}|_{Y=0},\quad v_e^{(1)}(\theta,r)=v_e^{(1)}(\theta+2\pi,r) ,\label{outer-1 order-bc}
\end{align}
where $v_p^{(1)}$ is the solution of Prandtl equations which will be derived in next subsection.

The equations for $(u_e^{(i)}, v_e^{(i)}, p_e^{(i)}), i=2,3,4$ will be derived later.

\subsection{Prandtl expansion near $r=1$}
\indent

We  introduce the scaled variable  $Y=\frac{r-1}{\varepsilon}\in (-\infty,0]$ and make the following Prandtl expansions near $r=1$
\begin{align}\label{first order expansion}
\begin{aligned}
&u^\varepsilon=u_e(r)+u_p^{(0)}(\theta,Y)+\varepsilon\big[u_e^{(1)}(\theta,r)+u_p^{(1)}(\theta,Y)\big]
+\cdots,\\[5pt]
&v^\varepsilon=v_p^{(0)}(\theta,Y)+\varepsilon\big[v_e^{(1)}(\theta,r)+v_p^{(1)}(\theta,Y)\big]
+\varepsilon^2\big[v_e^{(2)}(\theta,r)+v_p^{(2)}(\theta,Y)\big]+\cdots,\\[5pt]
&p^\varepsilon=p_e(r)+p_p^{(0)}(\theta,Y)+\varepsilon\big[p_e^{(1)}(\theta,r)+p_p^{(1)}(\theta,Y)\big]
+\varepsilon\big[p_e^{(2)}(\theta,r)+p_p^{(2)}(\theta,Y)\big]+\cdots,
\end{aligned}
\end{align}
where as $Y\rightarrow -\infty$
\begin{align}
  \partial_\theta^l\partial_Y^mv_p^{(i)}(\theta,Y)\rightarrow 0, \ \partial_\theta^l\partial_Y^mp_p^{(i)}(\theta,Y)\rightarrow 0,\label{matching condition}
\end{align}
here $l,m\geq 0, i=0,1,\cdots,$
which satisfies the following boundary condition
\begin{align*}
&u_e^{(0)}(\theta,1)+u_p^{(0)}(\theta,0)=\alpha+\eta f(\theta), \quad u_e^{(i)}(\theta,1)+u_p^{(i)}(\theta,0)=0,\ i\geq 1,\\
&v_e^{(i)}(\theta,1)+v_p^{(i)}(\theta,0)=0, i\geq 0.
\end{align*}
The boundary conditions of $u_p^{(i)}(\theta,Y)$ as $Y\rightarrow-\infty $ will be given later.

\subsubsection{Equations for $(v_p^{(0)}, p_p^{(0)})$}

\indent

By substituting the above expansions into (\ref{NS-curvilnear}) and collecting the $\frac{1}{\varepsilon}$ order terms, we get
\begin{align*}
\partial_Yv_p^{(0)}(\theta,Y)=0,\ \ \partial_Yp_p^{(0)}(\theta,Y)=0,
\end{align*}
which together with (\ref{matching condition}) imply
\begin{align*}
v_p^{(0)}=0,\ \ p_p^{(0)}=0.
\end{align*}

\subsubsection{Equations for $(u_p^{(0)},v_p^{(1)},p_p^{(1)})$}

\indent

By substituting the above expansions into (\ref{NS-curvilnear}) and collecting the $\varepsilon$-zeroth order terms, we obtain
 the following steady Prandtl equations for $(u_p^{(0)},v_p^{(1)})$
\begin{eqnarray}
\left \{
\begin {array}{ll}
\big(u_e(1)+u_p^{(0)}\big)\partial_\theta u_p^{(0)}+\big(v_e^{(1)}(\theta,1)+ v_p^{(1)}\big)\partial_Yu_p^{(0)}-\partial_{YY}u_p^{(0)}=0,\\[5pt]
\partial_\theta u_p^{(0)}+\partial_Yv_p^{(1)}=0,\\[5pt]
u_p^{(0)}(\theta,Y)=u_p^{(0)}(\theta+2\pi,Y),\quad v_p^{(1)}(\theta,Y)=v_p^{(1)}(\theta+2\pi,Y),\\[5pt]
u_p^{(0)}\big|_{Y=0}=\alpha+\eta f(\theta)-u_e(1),\ \ \lim_{Y\rightarrow -\infty}(u_p^{(0)},v_p^{(1)})=(0,0)
\label{prandtl problem 1}
\end{array}
\right.
\end{eqnarray}
and the pressure $p_p^{(1)}$ satisfies
\begin{align}\label{equation of first pressure}
\partial_Yp_p^{(1)}(\theta, Y)=(u_p^{(0)})^2(\theta,Y)+2u_e(1)u_p^{(0)}(\theta,Y), \ \lim_{Y\rightarrow -\infty}p_p^{(1)}(\theta,Y)=0.
\end{align}

\subsubsection{Equations for $(u_p^{(1)},v_p^{(2)})$}

\indent

By substituting the above expansions into the first and third equation in (\ref{NS-curvilnear}) and collecting the $\varepsilon$-order terms, we obtain
 the following linearized steady Prandtl equations for $(u_p^{(1)},v_p^{(2)})$
\begin{eqnarray}
\left \{
\begin {array}{ll}
\big(u_e(1)+u_p^{(0)}\big)\partial_\theta u_p^{(1)}+\big(v_e^{(1)}(\theta,1)+ v_p^{(1)}\big)\partial_Yu_p^{(1)}+(v_e^{(2)}(\theta,1)+v_p^{(2)})\partial_{Y}u_p^{(0)}\\[5pt]
\quad \quad \quad \quad +(u_p^{(1)}+u_e^{(1)}(\theta,1))\partial_\theta u_p^{(0)}-\partial_{YY}u_p^{(1)}=f_1(\theta,Y)\\[5pt]
\partial_\theta u_p^{(1)}+\partial_Yv_p^{(2)}+\partial_Y(Yv_p^{(1)})=0,\\[5pt]
u_p^{(1)}(\theta,Y)=u_p^{(1)}(\theta+2\pi,Y),\quad v_p^{(2)}(\theta,Y)=v_p^{(2)}(\theta+2\pi,Y),\\[5pt]
u_p^{(1)}\big|_{Y=0}=-u_e^{(1)}\big|_{r=1},\ \ \lim_{Y\rightarrow -\infty}(\partial_Yu_p^{(1)},v_p^{(2)})=(0,0),
\label{first linearized prandtl problem near 1}
\end{array}
\right.
\end{eqnarray}
where
\begin{align*}
f_1(\theta,Y)=&-\partial_\theta p_p^{(1)}+Y\partial_{YY}u_p^{(0)}+\partial_Yu_p^{(0)}\\[5pt]
&-u_p^{(0)}\big(\partial_\theta u_e^{(1)}(\theta,1)+v_e^{(1)}(\theta,1)+v_p^{(1)}\big)-u'_e(1)Y\partial_\theta u_p^{(0)}\\[5pt]
&-(\partial_rv_e^{(1)}(\theta,1)+v_e^{(1)}(\theta,1))Y\partial_Y u_p^{(0)}-(u_e'(1)+Y\partial_Yu_p^{(0)}+u_e(1))v_p^{(1)}.
\end{align*}
The equations for $(u_p^{(i)}, v_p^{(i+1)}), i=2,3,4$ will be derived later.

\subsection{Solvabilities of Euler equations and Prandtl equations}
\indent

The order in which we solve the equations is as follows
\begin{align*}
(u_e(r),0)\rightarrow (u_p^{(0)},v_p^{(1)})\rightarrow (u_e^{(1)},v_e^{(1)})\rightarrow(u_p^{(1)},v_p^{(2)})\rightarrow\cdot\cdot\cdot.
\end{align*}

\subsubsection{Prandtl equations and their solvabilities}
\indent

 We first derive some necessary condition for the solvability of Prandtl equations.
\begin{lemma}\label{BW}(Batchelor-Wood formula \cite{K1998}\cite{K2000}) Let $\int_0^{2\pi}f(\theta)d\theta=0$. If the nonlinear Prandtl equations (\ref{prandtl problem 1}) have a solution $(u_p^{(0)}, v_p^{(1)})$ which satisfies
\beno
u_e(1)+u_p^{(0)}(\theta,Y)>0, \quad \forall Y\leq 0, \quad
\|u_p^{(0)}\|_0\leq M,
\eeno
where $M>0$ is a constant,  then there holds
\begin{align}
u_{e}^2(1)=\alpha^2+\frac{\eta^2}{2\pi}\int_0^{2\pi}f^2(\theta)d\theta.\label{BW formula}
\end{align}
\end{lemma}
\begin{proof}
We introduce the von Mises variable
\begin{align*}
\psi=\int_0^Y\big(u_e(1)+u_p^{(0)}(\theta,z)\big)dz,\ \ \mathcal{U}^{(0)}(\theta,\psi)=u_e(1)+u_p^{(0)}(\theta,Y).
\end{align*}
Then from (\ref{prandtl problem 1}), we deduce that $\mathcal{U}^{(0)}$ satisfies
\begin{eqnarray}
\left \{
\begin {array}{ll}
2\mathcal{U}^{(0)}_\theta=\big((\mathcal{U}^{(0)})^2\big)_{\psi\psi},\\[5pt]
\mathcal{U}^{(0)}(\theta,\psi)=\mathcal{U}^{(0)}(\theta+2\pi,\psi),\\[5pt]
\mathcal{U}^{(0)}\big|_{\psi=0}=\alpha+\eta f(\theta),\ \ \lim_{\psi\rightarrow-\infty}\mathcal{U}^{(0)}=u_{e}(1).\label{modified prandtl equation}
\end{array}
\right.
\end{eqnarray}
Here we have used the facts:
\begin{align*}
&\partial_\theta u_p^{(0)}=\mathcal{U}^{(0)}_\theta+\mathcal{U}^{(0)}_\psi\int_0^Y\partial_\theta u_p^{(0)}(\theta,z)\big)dz
\\[5pt]&\qquad=\mathcal{U}^{(0)}_\theta+\mathcal{U}^{(0)}_\psi\big(v_p^{(1)}(\theta,0)-v_p^{(1)}(\theta,Y)\big)
=\mathcal{U}^{(0)}_\theta-\mathcal{U}^{(0)}_\psi\big(v^{(1)}(\theta,1)+v_p^{(1)}(\theta,Y)\big),
\\[5pt]&\partial_Yu_p^{(0)}=\mathcal{U}^{(0)}_\psi\mathcal{U}^{(0)}.
\end{align*}
Integrating the first equation in (\ref{modified prandtl equation}) from $0$ to $2\pi$ about $\theta$ leads to
\begin{align*}
\frac{\partial^2}{\partial\psi^2}\int_0^{2\pi}(\mathcal{U}^{(0)})^2(\theta,\psi)d\theta=0.
\end{align*}
Notice that $\mathcal{U}^{(0)}$ is bounded at $\psi\rightarrow-\infty$, we deduce that
\begin{align*}
\frac{\partial}{\partial\psi}\int_0^{2\pi}(\mathcal{U}^{(0)})^2(\theta,\psi)d\theta=0.
\end{align*}
Therefore combining the boundary condition in (\ref{modified prandtl equation}), we deduce that
\begin{align*}
u_{e}^2(1)=\frac{1}{2\pi}\int_0^{2\pi}\big(\alpha+\eta f(\theta)\big)^2(\theta)d\theta&=\alpha^2+\frac{\alpha\eta}{\pi}\int_0^{2\pi}f(\theta)d\theta+\frac{\eta^2}{2\pi}\int_0^{2\pi}f^2(\theta)d\theta
\\&=\alpha^2+\frac{\eta^2}{2\pi}\int_0^{2\pi}f^2(\theta)d\theta.
\end{align*}
Thus, we complete the proof of this lemma.
\end{proof}

The following  Corollary is a direct result of  Lemma \ref{BW}.
\begin{Corollary}\label{circular flow  formula} If Lemma \ref{BW} holds,  then
\begin{align*}
a=\Big(\alpha^2+\frac{\eta^2}{2\pi}\int_0^{2\pi}f^2(\theta)d\theta\Big)^{\frac12}.
\end{align*}
\end{Corollary}

Next we aim to solve the steady Prandtl equations (\ref{prandtl problem 1}) and we use the method of \cite{K2000}.
\begin{proposition}\label{decay estimates} There exists $\eta_0>0$ such that for any $\eta\in (0,\eta_0)$, the equation (\ref{modified prandtl equation}) has a unique solution $\mathcal{U}^{(0)}$ which satisfies the following estimates
\begin{align*}
\sum_{j+k\leq m}\int_{-\infty}^0\int_0^{2\pi}\Big|\partial_\theta^j\partial_\psi^k \big(\mathcal{U}^{(0)}-u_e(1)\big)\Big|^2\big<\psi\big>^{2l}d\theta d\psi\leq C(m,l)\eta^2,\ m,l\geq 0
\end{align*}
here $\big<\psi\big>=\sqrt{1+\psi^2}$.

\end{proposition}
\begin{proof}

We use the contraction mapping theorem  to  prove the desired conclusions and divide the proof into five steps.

{\bf Step 1:} {\it Derivation of equivalent equation.}

Let $Q(\theta,\psi):=(\mathcal{U}^{(0)})^2(\theta,\psi)-u_{e}^2(1)$ and we rewrite (\ref{modified prandtl equation}) as
\begin{eqnarray}
\left \{
\begin {array}{ll}
Q_\theta=\mathcal{U}^{(0)}Q_{\psi\psi},\\[7pt]
Q(\theta,\psi)=Q(\theta+2\pi,\psi),\\[5pt]
Q\big|_{\psi=0}=\alpha^2+2\alpha\eta f(\theta)+\eta^2 f^2(\theta)-u_{e}^2(1),\ \ Q\big|_{\psi\rightarrow-\infty}=0.\label{modified prandtl equation-1}
\end{array}
\right.
\end{eqnarray}

Defining
\begin{align*}
\mathcal{G}(Q)=Q-2u_{e}(1)\sqrt{Q+u_e^2(1)}+2u_{e}^2(1),
\end{align*}
then (\ref{modified prandtl equation-1}) is equivalent to
\begin{eqnarray}
\left \{
\begin {array}{ll}
Q_\theta-u_{e}(1)Q_{\psi\psi}=(\mathcal{G}(Q))_\theta,\\[7pt]
Q(\theta,\psi)=Q(\theta+2\pi,\psi),\\[5pt]
Q\big|_{\psi=0}=\alpha^2+2\alpha\eta f(\theta)+\eta^2 f^2(\theta)-u_{e}^2(1),\ \ Q\big|_{\psi\rightarrow-\infty}=0.\label{modified prandtl equation-2}
\end{array}
\right.
\end{eqnarray}
Let $Q_0$ be the solution to
\begin{eqnarray}
\left \{
\begin {array}{ll}
(Q_0)_\theta=u_{e}(1)(Q_0)_{\psi\psi},\\[7pt]
Q_0(\theta,\psi)=Q_0(\theta+2\pi,\psi),\\[5pt]
Q_0\big|_{\psi=0}=\alpha^2+2\alpha\eta f(\theta)+\eta^2 f^2(\theta)-u_{e}^2(1),\ \ Q_0\big|_{\psi\rightarrow-\infty}=0,\label{Q0}
\end{array}
\right.
\end{eqnarray}
which will be solved in Appendix A,  then (\ref{modified prandtl equation-2}) is equivalent to
\begin{eqnarray}
\left \{
\begin {array}{ll}
Q_\theta-u_{e}(1)Q_{\psi\psi}=(\mathcal{H}(Q))_\theta,\\[7pt]
Q(\theta,\psi)=Q(\theta+2\pi,\psi),\\[5pt]
Q\big|_{\psi=0}=0,\ \ Q\big|_{\psi\rightarrow-\infty}=0,\label{modified prandtl equation-3}
\end{array}
\right.
\end{eqnarray}
where
\begin{align*}
\mathcal{H}(Q)=Q+Q_0-2u_{e}(1)\sqrt{Q+Q_0+u_e^2(1)}+2u_{e}^2(1).
\end{align*}

Defining the linear operator $\mathcal{L}:\Lambda\longmapsto \Phi$ such that
\begin{eqnarray}
\Phi=\mathcal{L} \Lambda\Longleftrightarrow\left \{
\begin {array}{ll}
\Phi_\theta-u_{e}(1)\Phi_{\psi\psi}=\Lambda_\theta,\\[7pt]
\Phi(\theta,\psi)=\Phi(\theta+2\pi,\psi),\\[5pt]
\Phi\big|_{\psi=0}=0,\ \ \Phi\big|_{\psi\rightarrow-\infty}=0,\label{definition L}
\end{array}
\right.
\end{eqnarray}
then (\ref{modified prandtl equation-3}) is equivalent to
\begin{eqnarray}
\left \{
\begin {array}{ll}
Q=\mathcal{L} (\mathcal{H}(Q)),\\[7pt]
Q(\theta,\psi)=Q(\theta+2\pi,\psi),\\[5pt]
Q\big|_{\psi=0}=0,\ \ Q\big|_{\psi=-\infty}=0.\label{operator form}
\end{array}
\right.
\end{eqnarray}

Defining the function space $X$ as follows
\begin{align*}
X=\bigg\{&Q:Q(\theta,\psi)=Q(\theta+2\pi,\psi),\ Q\big|_{\psi=0}=Q\big|_{\psi\rightarrow-\infty}=0,\\&
\quad \quad  \|Q\|_X^2:=\sum\limits_{j+k\leq m, l\geq 0}\int_{-\infty}^0\int_0^{2\pi}\big|\partial_\theta^j\partial_\psi^k Q\big|^2\big<\psi\big>^ld\theta d\psi<+\infty.\bigg\}
\end{align*}
and a ball $B_0$ in $X$
\begin{align*}
B_0=\big\{Q\in X: \|Q\|_X\leq r_0\big\},
\end{align*}
here $r_0$ is a small number which will be determined later, $m$ is a positive integer.

 In the next three steps,  we will verify that $\mathcal{L}\circ\mathcal{H}$ is a contraction map from $B_0$ to $B_0$ with a suitable small $r_0$.

{\bf Step 2:}{ \it  Boundedness of $\mathcal{L}$ in $X$}.

 In this step, we prove that for any $m\geq 0, l\geq 0$, there holds
\begin{align}\label{estimate for $L$}
\sum_{j+k\leq m}\big\|\partial^j_\theta\partial^k_\psi\Phi\big<\psi\big>^l\big\|_2\leq C(m,l)\sum_{j+k\leq m, q\leq l}\big\|\partial^j_\theta\partial^k_\psi\Lambda\big<\psi\big>^q\big\|_2.
\end{align}

First, we prove (\ref{estimate for $L$}) for $l=0$. Multiplying   the equation in (\ref{definition L}) by $\Phi$  and integrating with respect to $(\theta,\psi)\in(0,2\pi)\times(-\infty,0)$, we obtain
\begin{align}
u_{e}(1)\|\Phi_\psi\|_2\leq C(\lambda)\|\Lambda\|_2+\lambda\|\Phi_\theta\|_2.\label{iorder-derivative-1}
\end{align}
Then, multiplying   the equation in (\ref{definition L}) by $\Phi_\theta$  and integrating with respect to $(\theta,\psi)\in(0,2\pi)\times(-\infty,0)$, one has
\begin{align}
\|\Phi_\theta\|_2\leq C(\lambda)\|\Lambda_\theta\|_2+\lambda\|\Phi_\theta\|_2.\label{iorder-derivative-2}
\end{align}
Combining (\ref{iorder-derivative-1})-(\ref{iorder-derivative-2}) and choosing small $\lambda>0$ we get
\begin{align}
\|(\Phi_\psi,\Phi_\theta)\|_2\leq C\|(\Lambda,\Lambda_\theta)\|_2.\label{iorder-derivative}
\end{align}
Integrating (\ref{definition L}) with respect to $\theta\in(0,2\pi)$ gives
\begin{align*}
\frac{d^2}{d\psi^2}\int_0^{2\pi}\Phi (\theta,\psi)d\theta=0,
\end{align*}
which combines $\Phi|_{\psi=0}=\Phi|_{\psi\rightarrow-\infty}=0$ imply
\begin{align}
\int_0^{2\pi}\Phi (\theta,\psi)d\theta=0 .\label{zero average}
\end{align}
Due to  (\ref{zero average}) and the Poincar\'{e} inequality we have
\begin{align}
\|\Phi\|_2\leq C\|\Phi_\theta\|_2\leq C\|(\Lambda,\Lambda_\theta)\|_2.\label{L2}
\end{align}
For any $j\geq 0, k\geq 2$, from the equation (\ref{definition L}), we deduce that
\begin{align}
\|\partial_\theta^j\partial_\psi^k\Phi\|_2\leq C (\|\partial_\theta^j\partial_\psi^{k-2}\Phi_\theta\|_2+\|\partial_\theta^j\partial_\psi^{k-2}\Lambda_\theta\|_2).\label{higher-order-derivative}
\end{align}
For any $j\geq 0$, applying $\partial^j_\theta$ to the equation in (\ref{definition L}), multiplying the resultant equation by   $\partial^j_\theta\Phi$, integrating with respect to $(\theta,\psi)\in(0,2\pi)\times(-\infty,0)$ and using the Young inequality  one has
\begin{align}
\|\partial^j_\theta\partial_\psi\Phi\|_2\leq C\|\partial^{j+1}_\theta\Lambda\|_2+C\|\partial^j_\theta\Phi\|_2.\label{2order-derivative-3}
\end{align}
For any $j\geq 0$, multiplying the equation in (\ref{definition L}) by   $\partial^{2j-1}_\theta\Phi$, integrating with respect to $(\theta,\psi)\in(0,2\pi)\times(-\infty,0)$, one can get
\begin{align}
\|\partial^j_\theta\Phi\|^2_2=\int_{-\infty}^0\int_0^{2\pi}\partial^{2j-1}_\theta\Phi\Lambda_\theta d\theta d\psi\leq C\|\partial^j_\theta\Phi\|_2\|\partial^j_\theta\Lambda\|_2.\label{2order-derivative-2}
\end{align}
Combining (\ref{higher-order-derivative}), (\ref{2order-derivative-3}) and (\ref{2order-derivative-2}), we obtain that for any $m\geq 2$, there holds
\begin{align}
\sum_{j+k=m}\|\partial^j_\theta\partial^k_\psi\Phi\|_2\leq C\sum_{j+k\leq m-1}\|\partial^j_\theta\partial^k_\psi\Phi\|_2+C\sum_{j+k\leq m}\|\partial^j_\theta\partial^k_\psi\Lambda\|_2.\nonumber
\end{align}
Thus, by induction and (\ref{iorder-derivative}), (\ref{L2}), we deduce that for any $m\geq 0$, there holds
\begin{align}\label{estimate without weight}
\sum_{j+k\leq m}\|\partial^j_\theta\partial^k_\psi\Phi\|_2\leq C(m)\sum_{j+k\leq m}\|\partial^j_\theta\partial^k_\psi\Lambda\|_2.
\end{align}

Next, we prove (\ref{estimate for $L$}) for $m\leq1, l\geq 1$.
Multiplying the equation in (\ref{definition L}) by $\Phi\psi^{2l}$, integrating with respect to $(\theta,\psi)\in(0,2\pi)\times(-\infty,0)$, we have
\begin{align}
\int_{-\infty}^0\int_0^{2\pi}\Phi_\psi^2\psi^{2l}d\theta d\psi&\leq C\int_{-\infty}^0\int_0^{2\pi}\Phi^2\psi^{2(l-1)}d\theta d\psi
\nonumber\\&\quad+ C(\lambda)\int_{-\infty}^0\int_0^{2\pi}\Lambda^2\psi^{2l}d\theta d\psi+\lambda\int_{-\infty}^0\int_0^{2\pi}\Phi_\theta^2\psi^{2l}d\theta d\psi
.\label{iorder-high-weight-derivative-1}
\end{align}
Multiplying   the equation in (\ref{definition L}) by $\Phi_\theta\psi^{2l}$, integrating with respect to $(\theta,\psi)\in(0,2\pi)\times(-\infty,0)$, we obtain
\begin{align}
\int_{-\infty}^0\int_0^{2\pi}\Phi_\theta^2\psi^{2l}d\theta d\psi&\leq C\int_{-\infty}^0\int_0^{2\pi}\Phi_\psi^{2}\psi^{2(l-1)}d\theta d\psi
+C\int_{-\infty}^0\int_0^{2\pi}\Lambda_\theta^2\psi^{2l}d\theta d\psi.\label{iorder-high-weight-derivative-2}
\end{align}
Combining (\ref{iorder-high-weight-derivative-1})-(\ref{iorder-high-weight-derivative-2}), using the Poincar\'{e} inequality and choosing small $\lambda>0$ we get
\begin{align*}
&\int_{-\infty}^0\int_0^{2\pi}\big(\Phi^2+\Phi_\psi^2+\Phi_\theta^2\big)\psi^{2l}d\theta d\psi\nonumber\\
\leq&\int_{-\infty}^0\int_0^{2\pi}\big(\Phi^2+\Phi_\psi^2+\Phi_\theta^2\big)\psi^{2(l-1)}d\theta d\psi+ C\int_{-\infty}^0\int_0^{2\pi}\big(\Lambda^2\psi^{2l}+\Lambda_\theta^2\psi^{2l}\big)d\theta d\psi.
\end{align*}
By induction on $l$ and using (\ref{estimate without weight}), we deduce that for any $l\geq 1$, there holds
\begin{align}\label{estimate on lower order derivative with weight}
\big\|\big(\Phi_\psi,\Phi_\theta,\Phi\big)\psi^{l}\big\|_2\leq C(l)\sum_{q\leq l}\big\|\big(\Lambda,\Lambda_\theta\big)\psi^q\big\|_2.
\end{align}

Finally, we prove (\ref{estimate for $L$}) any $m\geq 2, l\geq 1$. When $j+k=m, k\geq 2$, from the equation (\ref{definition L}), we deduce that
\begin{align}
\Big\|\partial_\theta^j\partial_\psi^k\Phi \psi^l\Big\|_2\leq C \Big(\Big\|\partial_\theta^j\partial_\psi^{k-2}\partial_\theta\Phi\psi^l\Big\|_2
+\Big\|\partial_\theta^j\partial_\psi^{k-2}\partial_\theta\Lambda\psi^l\Big\|_2\Big).\label{higher-order-derivative-weight}
\end{align}
Applying $\partial^{m-1}_\theta$ to the equation in (\ref{definition L}), multiplying the resultant equation by   $\partial^{m-1}_\theta\Phi \psi^{2l}$, integrating with respect to $(\theta,\psi)\in(0,2\pi)\times(-\infty,0)$ and using the Young inequality  one has
\begin{align}
\Big\|\partial_\theta^{m-1}\partial_\psi\Phi \psi^l\Big\|_2\leq C\Big\|\partial_\theta^m\Lambda \psi^l\Big\|_2+C\Big\|\partial_\theta^{m-1}\Phi \psi^l\Big\|_2+C(l)\Big\|\partial_\theta^{m-1}\Phi \psi^{l-1}\Big\|_2.\label{2order-derivative-3-weight}
\end{align}
Multiplying the equation in (\ref{definition L}) by   $\partial^{2m-1}_\theta\Phi \psi^{2l}$, integrating with respect to $(\theta,\psi)\in(0,2\pi)\times(-\infty,0)$, one can get
\begin{align}
\Big\|\partial^m_\theta\Phi \psi^l\Big\|_2\leq C\Big\|\partial^m_\theta\Lambda  \psi^l \Big\|_2+C(l)\Big\|\partial_\theta^{m-1}\partial_\psi\Phi \psi^{l-1}\Big\|_2.\label{2order-derivative-2-weight}
\end{align}
Combining (\ref{higher-order-derivative-weight}), (\ref{2order-derivative-3-weight}) and (\ref{2order-derivative-2-weight}), we obtain that for any $m\geq 2, l\geq 1$, there holds
\begin{align}
\sum_{j+k=m}\Big\|\partial^j_\theta\partial^k_\psi\Phi \psi^l\Big\|_2\leq C(m,l)\sum_{j+k\leq m-1, q\leq l}\Big\|\partial^j_\theta\partial^k_\psi\Phi \psi^q\Big\|_2+C(m,l)\sum_{j+k\leq m}\Big\|\partial^j_\theta\partial^k_\psi\Lambda \psi^l\Big\|_2.\nonumber
\end{align}
Thus, by induction $m$ and using (\ref{estimate on lower order derivative with weight}), we deduce that for any $m\geq 2, l\geq 1$, there holds
\begin{align}\label{estimate with weight}
\sum_{j+k\leq m}\Big\|\partial^j_\theta\partial^k_\psi\Phi \psi^l\Big\|_2\leq C(m,l)\sum_{j+k\leq m,q\leq l}\Big\|\partial^j_\theta\partial^k_\psi\Lambda \psi^q\Big\|_2.
\end{align}
This complete the proof of (\ref{estimate for $L$}).
Thus, we obtain
\begin{align}
\|\mathcal{L}\Lambda\|_{X}\leq C \|\Lambda\|_{X},\ \ \forall \Lambda\in X.\label{boundedness}
\end{align}

{\bf Step 3:}{ \it $\mathcal{L}\circ\mathcal{H}$ is a continuous map from $B_0$ to $B_0$.} 

In this section, we first prove that for any $m\geq 2, l\geq 0$, there holds
\begin{align}\label{nonlinear estimate on H}
\sum_{j+k\leq m}\Big\|\partial_\theta^j\partial_\psi^k \mathcal{H}(Q)\big<\psi\big>^l\Big\|_2\leq C(m,l)\Big(\sum_{j+k\leq m}\Big\|\partial_\theta^j\partial_\psi^k Q\big<\psi\big>^l\Big\|_2+\sum_{j+k\leq m}\Big\|\partial_\theta^j\partial_\psi^k Q_0\big<\psi\big>^l\Big\|_2\Big)^2.
\end{align}

Set
\beno
\tilde{H}(x)=x-2u_e(1)\sqrt{x+u_e^2(1)}+2u^2_e(1), |x|\ll u_e^2(1),
\eeno
then $\mathcal{H}(Q)=\tilde{H}(Q+Q_0).$ Direct computation gives
\begin{align}\label{estimate on H}
|\tilde{H}'(x)|\leq C|x|, \quad |\tilde{H}^{(k)}(x)|\leq C,\quad k\geq 2.
\end{align}
Since
\begin{align*}
\mathcal{H}(Q)=\big(\sqrt{Q+Q_0+u_e^2(1)}-u_e(1)\big)^2=\bigg(\frac{Q+Q_0}{\sqrt{Q+Q_0+u_e^2(1)}+u_e(1)}\bigg)^2,
\end{align*}
it's easy to get
\begin{align*}
\mathcal{H}^2(Q)\big<\psi\big>^l\leq C\|Q+Q_0\|_{L^\infty}^2\big(Q^2\big<\psi\big>^l
+Q_0^2\big<\psi\big>^l\big), \quad l\geq 0.
\end{align*}
Using (\ref{estimate on H}), we deduce that
\begin{align*}
\big(\partial_\psi\mathcal{H}(Q)\big)^2\big<\psi\big>^l\leq& C\|Q+Q_0\|_{L^\infty}^2
\big(Q_\psi^2\big<\psi\big>^l+(Q_0)_\psi^2\big<\psi\big>^l\big),\quad l\geq 0,\\
\big(\partial_\theta\mathcal{H}(Q)\big)^2\big<\psi\big>^l\leq& C\|Q+Q_0\|_{L^\infty}^2
\big(Q_\theta^2\big<\psi\big>^l+(Q_0)_\theta^2\big<\psi\big>^l\big),\quad l\geq 0.
\end{align*}
Thus, we obtain
\begin{align}\label{lower order estimate on H}
&\sum_{j+k\leq 1}\Big\|\partial_\theta^j\partial_\psi^k \mathcal{H}(Q)\big<\psi\big>^l\Big\|_2\nonumber\\
\leq& C\|Q+Q_0\|_{L^\infty}\Big(\sum_{j+k\leq 1}\Big\|\partial_\theta^j\partial_\psi^k Q\big<\psi\big>^l\Big\|_2+\sum_{j+k\leq 1}\Big\|\partial_\theta^j\partial_\psi^k Q_0\big<\psi\big>^l\Big\|_2\Big).
\end{align}
Since
\begin{align*}
\partial_{\psi\theta}\mathcal{H}(Q)=\tilde{H}'(Q+Q_0)(\partial_{\psi\theta}Q+\partial_{\psi\theta}Q_0)
+\tilde{H}''(Q+Q_0)(\partial_{\theta}Q+\partial_{\theta}Q_0)(\partial_{\psi}Q+\partial_{\psi}Q_0),
\end{align*}
thus, using (\ref{estimate on H}), we obtain
\begin{align*}
\Big\|\partial_{\psi\theta}\mathcal{H}(Q)\big<\psi\big>^l\Big\|_{L^2}&\leq C\|Q+Q_0\|_{L^\infty}\Big(\Big\|Q_{\psi\theta}\big<\psi\big>^l\Big\|_{L^2}+\Big\|(Q_0)_{\psi\theta}\big<\psi\big>^l\Big\|_{L^2}\Big)
\nonumber\\&\quad+C\big\|Q_{\psi}
+(Q_0)_{\psi}\big\|_{L^4}\big\|(Q_{\theta}+(Q_0)_{\theta})\big<\psi\big>^l\big\|_{L^4}
\nonumber\\&\leq C\|Q+Q_0\|_{L^\infty}\Big(\Big\|Q_{\psi\theta}\big<\psi\big>^l\Big\|_{L^2}+\Big\|(Q_0)_{\psi\theta}\big<\psi\big>^l\Big\|_{L^2}\Big)
\nonumber\\&\quad+C\big\|Q_{\psi}
+(Q_0)_{\psi}\big\|_{H^1}\big\|(Q_{\theta}+(Q_0)_{\theta})\big<\psi\big>^l\big\|_{H^1}.
\end{align*}
Hence by the Sobolev embedding, we deduce that
\begin{align*}
\Big\|\partial_{\psi\theta}\mathcal{H}(Q)\big<\psi\big>^l\Big\|_2\leq C\Big(\sum_{j+k\leq 2}\Big\|\partial_\theta^j\partial_\psi^k Q\big<\psi\big>^l\Big\|_2+\sum_{j+k\leq 2}\Big\|\partial_\theta^j\partial_\psi^k Q_0\big<\psi\big>^l\Big\|_2\Big)^2.
\end{align*}
Same estimates hold for $\Big\|\partial_{\theta\theta}\mathcal{H}(Q)\big<\psi\big>^l\Big\|_2$ and $\Big\|\partial_{\psi\psi}\mathcal{H}(Q)\big<\psi\big>^l\Big\|_2$. Thus, combing the estimate (\ref{lower order estimate on H}), we arrive at
\begin{align*}
&\sum_{j+k\leq 2}\Big\|\partial_\theta^j\partial_\psi^k \mathcal{H}(Q)\big<\psi\big>^l\Big\|_2\\
\leq& C\Big(\sum_{j+k\leq 2}\Big\|\partial_\theta^j\partial_\psi^k Q\big<\psi\big>^l\Big\|_2+\sum_{j+k\leq 2}\Big\|\partial_\theta^j\partial_\psi^k Q_0\big<\psi\big>^l\Big\|_2\Big)^2.
\end{align*}
Using (\ref{estimate on H}) and repeat the above argument, we obtain (\ref{nonlinear estimate on H}).

Consequently,  if we take $r_0=\|Q_0\|_{X}$ and $\eta$ small enough, then
\begin{align*}
\|\mathcal{L}\mathcal{H}(Q)\|_{X}\leq C\|\mathcal{H}(Q)\|_{X}\leq 4C r_0^2\leq 4C\eta r_0\leq r_0,\ \ \forall Q\in B_0,
\end{align*}
here we have used (\ref{norm of q0}). Thus, $\mathcal{L}\mathcal{H}$ is a continuous map from $B_0$ to $B_0$ for small $\eta$.\\

{\bf Step 4:}{ \it $\mathcal{L}\circ\mathcal{H}$ is a contraction map in $B_0$.}

Noting firstly that
\begin{align*}
\mathcal{H}(Q_1)-\mathcal{H}(Q_2)&=(Q_1-Q_2)\bigg(1-\frac{2u_e(1)}{\sqrt{Q_1+Q_0+u_e^2(1)}+\sqrt{Q_2+Q_0+u_e^2(1)}}\bigg)
\end{align*}
and
\begin{align*}
&1-\frac{2u_e(1)}{\sqrt{Q_1+Q_0+u_e^2(1)}+\sqrt{Q_2+Q_0+u_e^2(1)}}
\nonumber\\&=\frac{Q_1+Q_0}{\big(\sqrt{Q_1+Q_0+u_e^2(1)}+\sqrt{Q_2+Q_0+u_e^2(1)}\big)\big(\sqrt{Q_1+Q_0+u_e^2(1)}+u_e(1)\big)}
\nonumber\\&\quad+\frac{Q_2+Q_0}{\big(\sqrt{Q_1+Q_0+u_e^2(1)}+\sqrt{Q_2+Q_0+u_e^2(1)}\big)\big(\sqrt{Q_2+Q_0+u_e^2(1)}+u_e(1)\big)}.
\end{align*}

With the help of the similar arguments in Step 3, we can obtain that there exist $\eta_0>0$ such that for any $\eta\in (0,\eta_0)$, there holds
\begin{align*}
\big\|\mathcal{L}\mathcal{H}(Q_1)-\mathcal{L}\mathcal{H}(Q_2)\big\|_X
&\leq C\big\|\mathcal{H}(Q_1)-\mathcal{H}(Q_2)\big\|_X
\\[5pt]&\leq C\big\|Q_1-Q_2\big\|_X\big(\|Q_0\|_X+\|Q_1\|_X+\|Q_2\|_X\big)
\\[5pt]&\leq C\|Q_0\|_X\big\|Q_1-Q_2\big\|_X
\\[5pt]&\leq\frac{1}{2}\big\|Q_1-Q_2\big\|_X,
\end{align*}
that is, $\mathcal{L}\mathcal{H}$ is a contraction map in $B_0$.

{\bf Step 5:} { \it Existence and uniqueness of the equation (\ref{modified prandtl equation}).}
 
 By the standard contraction mapping principle, we know that there exists $\eta_0>0$ such that for any $\eta\in (0,\eta_0)$ and any $m,l\in \mathbb{N}\cup \{0\}$, the equation (\ref{modified prandtl equation}) has a unique solution $\mathcal{U}^{(0)}$ which satisfies
\begin{align*}
\sum_{j+k\leq m}\int_{-\infty}^0\int_0^{2\pi}\Big|\partial_\theta^j\partial_\psi^k \big(\mathcal{U}^{(0)}-u_e(1)\big)\Big|^2\big<\psi\big>^{2l}d\theta d\psi\leq C(m,l)\eta^2,
\end{align*}
this completes the proof of this proposition.
\end{proof}

Return to the equations (\ref{prandtl problem 1}), we have the following result.
\begin{Corollary} There exists $\eta_0>0$ such that for any $\eta\in(0,\eta_0)$, equations (\ref{prandtl problem 1}) have a unique solution $(u_p^{(0)},v_p^{(1)})$ which satisfies
\begin{align}\label{decay behavior-prandtl}
\sum_{j+k\leq m}\int_{-\infty}^0\int_0^{2\pi}\Big|\partial_\theta^j\partial_Y^k (u_p^{(0)},v_p^{(1)})\Big|^2\big<Y\big>^{2l}d\theta dY\leq C(m,l)\eta^2, \ m,l \geq 0.
\end{align}
\end{Corollary}

Notice that
$$v^{(1)}_p(\theta,Y)=\int_{-\infty}^Y\partial_Yv^{(1)}_{p}(\theta,Y')dY'=-\int_{-\infty}^Y\partial_\theta u^{(0)}_{p}(\theta,Y')dY',$$
we have
\begin{align*}
\int_{0}^{2\pi}v^{(1)}_p(\theta, Y)d\theta=0, \ \forall\ Y\leq 0.
\end{align*}
Finally, solving (\ref{equation of first pressure}), we obtain $p_p^{(1)}(\theta,Y)$  which decay very fast as $Y\rightarrow -\infty$.

\begin{Remark}
Kim established the well-posedness of Prandtl equations (\ref{prandtl problem near 1}) in \cite{K2000}. He first rewrote the equations in an equivalent integral form, then constructed a sequence of  approximate solutions by Picard iteration and studied the structure of approximate solution sequence by infinite series expansion, finally proved the convergence of approximate solution sequence in some suitable space. In this paper, we use a different approach. Like Kim, we also first rewrite the equations in an equivalent integral form, then we establish a priori estimate in a different space by energy method, and the well-posedness of Prandtl equations (\ref{prandtl problem near 1}) can be obtained by contraction mapping principle directly. The flows constructed above have non-monotone velocity profile. It was shown by Renardy that solutions of steady Prandtl equations with monotone profile do not exist.
\end{Remark}
\subsubsection{Linearized Euler equations for $(u_e^{(1)}, v_e^{(1)}, p_e^{(1)})$ and their solvabilities}
\begin{Proposition}\label{solvability of first Euler}
The linearized Euler equations (\ref{outer-1 order equation}) have a solution $(u_e^{(1)}, v_e^{(1)}, p_e^{(1)}))$ which satisfies
\begin{align}
|\partial_\theta u_e^{(1)}+v_e^{(1)}|(\theta,r)\leq& C\eta r, \ |\partial_\theta v_e^{(1)}-u_e^{(1)}|(\theta,r)\leq C\eta r,  \ \forall (\theta, r)\in \Omega, \label{Estimate of first combined linearized Euler equation}\\[5pt]
 \|\partial^k_\theta\partial^j_r(u_e^{(1)},v_e^{(1)})\|_2\leq& C(k,j)\eta, \quad \forall j,k\geq 0,\label{Estimate of first linearized Euler equation}\\[5pt]
 r^2\triangle u_e^{(1)}-u_e^{(1)}+2\partial_\theta& v_e^{(1)}=0,\quad \int_{0}^{2\pi}v_e^{(1)}(\theta,r)d\theta=0,\label{identity for first Euler}
\end{align}
here and below, $\Delta=\partial_{rr}+\frac{\partial_r}{r}+\frac{\partial_{\theta\theta}}{r^2}.$
\end{Proposition}
\begin{proof}
 Eliminating the pressure $p_e^{(1)}$ in the equation (\ref{outer-1 order equation}), we obtain the following equation for $rv_e^{(1)}$ in $\Omega$
\begin{eqnarray}\label{equation for first Euler normal}
\left \{
\begin {array}{ll}
-r\triangle(rv_e^{(1)})=0,\\[5pt]
rv_e^{(1)}|_{r=1}=-v_p^{(1)}|_{Y=0}.
\end{array}
\right.
\end{eqnarray}
Since $\int_{0}^{2\pi}v^{(1)}_pd\theta=0$, we can assume
\begin{align*}
-v_p^{(1)}(\theta,0)=\sum_{n=1}^{+\infty}[a_{n1} \cos(n\theta)+b_{n1}\sin(n\theta)].
\end{align*}
By (\ref{decay behavior-prandtl}), we deduce that
\begin{align}\label{deacy of Fourier coff}
|a_{n1}|+|b_{n1}|\leq C\frac{\eta}{n^k}, \ \forall k\geq 0.
\end{align}
  It's easy to justify that
\begin{align*}
v_e^{(1)}(\theta,r)=\sum_{n=1}^{+\infty}[a_{n1} r^{n-1} \cos(n\theta)+b_{n1}r^{n-1}\sin(n\theta)]
\end{align*}
solves the equation (\ref{equation for first Euler normal}).
Set
\begin{align*}
u_e^{(1)}(\theta, r)=\sum_{n=1}^{+\infty}[-a_{n1} r^{n-1} \sin(n\theta)+b_{n1}r^{n-1}\cos(n\theta)],
\end{align*}
then
\begin{align*}
\partial_\theta u_e^{(1)}+\partial_r(r v_e^{(1)})&=0, \   r^2\triangle u_e^{(1)}-u_e^{(1)}+2\partial_\theta v_e^{(1)}=0 , \\[5pt]
\|\partial^k_\theta\partial^j_r(u_e^{(1)},v_e^{(1)})\|_2&\leq C(k,j)\eta, \quad \forall k,j\geq 0,
\end{align*}
which give (\ref{Estimate of first linearized Euler equation}) and (\ref{identity for first Euler}).

Moreover, there hold
\begin{align*}
\partial_\theta u_e^{(1)}(\theta,r)+v_e^{(1)}(\theta,r)&=\sum_{n=2}^{+\infty}[(1-n)a_{n1} r^{n-1} \cos(n\theta)+(1-n)b_{n1}r^{n-1}\sin(n\theta)], \\
\partial_\theta v_e^{(1)}(\theta,r)-u_e^{(1)}(\theta,r)&=\sum_{n=2}^{+\infty}[(1-n)a_{n1} r^{n-1} \sin(n\theta)+(n-1)b_{n1}r^{n-1}\cos(n\theta)].
\end{align*}
Thus, we obtain (\ref{Estimate of first combined linearized Euler equation}) by using (\ref{deacy of Fourier coff}).
After obtaining $(u_e^{(1)}, v_e^{(1)})$, we construct $p_e^{(1)}$ as following
\beno
p_e^{(1)}(\theta,r):=\phi(r)-\int_0^\theta[u_e(r)\partial_{\theta'} u_e^{(1)}+ru'_e v_e^{(1)}+u_ev_e^{(1)}](\theta',r)d\theta',
\eeno
where $\phi(r)$ is a function which satisfies
\beno
r\partial_r\phi(r)+u_e(r)\partial_\theta v_e^{(1)}(0,r)-2u_e(r)u_e^{(1)}(0,r)=0.
\eeno
Combining the equation of $(u_e^{(1)}, v_e^{(1)})$, it's direct to obtain
\beno
u_e \partial_\theta v_e^{(1)}-2u_eu_e^{(1)}+r\partial_rp_e^{(1)}=0.
\eeno
Hence, $(u_e^{(1)}, v_e^{(1)},p_e^{(1)})$ solves the equations (\ref{outer-1 order equation}).
\end{proof}

\subsubsection{Linearized Prandtl equations for  $(u_p^{(1)},v_p^{(2)})$ and their solvabilities}
\indent

In this subsubsection, we consider the solvability of linearized Prandtl equations (\ref{first linearized prandtl problem near 1})

\begin{Proposition}\label{decay estimates of linearized Prandtl} There exists $\eta_0>0$ such that for any $\eta\in(0,\eta_0)$, equations (\ref{first linearized prandtl problem near 1}) have a unique solution $(u_p^{(1)},v_p^{(2)})$ which satisfies
\begin{align}
&\sum_{j+k\leq m}\int_{-\infty}^0\int_0^{2\pi}\big|\partial_\theta^j\partial_Y^k \big(u_p^{(1)}-A_{1\infty},v_p^{(2)}\big)\big|^2\big<Y\big>^{2l}d\theta dY\leq C(m,l)\eta^2, \ \ m, l\geq 0; \nonumber\\
& \int_0^{2\pi}v_p^{(2)}(\theta,Y)d\theta =0, \ \forall \ Y\leq 0,  \label{decay behavior-prandtl-1}
\end{align}
where
$A_{1\infty}:=\lim_{Y\rightarrow -\infty}u_p^{(1)}(\theta,Y)$ is a constant which satisfies $|A_{1\infty}|\leq C\eta.$
\end{Proposition}
\begin{proof}
Let $\eta\in C_c^\infty ((-\infty,0])$ satisfy
\begin{align*}
\eta(0)=1,\quad \int_0^{+\infty}\eta(y)dy=0.
\end{align*}
For simple, we set
\begin{align*}
\bar{u}:&=u_e(1)+u_p^{(0)}, \quad \bar{v}:=v_e^{(1)}(\theta,1)+v_p^{(1)},\\
u:&=u_p^{(1)}+u_e^{(1)}(\theta,1)\eta(Y), \quad v:=v_p^{(2)}-v_p^{(2)}(\theta,0)+Y v_p^{(1)}-\partial_\theta u_e^{(1)}(\theta,1)\int_0^Y\eta(z)dz.
\end{align*}
Then, the equations (\ref{first linearized prandtl problem near 1}) reduce to
\begin{eqnarray}\label{new linearized Prandtl equation}
\left \{
\begin {array}{ll}
\bar{u}\partial_\theta u+\bar{v}\partial_Yu+u\partial_\theta \bar{u}+v\partial_Y\bar{u}-\partial_{YY}u=\tilde{f},\\[7pt]
\partial_\theta u+\partial_Yv=0,\\[5pt]
u(\theta,Y)=u(\theta+2\pi,Y),\quad v(\theta,Y)=v(\theta+2\pi,Y)\\[5pt]
u|_{Y=0}=v|_{Y=0}=0,\ \ \lim_{Y\rightarrow -\infty}\partial_Yu=0,
\end{array}
\right.
\end{eqnarray}
where $\tilde{f}(\theta,Y)$ is $2\pi$-periodic function and decay fast as $Y\rightarrow -\infty$.
We can solve the equations (\ref{new linearized Prandtl equation}) by consider the following approximate system. Let $\delta>0$ be a constant, we consider the following elliptic equation.
\begin{eqnarray}\label{appro linear prandtl}
\left \{
\begin {array}{ll}
\bar{u}\partial_\theta \ude+\bar{v}\partial_Yu^\de+\big[\int_Y^0\p_{\theta}\ude(\theta,z)dz \big]\partial_Y\bar{u}+\ude\partial_\theta \bar{u}-\partial_{YY}\ude-\de \p_{\theta\theta} u^\de=\tilde{f},\\[7pt]
\ude(\theta,Y)=\ude(\theta+2\pi,Y),\\[5pt]
\ude|_{Y=0}=0.
\end{array}
\right.
\end{eqnarray}
We expect the solution of this equation is in $\dot{H}^1_0=\{u|\p_{\theta}u\in L^2,\p_Y u\in L^2, u|_{Y=0}=0\}$ rather than $H^1_0=\{u|u\in L^2,\p_{\theta}u\in L^2,\p_Y u\in L^2, u|_{Y=0}=0\}$. Now we establish apriori estimate of equation (\ref{appro linear prandtl}).
Multiplying the first equation in (\ref{appro linear prandtl}) by $\ude$ and integrating in $(\theta,r)\in(0,2\pi)\times(-\infty,0)$, we obtain that
\begin{align*}
&\int_{-\infty}^{0}\int_0^{2\pi}\Big[\bar{u}\partial_\theta \ude+\bar{v}\partial_Yu^\de+\Big(\int_Y^0\p_{\theta}\ude(\theta,z)dz \Big)\partial_Y\bar{u}+\ude\partial_\theta \bar{u}-\partial_{YY}\ude-\de \p_{\theta\theta} u^\de\Big]\ude d\theta dY\nonumber\\
=&\int_{-\infty}^{0}\int_0^{2\pi}\tilde{f}\ude d\theta dY.
\end{align*}
It's easy to get
\beno
\int_{-\infty}^{0}\int_0^{2\pi}\big[-\partial_{YY}\ude-\de \p_{\theta\theta} u^\de\big]\ude d\theta dY=\|\partial_Y\ude\|_2^2 +\de \|\p_{\theta} u^\de\|_2^2.
\eeno

Recall the estimates (\ref{decay behavior-prandtl}) and (\ref{Estimate of first linearized Euler equation}), we have
\begin{align*}
 &\big|\p^j_{\theta}\p^k_{Y}(\bar{u}-u_e(1))\big<Y\big>^l\big|\leq C(j,k,l)\eta,\\[5pt] &\big|\p^j_{\theta}\p^k_{Y}(\bar{v}-v^{(1)}_e(\theta,1))\big<Y\big>^l\big|\leq C(j,k,l)\eta,\quad
\big|\p^j_\theta v^{(1)}_e(\theta,1)\big|\leq C(j)\eta,
\end{align*}
thus, we can deduce that
\begin{align*}
&-\int_{-\infty}^{0}\int_0^{2\pi}\Big[\bar{u}\partial_\theta \ude+\bar{v}\partial_Yu^\de+\Big(\int_Y^0\p_{\theta}\ude(\theta,z)dz \Big)\partial_Y\bar{u}+\ude\partial_\theta \bar{u}\Big]\ude d\theta dY+\int_{-\infty}^{0}\int_0^{2\pi}\tilde{f}\ude d\theta dY\\
\leq&\int_{-\infty}^{0}\int_0^{2\pi}\frac{1}{2}\big[\p_\theta\bar{u}+\p_Y\bar{v}\big]\big(\ude\big)^2 d\theta dr+\|Y^2\bar{u}_Y\|_{\infty}\Big\|\frac{\int_Y^0\p_{\theta}\ude dz}{Y}\Big\|_2\Big\|\frac{\ude}{Y}\Big\|_2\\
    &+\|Y^2\bar{u}_\theta\|_{\infty}\Big\|\frac{\ude}{Y}\Big\|_2^2+\|Y\tilde{f}\|_2\Big\|\frac{\ude}{Y}\Big\|_2\\[5pt]
\leq& C\eta \big[\|\p_{\theta}\ude\|_2^2+\|\p_Y\ude\|_2^2\big]+C\|Y\tilde{f}\|_2\|\p_Y \ude\|_2,
\end{align*}
where we used the $\bar{u}_\theta+\bar{v}_Y=0$ and the Hardy inequality
$$
\Big\|\frac{\int_Y^0\p_{\theta}\ude dz}{Y}\Big\|_2\leq C\|\p_{\theta}\ude\|_2,\quad \Big\|\frac{\ude}{Y}\Big\|_2\leq C\|\p_Y\ude\|_2.
$$
Collecting the above estimates, we obtain
\begin{align*}
&\|\partial_Y\ude\|_2^2 +\de \|\p_{\theta} u^\de\|_2^2\leq C\eta \big[\|\p_{\theta}\ude\|_2^2+\|\p_Y\ude\|_2^2\big]+C\|Y\tilde{f}\|_2\|\p_Y \ude\|_2,
\end{align*}
where $C$ is independent on $\eta$ and $\de$. If $\eta$ is small enough, there holds
\begin{align}\label{energy estimate prandtl}
&\|\partial_Y\ude\|_2^2 +\de \|\p_{\theta} u^\de\|_2^2\leq C\eta \|\p_{\theta}\ude\|_2^2+C\|Y\tilde{f}\|_2^2.
\end{align}
Next we multiply the first equation in (\ref{appro linear prandtl}) by $\p_{\theta}\ude$ and integrate in $(\theta,r)\in(0,2\pi)\times(-\infty,0)$, we arrive at
\begin{align*}
&\int_{-\infty}^{0}\int_0^{2\pi}\Big[\bar{u}\partial_\theta \ude+\bar{v}\partial_Yu^\de+\Big(\int_Y^0\p_{\theta}\ude(\theta,z)dz \Big)\partial_Y\bar{u}+\ude\partial_\theta \bar{u}-\partial_{YY}\ude-\de \p_{\theta\theta} u^\de\Big]\p_{\theta}\ude d\theta dY\nonumber\\
=&\int_{-\infty}^{0}\int_0^{2\pi}\tilde{f}\p_{\theta}\ude d\theta dY.
\end{align*}
It's direct to obtain
\beno
\int_{-\infty}^{0}\int_0^{2\pi}\bar{u}\partial_\theta \ude \p_{\theta}\ude d\theta dY= \int_{-\infty}^{0}\int_0^{2\pi}\bar{u}\big|\partial_\theta \ude\big|^2d\theta dY\geq(\alpha-C\eta)\|\p_\theta \ude\|^2_2.
\eeno
The diffusion term can be computed as follows
\begin{align*}
\int_{-\infty}^{0}\int_0^{2\pi}\big[-\partial_{YY}\ude-\de \p_{\theta\theta} u^\de\big]\p_\theta \ude d\theta dY=\int_{-\infty}^{0}\int_0^{2\pi}\Big(\frac{1}{2}\p_{\theta}\big(\p_{Y}\ude\big)^2-\frac{\de}{2}\p_{\theta}\big(\p_\theta \ude\big)^2\Big)d\theta dY=0.
\end{align*}

Moreover, there holds
\begin{align*}
&-\int_{-\infty}^{0}\int_0^{2\pi}\Big[\bar{v}\partial_Yu^\de+\Big(\int_Y^0\p_{\theta}\ude(\theta,z)dz \Big)\partial_Y\bar{u}+\ude\partial_\theta \bar{u}\Big]\p_{\theta}\ude d\theta dY+\int_{-\infty}^{0}\int_0^{2\pi}\tilde{f}\p_{\theta}\ude d\theta dY\\
\leq& \|\bar{v}\|_{\infty}\|\p_Y \ude\|_2\|\p_{\theta}\ude\|_2+\|Y \bar{u}_Y\|_{\infty}\Big\|\frac{\int_Y^0\p_{\theta}\ude dz}{Y}\Big\|_2\|\p_{\theta} \ude\|_2\\
&+\|Y\p_{\theta}\bar{u}\|_{\infty}\Big\|\frac{\ude}{Y}\Big\|_2\|\p_{\theta}\ude\|_2+\|\tilde{f}\|_2\|\p_{\theta}\ude\|_2\\[5pt]
\leq&C\eta\big[\|\p_Y\ude\|^2_2+\|\p_\theta \ude\|^2_2\big]+\|\tilde{f}\|_2\|\p_{\theta}\ude\|_2.
\end{align*}
Thus, we obtain
\begin{align*}
&\alpha\|\p_{\theta}\ude\|_2^2\leq C \eta\big[\|\p_\theta\ude\|_2^2+\|\p_Y\ude\|_2^2\big]+\|\tilde{f}\|_2\|\p_{\theta}\ude\|_2.
\end{align*}
Since $C$ is independent on $\eta$ and $\de$, $\eta$ is small enough, hence there holds
\begin{align}\label{positive estimate prandtl}
&\alpha\|\p_{\theta}\ude\|_2^2\leq C \eta\big\|\p_Y\ude\|_2^2+\|\tilde{f}\|_2^2.
\end{align}

Combining (\ref{energy estimate prandtl}) and (\ref{positive estimate prandtl}), we have
\begin{align}\label{closed prandtl}
&\alpha\|\p_{\theta}\ude\|_2^2+\|\partial_Y\ude\|_2^2 +\de \|\p_{\theta} u^\de\|_2^2\leq C\|Y\tilde{f}\|_2^2+C\|\tilde{f}\|_2^2.
\end{align}
According the first equation in (\ref{appro linear prandtl}), we deduce
\begin{align}\label{second derivative estimate}
&\|\p_{YY}\ude\|^2_2+2\de\|\p_{\theta Y}\ude\|^2_2+\de^2\|\p_{\theta \theta}\ude\|^2_2=\Big\|\partial_{YY}\ude+\de \p_{\theta\theta} u^\de\Big\|_2^2\nonumber\\
=&\bigg\|\bar{u}\partial_\theta \ude+\bar{v}\partial_Yu^\de+\Big(\int_Y^0\p_{\theta}\ude(\theta,z)dz \Big)\partial_Y\bar{u}+\ude\partial_\theta \bar{u}-\tilde{f}\bigg\|_2^2\nonumber\\[5pt]
\leq &C\big[\|\tilde{f}\|^2_2+\|\p_{\theta}\ude\|_2^2+\|\partial_Y\ude\|_2^2\big]\leq C\|Y\tilde{f}\|_2^2+C\|\tilde{f}\|_2^2.
\end{align}
Collecting the estimate (\ref{closed prandtl}) and (\ref{second derivative estimate}), we obtain
\begin{align*}
&\alpha\|\p_{\theta}\ude\|_2^2+\|\partial_Y\ude\|_2^2 +\|\p_{YY}\ude\|^2_2+\de \|\p_{\theta} u^\de\|_2^2+2\de\|\p_{\theta Y}\ude\|^2_2+\de^2\|\p_{\theta \theta}\ude\|^2_2\leq C\|\big<Y\big>\tilde{f}\|_2^2.
\end{align*}
The above inequality shows the existences and uniqueness of solution about system (\ref{appro linear prandtl}) for any $\de>0$ in space $\dot{H}^1_0$, moreover, the solution is smooth if $\tilde{f}$ is smooth enough. Set
 $$u:=\lim_{\de\rightarrow0} \ude,\quad v:=\int_{Y}^0\p_{\theta}u(\theta,d)dz,$$
then
\begin{align}\label{new linear prandtl estimate}
&\alpha\|\p_{\theta}u\|_2^2+\|\partial_Y u\|_2^2 +\|\p_{YY} u\|^2_2\leq C\|\big<Y\big>\tilde{f}\|^2_2.
\end{align}
and $[u,v]$ solves the system (\ref{new linearized Prandtl equation}).

Finally, we show that the derivatives of $u,v$ decay fast as $Y\rightarrow -\infty$. Let $\psi=\int_0^Y \bar{u}(\theta,z)dz$ and
\begin{align*}
\tilde{u}(\theta,\psi)=u(\theta, Y(\theta,\psi)), \ \tilde{v}(\theta,\psi)=v(\theta, Y(\theta,\psi)), \
F(\theta,\psi)=\frac{\tilde{f}(\theta, Y(\theta,\psi))}{\bar{u}(\theta, Y(\theta,\psi))},
\end{align*}
then there hold
\begin{eqnarray}\label{linear prandtl in new variable}
\left \{
\begin {array}{ll}
\partial_\theta \tilde{u}-\partial_\psi(a(\theta,\psi)\partial_\psi \tilde{u})+b(\theta,\psi)\tilde{u}+c(\theta,\psi)\tilde{v}=F(\theta,\psi),\\[7pt]
\tilde{u}(\theta+2\pi,\psi)=\tilde{u}(\theta,\psi),\\[5pt]
\tilde{u}(\theta, 0)=0, \ \lim_{\psi\rightarrow -\infty}\partial_\psi\tilde{u}(\theta, \psi)=0,
\end{array}
\right.
\end{eqnarray}
where
\begin{align*}
a(\theta,\psi)=\bar{u}(\theta, Y(\theta,\psi)),
 \ b(\theta,\psi)=\frac{\partial_\theta\bar{u}(\theta, Y(\theta,\psi))}{\bar{u}(\theta, Y(\theta,\psi))},
 \ c(\theta,\psi)=\frac{\partial_Y\bar{u}(\theta, Y(\theta,\psi))}{\bar{u}(\theta, Y(\theta,\psi))}.
\end{align*}
Notice that there exist $\eta_0>0$ such that for any $\eta\in (0,\eta_0)$, there holds
\begin{align*}
\frac{\alpha}{2}\leq \bar{u}(\theta,Y)\leq\alpha, \ \forall (\theta, Y)\in [0,2\pi]\times (-\infty,0].
\end{align*}
 Thus, we deduce that $\frac{\alpha}{2}\leq \frac{|\psi|}{|Y|}\leq \alpha$.

 We claim that for any $l\in \mathbb{N}$, there holds
 \begin{align}\label{weight estimate in new variable}
 \|\partial_\theta \tilde{u} \psi^l\|_2+\|\partial_\psi \tilde{u}_{\neq}\psi^l\|_2+\|\partial_{\psi\psi} \tilde{u} \psi^l\|_2\leq C(\delta,l)\|F \big<\psi\big>^{l+1}\|_2,
 \end{align}
 where
 $$\tilde{u}_{\neq}=\tilde{u}-u_0(\psi), \ u_0(\psi)=\frac{1}{2\pi}\int_0^{2\pi}\tilde{u}(\theta,\psi)d\theta.$$
 From (\ref{new linear prandtl estimate}), we deduce that
 \begin{align*}
 \|\partial_\theta \tilde{u} \|_2+\|\partial_\psi \tilde{u}\|_2+\|\partial_{\psi\psi} \tilde{u} \|_2\leq C(\delta)\|F \big<\psi\big>\|_2,
 \end{align*}
 thus (\ref{weight estimate in new variable}) holds for $l=0$.

 For any $l\geq 1$, multiplying $\tilde{u}_{\neq}\psi^{2l}$ in (\ref{linear prandtl in new variable}) and integrating in $[0,2\pi]\times (-\infty,0]$, we obtain
 \begin{align*}
 &\underbrace{\int_0^{2\pi}\int_{-\infty}^0\partial_\theta \tilde{u}\tilde{u}_{\neq}\psi^{2l}d\psi d\theta}_{I_1} -\underbrace{\int_0^{2\pi}\int_{-\infty}^0\partial_\psi(a(\theta,\psi)\partial_\psi \tilde{u})\tilde{u}_{\neq}\psi^{2l}d\psi d\theta}_{I_2}\\
 =&\underbrace{\int_0^{2\pi}\int_{-\infty}^0[F(\theta,\psi)-b(\theta,\psi)\tilde{u}-c(\theta,\psi)\tilde{v}]\tilde{u}_{\neq}\psi^{2l}d\psi d\theta}_{I_3}.
 \end{align*}
Obviously, $I_1=0$. Due to the fast decay of $b(\theta,\psi), c(\theta,\psi)$ as $\psi\rightarrow -\infty$, we deduce that
\begin{align*}
|I_3|\leq& C\|F \big<\psi\big>^{l+1}\|_2\|\partial_\theta\tilde{u}_{\neq}\psi^{l-1}\|_2+C(\delta)(\|\partial_\theta \tilde{u} \|^2_2+\|\partial_\psi \tilde{u}\|^2_2)\\[5pt]
\leq &C(\delta)\|F \big<\psi\big>^{l+1}\|_2^2+C(\delta)\|\partial_\theta\tilde{u}\psi^{l-1}\|_2^2.
\end{align*}
Moreover, there holds
\begin{align*}
I_2=\underbrace{\int_0^{2\pi}\int_{-\infty}^0a(\theta,\psi)\partial_\psi \tilde{u}\partial_\psi\tilde{u}_{\neq}\psi^{2l}d\psi d\theta}_{I_{21}}
+\underbrace{2l\int_0^{2\pi}\int_{-\infty}^0a(\theta,\psi)\partial_\psi \tilde{u}\tilde{u}_{\neq}\psi^{2l-1}d\psi d\theta}_{I_{22}}.
\end{align*}
Notice that $a(\theta,\psi)=u_e(1)+u_p^{(0)}(\theta,Y(\theta,\psi))$, we deduce that
\begin{align*}
I_{21}=&u_e(1)\int_0^{2\pi}\int_{-\infty}^0\partial_\psi \tilde{u}\partial_\psi\tilde{u}_{\neq}\psi^{2l}d\psi d\theta+\int_0^{2\pi}\int_{-\infty}^0u_p^{(0)}(\theta,Y(\theta,\psi))\partial_\psi \tilde{u}\partial_\psi\tilde{u}_{\neq}\psi^{2l}d\psi d\theta\\
=&u_e(1)\int_0^{2\pi}\int_{-\infty}^0\partial_\psi \tilde{u}_{\neq}\partial_\psi\tilde{u}_{\neq}\psi^{2l}d\psi d\theta+\int_0^{2\pi}\int_{-\infty}^0u_p^{(0)}(\theta,Y(\theta,\psi))\partial_\psi \tilde{u}\partial_\psi\tilde{u}_{\neq}\psi^{2l}d\psi d\theta\\
\geq& \frac{\alpha}{2}\|\partial_\psi \tilde{u}_{\neq}\psi^l\|^2_2-C(\delta)\|\partial_\psi \tilde{u}\|^2_2.
\end{align*}
Moreover, by H\"{o}lder inequality and P\'{o}incare inequality,  there holds
\begin{align*}
|I_{22}|\leq C(\delta,l)\|\partial_\psi \tilde{u}_{\neq}\psi^l\|_2\|\partial_\theta \tilde{u}\psi^{l-1}\|_2.
\end{align*}
Thus, we obtain
\begin{align}\label{first order norm estimate}
\|\partial_\psi \tilde{u}_{\neq}\psi^l\|_2\leq C(\delta,l)\|\partial_\theta \tilde{u}\psi^{l-1}\|_2+C(\delta)\|F \big<\psi\big>^{l+1}\|_2.
\end{align}

Moreover, we can easily get
\begin{align*}
\partial_\theta \tilde{u}-a(\theta,\psi)\partial_{\psi\psi }\tilde{u}=F(\theta,\psi)+\partial_\psi a(\theta,\psi)\partial_\psi \tilde{u}-b(\theta,\psi)\tilde{u}-c(\theta,\psi)\tilde{v},
\end{align*}
thus there holds
\begin{align*}
\|[\partial_\theta \tilde{u}-a(\theta,\psi)\partial_{\psi\psi }\tilde{u}]\psi^l\|_2^2=\|[F(\theta,\psi)+\partial_\psi a(\theta,\psi)\partial_\psi \tilde{u}-b(\theta,\psi)\tilde{u}-c(\theta,\psi)\tilde{v}]\psi^l\|_2^2.
\end{align*}
The right side can be controlled by
\begin{align*}
C(\delta)\|F \big<\psi\big>^{l}\|_2^2+C(\delta)(\|\partial_\theta \tilde{u} \|^2_2+\|\partial_\psi \tilde{u}\|^2_2).
\end{align*}
Moreover, there holds
\begin{align*}
&\|[\partial_\theta \tilde{u}-a(\theta,\psi)\partial_{\psi\psi }\tilde{u}]\psi^l\|_2^2\\
=&\|\partial_\theta \tilde{u} \psi^l\|_2^2+\| a(\theta,\psi)\partial_{\psi\psi }\tilde{u}\psi^l\|_2^2-2\int_0^{2\pi}\int_{-\infty}^0a(\theta,\psi)\partial_\theta \tilde{u}\partial_{\psi\psi }\tilde{u}\psi^{2l}d\psi d\theta\\
\geq &\|\partial_\theta \tilde{u} \psi^l\|_2^2+\frac{\alpha}{2}\|\partial_{\psi\psi }\tilde{u}\psi^l\|_2^2\underbrace{-2\int_0^{2\pi}\int_{-\infty}^0a(\theta,\psi)\partial_\theta \tilde{u}\partial_{\psi\psi}\tilde{u}\psi^{2l}d\psi d\theta}_{I}.
\end{align*}
Integrating by parts, we deduce that
\begin{align*}
I=&\underbrace{2\int_0^{2\pi}\int_{-\infty}^0\partial_\psi a(\theta,\psi)\partial_\theta\tilde{u}\partial_{\psi}\tilde{u}\psi^{2l}d\psi d\theta}_{I_1}+\underbrace{4l\int_0^{2\pi}\int_{-\infty}^0a(\theta,\psi)\partial_\theta\tilde{u}\partial_{\psi}\tilde{u}\psi^{2l-1}d\psi d\theta}_{I_2}\\
&+\underbrace{2\int_0^{2\pi}\int_{-\infty}^0a(\theta,\psi)\partial_{\theta \psi}\tilde{u}\partial_{\psi}\tilde{u}\psi^{2l}d\psi d\theta}_{I_3}.
\end{align*}
Obviously, there holds
\begin{align*}
|I_1|+|I_3|\leq C(\delta)(\|\partial_\theta \tilde{u} \|^2_2+\|\partial_\psi \tilde{u}\|^2_2).
\end{align*}
Moreover,
\begin{align*}
|I_2|=&\Big|4lu_e(1)\int_0^{2\pi}\int_{-\infty}^0\partial_\theta\tilde{u}\psi^{2l-1}\partial_{\psi}\tilde{u}d\psi d\theta+4l\int_0^{2\pi}\int_{-\infty}^0u_p^{(0)}(\theta,Y(\theta,\psi))\partial_\theta\tilde{u}\psi^{2l-1}\partial_{\psi}\tilde{u}d\psi d\theta\Big|\\
=&\Big|4lu_e(1)\int_0^{2\pi}\int_{-\infty}^0\partial_\theta\tilde{u}\psi^{2l-1}\partial_{\psi}\tilde{u}_{\neq}d\psi d\theta+4l\int_0^{2\pi}\int_{-\infty}^0u_p^{(0)}(\theta,Y(\theta,\psi))\partial_\theta\tilde{u}\psi^{2l-1}\partial_{\psi}\tilde{u}d\psi d\theta\Big|\\[5pt]
\leq & C(l)\|\partial_\theta \tilde{u}\psi^l \|_2\|\partial_\psi \tilde{u}_{\neq}\psi^{l-1}\|_2+C(\delta,l)(\|\partial_\theta \tilde{u} \|^2_2+\|\partial_\psi \tilde{u}\|^2_2).
\end{align*}
Thus, we obtain
\begin{align}\label{second estimate on norm}
\|\partial_\theta \tilde{u} \psi^l\|_2+\|\partial_{\psi\psi }\tilde{u}\psi^l\|_2\leq C(\delta,l)\|\partial_\psi \tilde{u}_{\neq}\psi^{l-1}\|_2+C(\delta,l)\|F \big<\psi\big>^{l+1}\|_2.
\end{align}
Combining the estimate (\ref{first order norm estimate}) and (\ref{second estimate on norm}), we obtain that for any $l\geq 1$, there holds
\begin{align*}
\|\partial_\theta \tilde{u} \psi^l\|_2+\|\partial_\psi \tilde{u}_{\neq}\psi^l\|_2+\|\partial_{\psi\psi} \tilde{u} \psi^l\|_2\leq C(\delta,l)(\|\partial_\psi \tilde{u}_{\neq}\psi^{l-1}\|_2+\|\partial_\theta \tilde{u}\psi^{l-1}\|_2)+C(\delta,l)\|F \big<\psi\big>^{l}\|_2.
\end{align*}
Thus, by induction, we obtain (\ref{weight estimate in new variable}).

Furthermore, by the Hardy inequality, we have
\begin{align*}
\|\partial_\psi \tilde{u}\psi^{l-1}\|_2\leq C\|\partial_{\psi\psi} \tilde{u} \psi^{l}\|_2,
\end{align*}
hence there holds
\begin{align*}
 \|\partial_\theta \tilde{u} \psi^l\|_2+\|\partial_\psi \tilde{u}\psi^{l-1}\|_2+\|\partial_{\psi\psi} \tilde{u} \psi^l\|_2\leq C(\delta)\|F \big<\psi\big>^{l+1}\|_2.
 \end{align*}
Returning to the original variable, we obtain that for any $l\geq 1$, there holds
\begin{align*}
\big\|\big<Y\big>^{l}\p_{YY}u\big\|_2^2+\|\big<Y\big>^{l-1}\p_{\theta}u\|^2_2+\|\big<Y\big>^{l-1}\p_{Y}u\|^2_2\leq C(\delta, l)\|\big<Y\big>^{l+1}\tilde{f}\|^2_2.
\end{align*}

Furthermore, by induction, we obtain that for any $m\in \mathbb{N}_+, l\in \mathbb{N}$, there holds
\begin{align*}
&\sum_{j+k\leq m}\Big(\big\|\big<Y\big>^{l}\p^j_{\theta}\p^k_Y\p_{YY}u\big\|_2^2+\big\|\big<Y\big>^{l-1}\p^j_{\theta}\p^k_Y \p_\theta u\big\|_2^2+\big\|\big<Y\big>^{l-1}\p^j_{\theta}\p^k_Y \p_Y u\big\|_2^2\Big)\\
 &\leq C(\delta,m,l)\sum\limits_{j+k\leq m}\big\|\big<Y\big>^{l+1}\p^{j}_{\theta}\p^{k}_Y\tilde{f}\big\|^2_2\leq C(\delta,m,l)\eta^2.
\end{align*}

Noting that  $\lim\limits_{Y\rightarrow -\infty}(u_\theta,u_Y)=0$ and $A_{1\infty}:=\lim\limits_{Y\rightarrow -\infty}u(\theta,Y)$ is a constant independent on $\theta$,  then by the Hardy inequality  we have for any $l\geq2$
\begin{align*}
\|Y^{l-2} (u-A_{1\infty})\|_2\leq C(\delta,l) \|Y^{l-1} \partial_{Y}u\|_2\leq C(\delta,l)\eta^2.
\end{align*}
This completes the proof of this proposition.
\end{proof}

\begin{Remark}
Unlike the usual case, we can not expect $\lim_{Y\rightarrow-\infty}u^{(1)}_p=0$ due to the periodic boundary condition on $\theta$ direction. However, the behaviour on $-\infty$ of boundary layer profile $u^{(1)}_p$ does not affect the out flow $[u^{(1)}_e, v^{(1)}_e]$ because $[u^{(1)}_e+A_{1\infty}, v^{(1)}_e]$ also solves the equations (\ref{outer-1 order equation}). Motivated by this observation, we modify the Euler flow $[u^{(1)}_e, v^{(1)}_e]$ in (\ref{modify Euler}).
\end{Remark}
Next, we construct the pressure $p_p^{(2)}(\theta,Y)$. Consider the equation
\begin{align}\label{equation for second pressure}
\partial_Yp_p^{(2)}(\theta, Y)=g_1(\theta,Y), \quad \lim_{Y\rightarrow -\infty}p_p^{(2)}(\theta,Y)=0,
\end{align}
where
\begin{align*}
g_1(\theta,Y)=&-Y\partial_Yp_p^{(1)}+\partial_{YY}v_p^{(1)}-u_e(1)\partial_\theta v_p^{(1)}-u_p^{(0)} (\partial_\theta v_e^{(1)}(\theta,1)+\partial_\theta v_p^{(1)})\\[5pt]
&-\partial_Yv_p^{(1)}(v_e^{(1)}(\theta,1)+v_p^{(1)})-2(Yu'_e(1)u_p^{(0)}+u_e(1)\tilde{u}_p^{(1)}+[u_e^{(1)}(\theta,1)+A_1]u_p^{(0)}+u_p^{(0)}\tilde{u}_p^{(1)},
\end{align*}
here and below,
$$\tilde{u}_p^{(1)}=u_p^{(1)}-A_{1\infty}.$$
$g_1(\theta,Y)$ can be obtained by replacing $u_p^{(1)}$ by $\tilde{u}_p^{(1)}$ in the expansion (\ref{first order expansion}) and putting the new expansion into the second equation of (\ref{NS-curvilnear}), then collecting the $\varepsilon^1$-order terms together.
 Notice that $g_1(\theta,Y)$ decay fast as $Y\rightarrow -\infty$, we can get $p_p^{(2)}(\theta,Y)$ by solving (\ref{equation for second pressure}) and deduce that $p_p^{(2)}(\theta,Y)$ decay fast as $Y\rightarrow -\infty$.

\subsubsection{Linearized Euler equations for $(u_e^{(2)}, v_e^{(2)}, p_e^{(2)})$ and their solvabilities}
\indent

Let $\chi(r)\in C^\infty([0,1])$ be an increasing smooth function such that
\begin{align}\label{cut-off function}
\chi(r)=
\left\{
\begin{array}{lll}
0, \quad r\in [0,\frac12], \\[5pt]
1, \quad r\in [\frac34,1].
\end{array}
\right.
\end{align}
Then, let $\phi_1(r)=-A_{1\infty}\big(r\chi''(r)+\chi'(r)-\frac{\chi(r)}{r}\big)$ and
\begin{align*}
A_1(r):=a_1r+r\int_0^r\frac{\phi_1(s)}{2s}-\frac{1}{r}\int_0^r\frac{s\phi_1(s)}{2}ds,
\end{align*}
where $a_1$ is a constant such that $A_1(1)=0$. Obviously, $|a_1|\leq C\eta$ and
\begin{align}\label{corrector of first order Euler equation}
\left\{
\begin{array}{ll}
rA''_1(r)+A'_1(r)-\frac{A_1(r)}{r}=-\phi_1(r),\ 0<r\leq 1 \\[5pt]
A_1(1)=0.
 \end{array}
 \right.
\end{align}
Direct computation gives $\|\partial_r^kA_1(r)\|_\infty\leq C(k)\eta.$  Moreover, notice that $\chi(r)=0$ for $r\leq \frac12$, we deduce that $A_1(r)=a_1r$ for $r\leq \frac12$.

Set
\begin{align}\label{modify Euler}
\begin{aligned}
\tilde{u}_e^{(1)}(\theta,r):&=u_e^{(1)}(\theta,r)+\chi(r)A_{1\infty}+A_1(r), \\[5pt]
\tilde{v}_e^{(1)}(\theta,r):&=v_e^{(1)}(\theta,r),\\
\tilde{p}_e^{(1)}(\theta,r):&=p_e^{(1)}(\theta,r)+2a\int_{0}^r[\chi(s)A_{1\infty}+A_1(s)]ds,
\end{aligned}
\end{align}
then $(\tilde{u}_e^{(1)},\tilde{v}_e^{(1)},\tilde{p}_e^{(1)})$ also satisfies the linearized Euler equations (\ref{outer-1 order equation})
with the boundary condition (\ref{outer-1 order-bc}). Moreover, there holds
\begin{align}\label{Estimate of modified first linearized Euler equation}
\begin{aligned}
|\partial_\theta \tilde{u}_e^{(1)}+\tilde{v}_e^{(1)}|(\theta,r)\leq& C\eta r, \ |\partial_\theta \tilde{v}_e^{(1)}-\tilde{u}_e^{(1)}|(\theta,r)\leq C\eta r, \ \forall (\theta, r)\in \Omega, \\[5pt]
 \|\partial^k_\theta\partial^j_r(\tilde{u}_e^{(1)},\tilde{v}_e^{(1)})\|_2\leq& C(k,j)\eta, \quad \forall j,k\geq 0;\\
  r^2\triangle \tilde{u}_e^{(1)}-\tilde{u}_e^{(1)}+2\partial_\theta& \tilde{v}_e^{(1)}=0, \quad
  \int_{0}^{2\pi}\tilde{v}_e^{(1)}d\theta=0.
\end{aligned}
\end{align}
Putting
\begin{align*}
&u^{\varepsilon}(\theta,r)=u_e(r)+\varepsilon \tilde{u}_e^{(1)}(\theta,r)+\varepsilon^2 u_e^{(2)}(\theta,r)+\cdots,\\[5pt]
&v^{\varepsilon}(\theta,r)=\varepsilon \tilde{v}_e^{(1)}(\theta,r)+\varepsilon^2 v_e^{(2)}(\theta,r)+\cdots,\\[5pt]
&p^{\varepsilon}(\theta,r)=p_e(r)+\varepsilon \tilde{p}_e^{(1)}(\theta,r)+\varepsilon^2 p_e^{(2)}(\theta,r)+\cdots
\end{align*}
into the Navier-Stokes equations (\ref{NS-curvilnear}), we obtain the following linearized Euler equations for $(u_e^{(2)},v_e^{(2)}, p_e^{(2)})$
\begin{eqnarray}
\left \{
\begin {array}{ll}
ar \partial_\theta u_e^{(2)}+2arv_e^{(2)}+\partial_\theta p_e^{(2)}+\tilde{u}_e^{(1)}\partial_\theta \tilde{u}_e^{(1)}+\tilde{v}_e^{(1)}r\partial_r\tilde{u}_e^{(1)}+\tilde{u}_e^{(1)}\tilde{v}_e^{(1)}=0,\\[5pt]
ar\partial_\theta v_e^{(2)}-2aru_e^{(2)}+r\partial_rp_e^{(2)}+\tilde{u}_e^{(1)}\partial_\theta \tilde{v}_e^{(1)}+\tilde{v}_e^{(1)}r\partial_r \tilde{v}_e^{(1)}-(\tilde{u}_e^{(1)})^2=0,\\[7pt]
\partial_\theta u_e^{(2)}+r\partial_rv_e^{(2)}+ v_e^{(2)}=0,\label{outer-2 order equation}
\end{array}
\right.
\end{eqnarray}
 with the boundary condition
\begin{align}\label{boundary condition of second Euler}
 v_e^{(2)}|_{r=1}=-v_p^{(2)}|_{Y=0}, \quad v_e^{(2)}(\theta,r)=v_e^{(2)}(\theta+2\pi,r).
\end{align}

\begin{Proposition}\label{solvability of second Euler equation}
The linearized Euler equations  (\ref{outer-2 order equation}) have a solution $(u_e^{(2)}, v_e^{(2)}, p_e^{(2)}))$ which satisfies
\begin{align*}
|\partial_\theta u_e^{(2)}+v_e^{(2)}|(\theta,r)\leq& C\eta r, \ |\partial_\theta v_e^{(2)}-u_e^{(2)}|(\theta,r)\leq C\eta r,  \ \forall (\theta, r)\in \Omega, \\[5pt]
 \|\partial^k_\theta\partial^j_r(u_e^{(2)},v_e^{(2)})\|_2\leq& C(k,j)\eta, \quad \forall j,k\geq 0,\\
  r^2\triangle u_e^{(2)}-u_e^{(2)}+2\partial_\theta &v_e^{(2)}=0, \quad \int_{0}^{2\pi}v_e^{(2)}d\theta=0.
\end{align*}
\end{Proposition}
\begin{proof}
 Eliminating the pressure $p_e^{(2)}$ in the equation (\ref{outer-2 order equation}), we obtain
\begin{align*}
-ar^2\triangle(rv_e^{(2)})-\tilde{u}_e^{(1)}r\triangle(r\tilde{v}_e^{(1)})+\tilde{v}_e^{(1)}(r^2\triangle \tilde{u}_e^{(1)}-\tilde{u}_e^{(1)}+2\partial_\theta \tilde{v}_e^{(1)})=0.
\end{align*}
Recall that $\triangle(rv_e^{(1)})=0$ and using (\ref{Estimate of modified first linearized Euler equation}), we obtain the following equation for $rv_e^{(2)}$ in $\Omega$
\begin{align*}
\left\{\begin{array}{lll}
-ar^2\triangle(rv_e^{(2)})=0,\\[5pt]
rv_e^{(2)}|_{r=1}=-v_p^{(2)}(\theta,0).
\end{array}
\right.
\end{align*}
Then, we can complete the proof of this proposition by following the argument of Proposition \ref{solvability of first Euler} line by line, we omit the details.
\end{proof}

\subsubsection{Linearized Prandtl equations for  $(u_p^{(2)},v_p^{(3)})$ and their solvabilities}
\indent

Putting the expansion
\begin{align*}
&u^\varepsilon(\theta,r)=u_e(r)+u_p^{(0)}(\theta,Y)+\varepsilon\big[\tilde{u}_e^{(1)}(\theta,r)+\tilde{u}_p^{(1)}(\theta,Y)\big]
+\varepsilon^2\big[u_e^{(2)}(\theta,r)+u_p^{(2)}(\theta,Y)\big]+\cdots,\\[5pt]
&v^\varepsilon(\theta,r)=\varepsilon\big[\tilde{v}_e^{(1)}(\theta,r)+v_p^{(1)}(\theta,Y)\big]
+\varepsilon^2\big[v_e^{(2)}(\theta,r)+v_p^{(2)}(\theta,Y)\big]+\varepsilon^3[v_e^{(3)}(\theta,r)+v_p^{(3)}(\theta,Y)]+\cdots,\\[5pt]
&p^\varepsilon(\theta,r)=p_e(r)+\varepsilon\big[\tilde{p}_e^{(1)}(\theta,r)+p_p^{(1)}(\theta,Y)\big]
+\varepsilon^2\big[p_e^{(2)}(\theta,r)+p_p^{(2)}(\theta,Y)\big]+\varepsilon^3p_p^{(3)}(\theta,Y)+\cdots
\end{align*}
with the boundary conditions
\begin{align*}
 u_e^{(2)}(\theta,1)+u_p^{(2)}(\theta,0)=0,\ v_e^{(3)}(\theta,1)+v_p^{(3)}(\theta,0)=0, \ \lim_{Y\rightarrow -\infty}(\partial_Yu_p^{(2)},v_p^{(3)})=(0,0)
\end{align*}
into the first and third equation of (\ref{NS-curvilnear}), collecting $\varepsilon^2$-order terms together, we obtain the following linearized steady Prandtl equations for $(u_p^{(2)},v_p^{(3)})$
\begin{eqnarray}
\left \{
\begin {array}{ll}
\big(u_e(1)+u_p^{(0)}\big)\partial_\theta u_p^{(2)}+\big(v_e^{(1)}(\theta,1)+ v_p^{(1)}\big)\partial_Yu_p^{(2)}+(v_e^{(3)}(\theta,1)+v_p^{(3)})\partial_{Y}u_p^{(0)}\\[5pt]
\quad \quad \quad \quad +(u_p^{(2)}+u_e^{(2)}(\theta,1))\partial_\theta u_p^{(0)}-\partial_{YY}u_p^{(2)}=f_2(\theta,Y),\\[5pt]
\partial_\theta u_p^{(2)}+\partial_Yv_p^{(3)}+\partial_Y(Yv_p^{(2)})=0,\\[5pt]
u_p^{(2)}(\theta,Y)=u_p^{(2)}(\theta+2\pi,Y),\quad v_p^{(3)}(\theta,Y)=v_p^{(3)}(\theta+2\pi,Y),\\[5pt]
u_p^{(2)}\big|_{Y=0}=-u_e^{(2)}\big|_{r=1},\ \ \lim_{Y\rightarrow -\infty}(\partial_Yu_p^{(2)},v_p^{(3)})=(0,0),
\label{second linearized prandtl problem near 1}
\end{array}
\right.
\end{eqnarray}
where
{\small\begin{align*}
f_2(\theta,Y)=&-\partial_\theta p_p^{(2)}+Y\partial_{YY}u_p^{(1)}+\partial_Yu_p^{(1)}+\partial_{\theta\theta}u_p^{(0)}-u_p^{(0)}
-\tilde{u}_p^{(1)}\partial_\theta u_p^{(1)}-v_p^{(2)}\partial_Y u_p^{(1)}-\sum_{i+j=2}v_p^{(i)}Y\partial_Y u_p^{(j)}\\[5pt]
&-\sum_{k=0}^2\sum_{i+j=2-k, (k,j)\neq (0,2)}\Big(\frac{\partial_r^k\tilde{u}_e^{(i)}(\theta,1)}{k!}Y^k \partial_\theta \tilde{u}_p^{(j)}+\tilde{u}_p^{(j)}\frac{\partial_r^k\partial_\theta \tilde{u}_e^{(i)}(\theta,1)}{k!}Y^k\Big)+u_e^{(2)}(\theta,1)\partial_\theta u_p^{(0)}\\[5pt]
&-\sum_{k=0}^1\sum_{i+j=2-k}\Big(\frac{\partial_r^k\tilde{v}_e^{(i)}(\theta,1)}{k!}Y^{k+1}\partial_Y \tilde{u}_p^{(j)}+v_p^{(i)}\frac{\partial_r^k(r\partial_r \tilde{u}_e^{(j)})(\theta,1)}{k!}Y^k\Big)\\[5pt]
&-\sum_{k=0}^2\sum_{i+j=3-k, (k,j)\neq (0,2),(0,0)}\frac{\partial_r^k\tilde{v}_e^{(i)}(\theta,1)}{k!}Y^k \partial_Y \tilde{u}_p^{(j)}
\end{align*}}
with $\tilde{u}_p^{(0)}=u_p^{(0)}, \ \tilde{u}_e^{(0)}=u_e(r),\ \tilde{u}_e^{(2)}=u_e^{(2)}, \ \tilde{v}_e^{(2)}=v_e^{(2)}.$

\begin{Proposition}\label{decay estimates of higher order linearized Prandtl} There exists $\eta_0>0$ such that for any $\eta\in(0,\eta_0)$, equations (\ref{second linearized prandtl problem near 1}) have a unique solution $(u_p^{(2)},v_p^{(3)})$ which satisfies
\begin{align}\label{decay behavior-prandtl-2}
&\sum_{j+k\leq m}\int_{-\infty}^0\int_0^{2\pi}\Big|\partial_\theta^j\partial_Y^k (u_p^{(2)}-A_{2\infty},v_p^{(3)} )\Big|^2\big<Y\big>^{2l}d\theta dY\leq C(m,l)\eta^2, \ \ m, l\geq 0, \nonumber\\
&\int_0^{2\pi}v_p^{(3)}(\theta, Y)d\theta=0, \ \forall\ Y\leq 0,
\end{align}
where $A_{2\infty}:=\lim_{Y\rightarrow -\infty}u_p^{(2)}(\theta,Y)$ is a constant which satisfies $|A_{2\infty}|\leq C\eta$.

\end{Proposition}
The proof is same with Proposition \ref{decay estimates of linearized Prandtl} by noticing that $f_2(\theta,Y)$ is decay very fast as $Y\rightarrow -\infty$, we omit the details.

We construct the pressure $p_p^{(3)}(\theta,Y)$ by considering the equation
\begin{align}\label{equation of second pressure}
\partial_Yp_p^{(3)}(\theta, Y)=g_2(\theta,Y), \quad \lim_{Y\rightarrow -\infty}p_p^{(3)}(\theta,Y)=0,
\end{align}
where
{\small\begin{align*}
g_2(\theta,Y)=&\partial_{YY}v_p^{(2)}+Y\partial_{YY}v_p^{(1)}+\partial_Yv_p^{(1)}-2\partial_\theta u_p^{(0)}-Y\partial_Yp_p^{(2)}-\sum_{i+j=3} v_p^{(i)}\partial_Yv_p^{(j)}\\[5pt]
&-\sum_{i+j=2}\Big(\tilde{u}_p^{(i)}\partial_\theta v_p^{(j)}+v_p^{(i)}Y\partial_Yv_p^{(j)}-\tilde{u}_p^{(i)}\tilde{u}_p^{(j)}+\tilde{v}_e^{(i)}(\theta,1)Y\partial_Yv_p^{(j)}
+v_p^{(i)}\partial_r\tilde{v}_e^{(j)}(\theta,1)\Big)\\[5pt]
&-\sum_{k=0}^1\sum_{i+j=2-k}\Big(\frac{\partial_r^k\tilde{u}_e^{(i)}(\theta,1)}{k!}Y^k\partial_\theta v_p^{(j)}+\frac{\partial_r^k\partial_\theta \tilde{v}_e^{(j)}(\theta,1)}{k!}Y^k\partial_\theta \tilde{u}_p^{(i)}\Big)\\[5pt]
&-\sum_{k=0}^2\sum_{i+j=2-k}\Big(\frac{\partial_r^k\tilde{u}_e^{(i)}(\theta,1)}{k!}Y^k\tilde{u}_p^{(j)}+\frac{\partial_r^k \tilde{u}_e^{(j)}(\theta,1)}{k!}Y^k \tilde{u}_p^{(i)}\Big),
\end{align*}}
where $\tilde{u}_p^{(0)}=u_p^{(0)},  \ \tilde{u}_e^{(0)}=u_e(r),\ \tilde{u}_e^{(2)}=u_e^{(2)}+A_{2\infty}, \ \tilde{v}_e^{(2)}=v_e^{(2)}$, and here and below $\tilde{u}_p^{(2)}=u_p^{(2)}-A_{2\infty}$.  $g_2(\theta,Y)$ can be derived by the same argument as $g_1(\theta,Y)$.
Moreover, notice that $g_2(\theta,Y)$ decay fast as $Y\rightarrow -\infty$, we can obtain $p_p^{(2)}$ by solving (\ref{equation of second pressure}) and $p_p^{(2)}$ also decay fast as $Y\rightarrow -\infty$.

\subsubsection{Linearized Euler equations for $(u_e^{(3)},v_e^{(3)},p_e^{(3)})$ and their solvabilities}
\indent

Let $\phi_2(s)=-A_{2\infty}\big(r\chi''(r)+\chi'(r)-\frac{\chi(r)}{r}\big)$ and
\begin{align*}
A_2(r):=a_2r+r\int_0^r\frac{\phi_2(s)}{2s}-\frac{1}{r}\int_0^r\frac{s\phi_2(s)}{2}ds,
\end{align*}
where $a_2$ is a constant such that $A_2(1)=0$. Obviously, $|a_2|\leq C\eta$,
\begin{align}\label{corrector of second order Euler equation}
\left\{
\begin{array}{ll}
rA''_2(r)+A'_2(r)-\frac{A_2(r)}{r}=-\phi_2(r), \ 0<r<1\\[5pt]
A_2(1)=0,
 \end{array}
 \right.
\end{align}
and $\|\partial_r^kA_2(r)\|_\infty\leq C(k)\eta.$ Moreover, notice that $\chi(r)=0$ for $r\leq \frac12$, we deduce that $A_2(r)=a_2r$ for $r\leq \frac12$.

Set
\begin{align*}
\tilde{u}_e^{(2)}(\theta,r):&=u_e^{(2)}(\theta,r)+\chi(r)A_{2\infty}+A_2(r), \\[5pt]
 \tilde{v}_e^{(2)}(\theta,r):&=v_e^{(2)}(\theta,r),\\
\tilde{p}_e^{(2)}(\theta,r):&=p_e^{(2)}(\theta,r)+2a\int_{0}^r[\chi(s)A_{2\infty}+A_2(s)]ds,
\end{align*}
then $(\tilde{u}_e^{(2)},\tilde{v}_e^{(2)},\tilde{p}_e^{(2)})$ also satisfies the linearized Euler equations (\ref{outer-2 order equation}) with the boundary condition (\ref{boundary condition of second Euler}). Moreover, there holds
\begin{align}\label{Estimate of modified second linearized Euler equation}
\begin{aligned}
|\partial_\theta \tilde{u}_e^{(2)}+\tilde{v}_e^{(2)}|(\theta,r)\leq& C\eta r, \ |\partial_\theta \tilde{v}_e^{(2)}-\tilde{u}_e^{(2)}|(\theta,r)\leq C\eta r,  \ \forall (\theta, r)\in \Omega,\\[5pt]
 \|\partial^k_\theta\partial^j_r(\tilde{u}_e^{(2)},\tilde{v}_e^{(2)})\|_2\leq& C(k,j)\eta, \quad \forall j,k\geq 0;\\
 r^2\triangle \tilde{u}_e^{(2)}-\tilde{u}_e^{(2)}+2\partial_\theta &\tilde{v}_e^{(2)}=0, \quad \int_0^{2\pi}\tilde{v}_e^{(2)}d\theta=0.
 \end{aligned}
\end{align}
Putting
\begin{align*}
&u^{\varepsilon}(\theta,r)=u_e(r)+\sum_{i=1}^2\varepsilon^{i} \tilde{u}_e^{(i)}(\theta,r)+\varepsilon^3u_e^{(3)}+\cdots,\\[5pt]  &v^{\varepsilon}(\theta,r)=\sum_{i=1}^2\varepsilon^{i} \tilde{v}_e^{(i)}(\theta,r)+\varepsilon^3v_e^{(3)}+\cdots, \\[5pt]
&p^{\varepsilon}(\theta,r)=p_e(r)+\sum_{i=1}^2\varepsilon^{i} \tilde{p}_e^{(i)}(\theta,r)+\varepsilon^3p_e^{(3)}+\cdots
\end{align*}
into the Navier-Stokes equations (\ref{NS-curvilnear}), we find that $(u_e^{(3)},v_e^{(3)},p_e^{(3)})$ satisfies the following linearized Euler equations in $\Omega$
\begin{eqnarray}
\left \{
\begin {array}{ll}
ar \partial_\theta u_e^{(3)}+2arv_e^{(3)}+\partial_\theta p_e^{(3)}+f_e(\theta,r)=0,\\[5pt]
ar \partial_\theta v_e^{(3)}-2aru_e^{(3)}+r\partial_rp_e^{(3)}+g_e(\theta,r)=0,\\[7pt]
\partial_\theta u_e^{(3)}+r\partial_rv_e^{(3)}+ v_e^{(3)}=0,\label{outer-3 order Euler equation}
\end{array}
\right.
\end{eqnarray}
and equipped with the boundary conditions
\begin{align}\label{boundary condition of third Euler}
rv_e^{(3)}|_{r=1}=-v_p^{(3)}|_{Y=0}, \quad v_e^{(3)}(\theta,r)=v_e^{(3)}(\theta+2\pi,r),
\end{align}
where
\begin{align*}
f_e(\theta,r)=&\tilde{u}_e^{(1)}\partial_\theta \tilde{u}_e^{(2)}+\tilde{u}_e^{(2)}\partial_\theta \tilde{u}_e^{(1)}+\tilde{v}_e^{(1)}r\partial_r\tilde{u}_e^{(2)}+\tilde{v}_e^{(2)}r\partial_r\tilde{u}_e^{(1)}
+\tilde{u}_e^{(1)}\tilde{v}_e^{(2)}+\tilde{u}_e^{(2)}\tilde{v}_e^{(1)}\\
&-\underbrace{\Big(\frac{\partial_{\theta\theta}\tilde{u}_e^{(1)}}{r}+r\partial_{rr}\tilde{u}_e^{(1)}+\partial_r\tilde{u}_e^{(1)}
+\frac{2}{r}\partial_\theta \tilde{v}_e^{(1)}-\frac{\tilde{u}_e^{(1)}}{r}\Big)}_{I_1},\\[5pt]
g_e(\theta,r)=&\tilde{u}_e^{(1)}\partial_\theta \tilde{v}_e^{(2)}+\tilde{u}_e^{(2)}\partial_\theta \tilde{v}_e^{(1)}+\tilde{v}_e^{(1)}r\partial_r \tilde{v}_e^{(2)}+\tilde{v}_e^{(2)}r\partial_r \tilde{v}_e^{(1)}-2\tilde{u}_e^{(1)}\tilde{u}_e^{(2)}\\[5pt]
&-\underbrace{\Big(\frac{\partial_{\theta\theta}\tilde{v}_e^{(1)}}{r}+r\partial_{rr}\tilde{v}_e^{(1)}+\partial_r\tilde{v}_e^{(1)}
-\frac{2}{r}\partial_\theta \tilde{u}_e^{(1)}-\frac{\tilde{v}_e^{(1)}}{r}\Big)}_{I_2}.
\end{align*}
We claim that $I_1=I_2=0$. In fact,
\begin{align*}
I_1=\frac{1}{r}[r^2\triangle \tilde{u}_e^{(1)} -\tilde{u}_e^{(1)}+2\partial_\theta \tilde{v}_e^{(1)} ]=0, \ \ I_2=\triangle(rv_e^{(1)})-\frac{2(\partial_\theta u^{(1)}_e+\partial_r(rv^{(1)}_e))}{r}=0,
\end{align*}
where we used (\ref{Estimate of modified first linearized Euler equation}).
\begin{Proposition}\label{solvability of third linearized Euler equation}
The linearized Euler equations  (\ref{outer-3 order Euler equation}) have a solution $(u_e^{(3)}, v_e^{(3)}, p_e^{(3)})$ which satisfies
\begin{align}\label{Estimate of third linearized Euler equation}
|\partial_\theta u_e^{(3)}+v_e^{(3)}|(\theta,r)\leq& C\eta r, \ |\partial_\theta v_e^{(3)}-u_e^{(3)}|(\theta,r)\leq C\eta r,  \ \forall (\theta, r)\in \Omega,\nonumber\\[5pt]
\|\partial^k_\theta\partial^j_r(u_e^{(3)},v_e^{(3)})\|_2\leq& C(j,k)\eta,  \quad j,k\geq 0;\\
r^2\triangle u_e^{(3)} -u_e^{(3)}+2\partial_\theta &v_e^{(3)}=0, \quad\int_0^{2\pi}v^{(3)}_ed\theta=0.\nonumber
\end{align}
\end{Proposition}
\begin{proof}
Using the same argument as Proposition \ref{solvability of second Euler equation}, we deduce that
\begin{align*}
\left\{\begin{array}{lll}
-ar^2\triangle(rv_e^{(3)})+v_e^{(2)}(r^2\triangle \tilde{u}_e^{(1)}-\tilde{u}_e^{(1)}+2\partial_\theta \tilde{v}_e^{(1)})+v_e^{(1)}(r^2\triangle \tilde{u}_e^{(2)}-\tilde{u}_e^{(2)}+2\partial_\theta \tilde{v}_e^{(2)})=0,\\[5pt]
rv_e^{(3)}|_{r=1}=-v_p^{(3)}(\theta,0).
\end{array}
\right.
\end{align*}
where we used $\triangle(rv_e^{(1)})=\triangle(rv_e^{(2)})=0$.
Using (\ref{Estimate of modified first linearized Euler equation}) and (\ref{Estimate of modified second linearized Euler equation}), we obtain the following equation for $rv_e^{(3)}$ in $\Omega$
  \begin{align*}
\left\{\begin{array}{lll}
-ar^2\triangle(rv_e^{(3)})=0,\\[5pt]
rv_e^{(3)}|_{r=1}=-v_p^{(3)}(\theta,0).
\end{array}
\right.
\end{align*}
Thus, we can complete the proof of this proposition by the same argument as Proposition \ref{solvability of first Euler}.
\end{proof}

\subsubsection{Linearize Prandtl equations for $(u_p^{(3)},v_p^{(4)})$ and their solvabilities}

\indent

Let
\begin{align*}
&u^\varepsilon(\theta,r)=u_e(r)+u_p^{(0)}(\theta,Y)+\sum_{i=1}^2\varepsilon^i\big[\tilde{u}_e^{(i)}(\theta,r)+\tilde{u}_p^{(i)}(\theta,Y)\big]
+\varepsilon^3[u_e^{(3)}(\theta,r)+ u_p^{(3)}(\theta,Y)\big]+\cdots,\\[5pt]
&v^\varepsilon(\theta,r)=\sum_{i=1}^2\varepsilon^i\big[\tilde{v}_e^{(i)}(\theta,r)+v_p^{(i)}(\theta,Y)\big]
+\varepsilon^3\big[v_e^{(i)}(\theta,r)+v_p^{(i)}(\theta,Y)\big]
+\varepsilon^4[v_e^{(4)}+v_p^{(4)}(\theta,Y)]+\cdots,\\[5pt]
&p^\varepsilon(\theta,r)=p_e(r)+\sum_{i=1}^2\varepsilon^i\big[\tilde{p}_e^{(i)}(\theta,r)+p_p^{(i)}(\theta,Y)\big]+
\varepsilon^3\big[p_e^{(3)}(\theta,r)+p_p^{(3)}(\theta,Y)\big]
+\varepsilon^4p_p^{(4)}(\theta,Y)+\cdots
\end{align*}
with the following boundary conditions
\begin{align*}
 u_e^{(3)}(\theta,1)+u_p^{(3)}(\theta,0)=0, \ v_e^{(4)}(\theta,1)+v_p^{(4)}(\theta,0)=0, \ \lim_{Y\rightarrow \infty}(\partial_Yu_p^{(3)},v_p^{(4)})=(0,0).
\end{align*}
As the derivation of equations for $(u_p^{(2)},v_p^{(3)})$, we obtain the following linearized Prandtl problem for $(u_p^{(3)},v_p^{(4)})$
\begin{eqnarray}
\left \{
\begin {array}{ll}
\big(u_e(1)+u_p^{(0)}\big)\partial_\theta u_p^{(3)}+\big(v_e^{(1)}(\theta,1)+ v_p^{(1)}\big)\partial_Yu_p^{(3)}+(u_p^{(3)}+u_e^{(3)}(\theta,1))\partial_\theta u_p^{(0)} \\[5pt]
\quad \quad \quad \quad (v_p^{(4)}+v_e^{(4)}(\theta,1))\partial_{Y}u_p^{(0)} -\partial_{YY}u_p^{(3)}=f_3(\theta,Y)\\[5pt]
\partial_\theta u_p^{(3)}+\partial_Yv_p^{(4)}+\partial_Y(Yv_p^{(3)})=0,\\[5pt]
u_p^{(3)}(\theta,Y)=u_p^{(3)}(\theta+2\pi,Y),\quad v_p^{(4)}(\theta,Y)=v_p^{(4)}(\theta+2\pi,Y),\\[5pt]
u_p^{(3)}\big|_{Y=0}=-u_e^{(3)}\big|_{r=1},\ \ \lim_{Y\rightarrow \infty}(\partial_Yu_p^{(3)},v_p^{(4)})=(0,0)
\label{third linearized prandtl problem near 1}
\end{array}
\right.
\end{eqnarray}
where
{\small \begin{align*}
f_3(\theta,Y)=&-\partial_\theta p_p^{(3)}+Y\partial_{YY}u_p^{(3)}+\partial_Yu_p^{(3)}+\sum_{k=0}^1(-1)^k\frac{Y^k\partial_{\theta\theta}\tilde{u}_p^{(1-k)}}{k!}+2\partial_\theta v_p^{(2)}-2Y\partial_\theta v_p^{(1)}\\[5pt]
&-\sum_{k=0}^1(-1)^k\frac{Y^k\tilde{u}_p^{(1-k)}}{k!}-\sum_{i+j=3, 1\leq i\leq 2}\tilde{u}_p^{(i)}\partial_\theta\tilde{u}_p^{(j)}-\sum_{i+j=3 }[v_p^{(i)}Y\partial_Y\tilde{u}_p^{(j)}+\tilde{u}_p^{(i)}v_p^{(j)}]\\[5pt]
&-\sum_{k=0}^3\sum_{i+j=3-k, (k,j)\neq (0,3)}\Big(\frac{\partial_r^k\tilde{u}_e^{(i)}(\theta,1)}{k!}Y^k\partial_\theta\tilde{u}_p^{(j)}
+\tilde{u}_p^{(i)}\frac{\partial_r^k\partial_\theta\tilde{u}_e^{(j)}(\theta,1)}{k!}Y^k\Big)\\[5pt]
&-\sum_{k=0}^3\sum_{i+j=3-k}\Big(\frac{\partial_r^k\tilde{v}_e^{(i)}(\theta,1)}{k!}Y^{k+1}\partial_Y\tilde{u}_p^{(j)}
+v_p^{(i)}\frac{\partial_r^k(r\partial_r\tilde{u}_e^{(j)})(\theta,1)}{k!}Y^k\Big)\\[5pt]
&-\sum_{k=0}^3\sum_{i+j=4-k, (k,j)\neq (0,3),(0,0),i\leq 3}\frac{\partial_r^k\tilde{v}_e^{(i)}(\theta,1)}{k!}Y^k\partial_Y\tilde{u}_p^{(j)}\\[5pt]
&-\sum_{k=0}^3\sum_{i+j=3-k}\Big(\frac{\partial_r^k\tilde{u}_e^{(i)}(\theta,1)}{k!}Y^{k}v_p^{(j)}
+\tilde{u}_p^{(i)}\frac{\partial_r^k\tilde{v}_e^{(j)}(\theta,1)}{k!}Y^k\Big)
\end{align*}}
with $\tilde{u}_p^{(0)}=u_p^{(0)},\ \tilde{u}^{(0)}_e=u_e(r), \ (\tilde{u}^{(3)}_e,\tilde{v}^{(3)}_e)=(u_e^{(3)},v_e^{(3)}).$

\begin{Proposition} There exists $\eta_0>0$ such that for any $\eta\in(0,\eta_0)$, equations (\ref{third linearized prandtl problem near 1}) have a unique solution $(u_p^{(3)},v_p^{(4)})$ which satisfies
\begin{align}\label{decay behavior-prandtl-3}
&\sum_{j+k\leq m}\int_{-\infty}^0\int_0^{2\pi}\big|\partial_\theta^j\partial_Y^k (u_p^{(3)}-A_{3\infty},v_p^{(4)} )\big|^2\big<Y\big>^{2l}d\theta dY\leq C(m,l)\eta^2, \ \ m, l\geq 0, \nonumber\\
& \int_0^{2\pi}v_p^{(3)}(\theta,Y)d\theta=0, \ \forall \ Y\leq 0,
\end{align}
where $A_{3\infty}:=\lim_{Y\rightarrow -\infty} u_p^{(3)}$ is a constant which satisfies $|A_{3\infty}|\leq C\eta$.
\end{Proposition}
The proof is same with Proposition \ref{decay estimates of linearized Prandtl}, we omit the details.
Moreover, we can construct $p_p^{(4)}(\theta,Y)$ by solving the equation
\begin{align*}
\partial_Yp_p^{(4)}(\theta, Y)=g_3(\theta,Y), \quad \lim_{Y\rightarrow -\infty}p_p^{(4)}(\theta,Y)=0,
\end{align*}
where
{\small \begin{align*}
g_3(\theta,Y)=&\partial_{YY}v_p^{(3)}+Y\partial_{YY}v_p^{(2)}+\partial_{Y}v_p^{(2)}
+\partial_{\theta\theta}v_p^{(1)}-\sum_{k=0}^1(-1)^k\frac{Y^k}{k!}\partial_\theta \tilde{u}_p^{(2-k)}-v_p^{(1)}-Y\partial_Yp_p^{(3)}\\[5pt]
&-\sum_{i+j=3}[\tilde{u}_p^{(i)}\partial_\theta v_p^{(j)}+v_p^{(i)}Y\partial_Y v_p^{(j)}-\tilde{u}_p^{(i)}\tilde{u}_p^{(j)}]-\sum_{i+j=4} v_p^{(i)}Y\partial_Y v_p^{(j)}-\sum_{k=0}^2\sum_{i+j=4-k}\frac{\partial_r^k\tilde{v}_e^{(i)}(\theta,1)}{k!}Y^{k}\partial_Y v_p^{(j)}\\[5pt]
&-\sum_{k=0}^3\sum_{i+j=3-k}\Big(\frac{\partial_r^k\tilde{u}_e^{(i)}(\theta,1)}{k!}Y^{k}\partial_\theta v_p^{(j)}
+\tilde{u}_p^{(i)}\frac{\partial_r^k\partial_\theta \tilde{v}_e^{(j)}(\theta,1)}{k!}Y^k\Big)\\[5pt]
&+\sum_{k=0}^3\sum_{i+j=3-k}\Big(\frac{\partial_r^k\tilde{u}_e^{(i)}(\theta,1)}{k!}Y^{k}\tilde{u}_p^{(j)}
+\tilde{u}_p^{(i)}\frac{\partial_r^k\tilde{u}_e^{(j)}(\theta,1)}{k!}Y^k\Big)\\[5pt]
&-\sum_{k=0}^1\sum_{i+j=3-k}\Big(\frac{\partial_r^k\tilde{v}_e^{(i)}(\theta,1)}{k!}Y^{k+1}\partial_Y v_p^{(j)}
+v_p^{(i)}\frac{\partial_r^k(r\partial_r \tilde{v}_e^{(j)})(\theta,1)}{k!}Y^k\Big)
\end{align*}}
with $\tilde{u}_p^{(0)}=u_p^{(0)},\ \tilde{u}^{(0)}_e=u_e(r), \ (\tilde{u}^{(3)}_e,\tilde{v}^{(3)}_e)=(u_e^{(3)}+A_{3\infty},v_e^{(3)})$, and here and below $\tilde{u}_p^{(3)}=u_p^{(3)}-A_{3\infty}.$
$g_3(\theta,Y)$ was obtained by the same argument as $g_1(\theta,Y)$. Moreover, due to $g_3(\theta,Y)$ decay fast as $Y\rightarrow -\infty$, we can obtain $p_p^{(4)}$ and deduce that $p_p^{(4)}$ also decay fast as $Y\rightarrow -\infty$.

\subsubsection{Linearized Euler equations for $(u_e^{(4)},v_e^{(4)},p_e^{(4)})$ and their solvabilities}
\indent

Let $\phi_3(s)=-A_{3\infty}\big(r\chi''(r)+r\chi'(r)-\frac{\chi(r)}{r}\big)$ and
\begin{align*}
A_3(r):=a_3r+r\int_0^r\frac{\phi_3(s)}{2s}-\frac{1}{r}\int_0^r\frac{s\phi_3(s)}{2}ds,
\end{align*}
where $a_3$ is a constant such that $A_3(1)=0$. Obviously, $|a_3|\leq C\eta$,
\begin{align}\label{corrector of second order Euler equation}
\left\{
\begin{array}{ll}
rA''_3(r)+A'_3(r)-\frac{A_3(r)}{r}=-\phi_3(r), \ 0<r\leq1 \\[5pt]
A_3(1)=0,
 \end{array}
 \right.
\end{align}
and $\|\partial_r^kA_3(r)\|_\infty\leq C(k)\eta.$ Moreover, notice that $\chi(r)=0$ for $r\leq \frac12$, we deduce that $A_3(r)=a_3r$ for $r\leq \frac12$.

Set
\begin{align*}
\tilde{u}_e^{(3)}(\theta,r):&=u_e^{(3)}(\theta,r)+\chi(r)A_{3\infty}+A_3(r), \\[5pt]
 \tilde{v}_e^{(3)}(\theta,r):&=v_e^{(3)}(\theta,r),\\
\tilde{p}_e^{(3)}(\theta,r):&=p_e^{(3)}(\theta,r)+2a\int_{0}^r[\chi(s)A_{3\infty}+A_3(s)]ds,
\end{align*}
then $(\tilde{u}_e^{(3)},\tilde{v}_e^{(3)},\tilde{p}_e^{(3)})$ also satisfies the linearized Euler equations (\ref{outer-3 order Euler equation}) with the boundary condition (\ref{boundary condition of third Euler}). Moreover, there holds
\begin{align}\label{Estimate of modified third linearized Euler equation}
|\partial_\theta \tilde{u}_e^{(3)}+\tilde{v}_e^{(3)}|(\theta,r)\leq& C\eta r, \ |\partial_\theta \tilde{v}_e^{(3)}-\tilde{u}_e^{(3)}|(\theta,r)\leq C\eta r,  \ \forall (\theta, r)\in \Omega,\nonumber\\[5pt]
 \|\partial^k_\theta\partial^j_r(\tilde{u}_e^{(3)},\tilde{v}_e^{(3)})\|_2\leq& C(k,j)\eta, \quad \forall j,k\geq 0;\\
 r^2\triangle \tilde{u}_e^{(3)} -\tilde{u}_e^{(3)}+2\partial_\theta &\tilde{v}_e^{(3)}=0,\quad \int_{0}^{2\pi}\tilde{v}^{(3)}_ed\theta=0.\nonumber
\end{align}
Putting
\begin{align*}
&u^{\varepsilon}(\theta,r)=u_e(r)+\sum_{i=1}^4\varepsilon^{i} \tilde{u}_e^{(i)}(\theta,r)+\cdots,\\[5pt]  &v^{\varepsilon}(\theta,r)=\sum_{i=1}^4\varepsilon^{i} \tilde{v}_e^{(i)}(\theta,r)+\cdots, \\[5pt]
&p^{\varepsilon}(\theta,r)=p_e(r)+\sum_{i=1}^4\varepsilon^{i} \tilde{p}_e^{(i)}(\theta,r)+\cdots
\end{align*}
into the Navier-Stokes equations (\ref{NS-curvilnear}), we find that $(\tilde{u}_e^{(4)},\tilde{v}_e^{(4)},\tilde{p}_e^{(4)})$ satisfies the following linearized Euler equations in $\Omega$
\begin{eqnarray}
\left \{
\begin {array}{ll}
ar \partial_\theta \tilde{u}_e^{(4)}+2ar\tilde{v}_e^{(4)}+\partial_\theta \tilde{p}_e^{(4)}+f_{4e}(\theta,r)=0,\\[5pt]
ar \partial_\theta \tilde{v}_e^{(4)}-2ar\tilde{u}_e^{(4)}+r\partial_rp_e^{(4)}+g_{4e}(\theta,r)=0,\\[7pt]
\partial_\theta \tilde{u}_e^{(4)}+r\partial_r\tilde{v}_e^{(4)}+ \tilde{v}_e^{(4)}=0,\label{outer-4 order equation}
\end{array}
\right.
\end{eqnarray}
and equipped with the  boundary conditions
\begin{align}
r\tilde{v}_e^{(4)}|_{r=1}=-v_p^{(4)}|_{Y=0}, \quad \tilde{v}_e^{(4)}(\theta,r)=\tilde{v}_e^{(4)}(\theta+2\pi,r),\nonumber
\end{align}
where
\begin{align*}
f_{4e}(\theta,r)=&\tilde{u}_e^{(1)}\partial_\theta \tilde{u}_e^{(3)}+\tilde{u}_e^{(3)}\partial_\theta \tilde{u}_e^{(1)}+\tilde{v}_e^{(1)}r\partial_r\tilde{u}_e^{(3)}+\tilde{v}_e^{(3)}r\partial_r\tilde{u}_e^{(1)}
+\tilde{u}_e^{(1)}\tilde{v}_e^{(3)}+\tilde{u}_e^{(3)}\tilde{v}_e^{(1)}\\[5pt]
&+\tilde{u}_e^{(2)}\partial_\theta \tilde{u}_e^{(2)}+\tilde{v}_e^{(2)}r\partial_r \tilde{u}_e^{(2)}+\tilde{u}_e^{(2)}\tilde{v}_e^{(2)}\\
&-\underbrace{\Big(\frac{\partial_{\theta\theta}\tilde{u}_e^{(2)}}{r}+r\partial_{rr}\tilde{u}_e^{(2)}+\partial_r\tilde{u}_e^{(2)}
+\frac{2}{r}\partial_\theta \tilde{v}_e^{(2)}-\frac{\tilde{u}_e^{(2)}}{r}\Big)}_{=0},\\[5pt]
g_{4e}(\theta,r)=&\tilde{u}_e^{(1)}\partial_\theta \tilde{v}_e^{(3)}+\tilde{u}_e^{(3)}\partial_\theta \tilde{v}_e^{(1)}+\tilde{v}_e^{(1)}r\partial_r \tilde{v}_e^{(3)}+\tilde{v}_e^{(3)}r\partial_r \tilde{v}_e^{(1)}-2\tilde{u}_e^{(1)}\tilde{u}_e^{(3)}\\[5pt]
&+\tilde{u}_e^{(2)}\partial_\theta \tilde{v}_e^{(2)}+\tilde{v}_e^{(2)}r\partial_r \tilde{v}_e^{(2)}-(\tilde{u}_e^{(2)})^2\\
&-\underbrace{\Big(\frac{\partial_{\theta\theta}\tilde{v}_e^{(2)}}{r}+r\partial_{rr}\tilde{v}_e^{(2)}+\partial_r\tilde{v}_e^{(2)}
-\frac{2}{r}\partial_\theta \tilde{u}_e^{(2)}-\frac{\tilde{v}_e^{(2)}}{r}\Big)}_{=0}.
\end{align*}

\begin{Proposition}
The linearized Euler equations  (\ref{outer-4 order equation}) have a solution $(\tilde{u}_e^{(4)}, \tilde{v}_e^{(4)}, \tilde{p}_e^{(4)})$ which satisfies
\begin{align}\label{Estimate of fourth linearized Euler equation}
|\partial_\theta \tilde{u}_e^{(4)}+\tilde{v}_e^{(4)}|(\theta,r)\leq& C\eta r, \ |\partial_\theta \tilde{v}_e^{(4)}-\tilde{u}_e^{(4)}|(\theta,r)\leq C\eta r,  \ \forall (\theta, r)\in \Omega,\nonumber\\[5pt]
 \|\partial^k_\theta\partial^j_r(\tilde{u}_e^{(4)},\tilde{v}_e^{(4)})\|_2\leq& C(k,j)\eta, \quad \forall j,k\geq 0; \\
 r^2\triangle \tilde{u}_e^{(4)} -\tilde{u}_e^{(4)}+2\partial_\theta &\tilde{v}_e^{(4)}=0,\quad\int_{0}^{2\pi} \tilde{v}_e^{(4)}d\theta=0.\nonumber
\end{align}
\end{Proposition}
\begin{proof}
The proof is same as Proposition \ref{solvability of third linearized Euler equation}, we omit the details.
\end{proof}

\subsubsection{Linearize Prandtl equations for $(u_p^{(4)},v_p^{(5)})$ and their solvabilities}

\indent

Let
\begin{align*}
&u^\varepsilon(\theta,r)=u_e(r)+u_p^{(0)}(\theta,Y)+\sum_{i=1}^3\varepsilon^i\big[\tilde{u}_e^{(i)}(\theta,r)+\tilde{u}_p^{(i)}(\theta,Y)\big]
+\varepsilon^4[\tilde{u}_e^{(4)}(\theta,r)+ u_p^{(4)}(\theta,Y)\big]+\cdots,\\[5pt]
&v^\varepsilon(\theta,r)=\sum_{i=1}^4\varepsilon^i\big[\tilde{v}_e^{(i)}(\theta,r)+v_p^{(i)}(\theta,Y)\big]
+\varepsilon^5v_p^{(5)}(\theta,Y)+\cdots,\\[5pt]
&p^\varepsilon(\theta,r)=p_e(r)+\sum_{i=1}^4\varepsilon^i\big[\tilde{p}_e^{(i)}(\theta,r)+p_p^{(i)}(\theta,Y)\big]
+\varepsilon^5p_p^{(5)}(\theta,Y)+\cdots
\end{align*}
with the boundary conditions
\begin{align*}
 \tilde{u}_e^{(4)}(\theta,1)+u_p^{(4)}(\theta,0)=0, \ \lim_{Y\rightarrow \infty}\partial_Yu_p^{(4)}(\theta,Y)=v_p^{(5)}(\theta,0)=0, \ p_p^{(5)}(\theta, 0)=0.
\end{align*}
As the derivation for equations of $(u_p^{(2)},v_p^{(3)})$, we obtain the following linearized Prandtl problem for $(u_p^{(4)},v_p^{(5)})$
\begin{eqnarray}
\left \{
\begin {array}{ll}
\big(u_e(1)+u_p^{(0)}\big)\partial_\theta u_p^{(4)}+\big(v_e^{(1)}(\theta,1)+ v_p^{(1)}\big)\partial_Yu_p^{(4)}+u_p^{(4)}\partial_\theta u_p^{(0)} \\[5pt]
\quad \quad \quad \quad v_p^{(5)}\partial_{Y}u_p^{(0)} -\partial_{YY}u_p^{(4)}=f_4(\theta,Y)\\[5pt]
\partial_\theta u_p^{(4)}+\partial_Yv_p^{(5)}+\partial_Y(Yv_p^{(4)})=0,\\[5pt]
u_p^{(4)}(\theta,Y)=u_p^{(4)}(\theta+2\pi,Y),\quad v_p^{(5)}(\theta,Y)=v_p^{(5)}(\theta+2\pi,Y),\\[5pt]
u_p^{(4)}\big|_{Y=0}=-\tilde{u}_e^{(4)}\big|_{r=1},\ \ \lim_{Y\rightarrow \infty}\partial_Yu_p^{(4)}(\theta,Y)=v_p^{(5)}(\theta,0)=0
\label{fourth linearized prandtl problem near 1}
\end{array}
\right.
\end{eqnarray}
and the pressure $p_p^{(5)}$ satisfies
\begin{align}\label{equation of fifth pressure}
\partial_Yp_p^{(5)}(\theta, Y)=g_4(\theta,Y), \quad p_p^{(5)}(\theta,0)=0,
\end{align}
where
{\small \begin{align*}
f_4(\theta,Y)=&-\partial_\theta p_p^{(4)}+Y\partial_{YY}u_p^{(3)}+\partial_Yu_p^{(3)}+\sum_{k=0}^2(-1)^k\frac{Y^k\partial_{\theta\theta}\tilde{u}_p^{(2-k)}}{k!}+2\partial_\theta v_p^{(2)}-2Y\partial_\theta v_p^{(1)}\\[5pt]
&-\sum_{k=0}^2(-1)^k\frac{Y^k\tilde{u}_p^{(2-k)}}{k!}-\sum_{i+j=4, 1\leq i\leq 3}\tilde{u}_p^{(i)}\partial_\theta\tilde{u}_p^{(j)}-\sum_{i+j=4 }[v_p^{(i)}Y\partial_Y\tilde{u}_p^{(j)}+\tilde{u}_p^{(i)}v_p^{(j)}]\\[5pt]
&-\sum_{k=0}^4\sum_{i+j=4-k, (k,j)\neq (0,4)}\Big(\frac{\partial_r^k\tilde{u}_e^{(i)}(\theta,1)}{k!}Y^k\partial_\theta\tilde{u}_p^{(j)}
+\tilde{u}_p^{(i)}\frac{\partial_r^k\partial_\theta\tilde{u}_e^{(j)}(\theta,1)}{k!}Y^k\Big)\\[5pt]
&-\sum_{k=0}^4\sum_{i+j=4-k}\Big(\frac{\partial_r^k\tilde{v}_e^{(i)}(\theta,1)}{k!}Y^{k+1}\partial_Y\tilde{u}_p^{(j)}
+v_p^{(i)}\frac{\partial_r^k(r\partial_r\tilde{u}_e^{(j)})(\theta,1)}{k!}Y^k\Big)\\[5pt]
&-\sum_{k=0}^4\sum_{i+j=5-k, (k,j)\neq (0,4),i\leq 4}\frac{\partial_r^k\tilde{v}_e^{(i)}(\theta,1)}{k!}Y^k\partial_Y\tilde{u}_p^{(j)}\\[5pt]
&-\sum_{k=0}^4\sum_{i+j=4-k}\Big(\frac{\partial_r^k\tilde{u}_e^{(i)}(\theta,1)}{k!}Y^{k}v_p^{(j)}
+\tilde{u}_p^{(i)}\frac{\partial_r^k\tilde{v}_e^{(j)}(\theta,1)}{k!}Y^k\Big)
\end{align*}
and
\begin{align*}
g_4(\theta,Y)=&\partial_{YY}v_p^{(4)}+Y\partial_{YY}v_p^{(3)}+\partial_{Y}v_p^{(3)}
+[\partial_{\theta\theta}v_p^{(2)}-Y\partial_{\theta\theta}v_p^{(1)}]-\sum_{k=0}^2(-1)^k\frac{Y^k}{k!}\partial_\theta \tilde{u}_p^{(2-k)}-[v_p^{(2)}-Yv_p^{(1)}]\\[5pt]
&-Y\partial_Yp_p^{(4)}-\sum_{i+j=4}[\tilde{u}_p^{(i)}\partial_\theta v_p^{(j)}+v_p^{(i)}Y\partial_Y v_p^{(j)}-\tilde{u}_p^{(i)}\tilde{u}_p^{(j)}]-\sum_{i+j=5} v_p^{(i)}Y\partial_Y v_p^{(j)}\\[5pt]
&-\sum_{k=0}^4\sum_{i+j=4-k}\Big(\frac{\partial_r^k\tilde{u}_e^{(i)}(\theta,1)}{k!}Y^{k}\partial_\theta v_p^{(j)}
+\tilde{u}_p^{(i)}\frac{\partial_r^k\partial_\theta \tilde{v}_e^{(j)}(\theta,1)}{k!}Y^k\Big)\\[5pt]
&+\sum_{k=0}^4\sum_{i+j=4-k}\Big(\frac{\partial_r^k\tilde{u}_e^{(i)}(\theta,1)}{k!}Y^{k}\tilde{u}_p^{(j)}
+\tilde{u}_p^{(i)}\frac{\partial_r^k\tilde{u}_e^{(j)}(\theta,1)}{k!}Y^k\Big)-\sum_{k=0}^3\sum_{i+j=5-k}\frac{\partial_r^k\tilde{v}_e^{(i)}(\theta,1)}{k!}Y^{k}\partial_Y v_p^{(j)}\\[5pt]
&-\sum_{k=0}^2\sum_{i+j=4-k}\Big(\frac{\partial_r^k\tilde{v}_e^{(i)}(\theta,1)}{k!}Y^{k+1}\partial_Y v_p^{(j)}
+v_p^{(i)}\frac{\partial_r^k(r\partial_r \tilde{v}_e^{(j)})(\theta,1)}{k!}Y^k\Big)
\end{align*}}
here $\tilde{u}_p^{(0)}=u_p^{(0)},\ \tilde{u}_p^{(4)}=u_p^{(4)}, \ \tilde{u}^{(0)}_e=u_e(r).$

\begin{Proposition} There exists $\eta_0>0$ such that for any $\eta\in(0,\eta_0)$, equations (\ref{fourth linearized prandtl problem near 1}) have a unique solution $(u_p^{(4)},v_p^{(5)})$ which satisfies
\begin{align}\label{decay behavior-prandtl-4}
&\sum_{0<j+k\leq m}\int_{-\infty}^0\int_0^{2\pi}\big|\partial_\theta^j\partial_Y^k (u_p^{(4)},v_p^{(5)} )\big|^2\big<Y\big>^{2l}d\theta dY\leq C(m,l)\eta^2, \ \ m\geq 1, l\geq 0, \nonumber\\
&\int_0^{2\pi}v_p^{(5)}(\theta,Y)d\theta=0,\ \forall \ Y\leq 0,
\end{align}
and
\begin{align*}
\|(u_p^{(4)},v_p^{(5)})\|_\infty\leq C\eta.
\end{align*}
\end{Proposition}
The proof is same with Proposition \ref{decay estimates of linearized Prandtl}, we omit the details.  Moreover, we can construct $p_p^{(5)}$ by solving the equation (\ref{equation of fifth pressure}).

\subsection{Approximate solutions}

\indent

In this subsection, we construct an approximate solution of Navier-Stokes equations (\ref{NS-curvilnear}).
Set
\begin{align*}
\tilde{u}_p^a(\theta,r):=&\chi(r)\Big(u_p^{(0)}(\theta,Y)+\sum_{i=1}^3\varepsilon^i\tilde{u}_p^{(i)}(\theta,Y)+
\varepsilon^4u_p^{(4)}(\theta,Y)\Big)
:=\chi(r)u_p^a,\\[5pt]
\tilde{v}_p^a(\theta,r):=&\chi(r)\Big(\sum_{i=1}^5\varepsilon^i v_p^{(i)}(\theta,Y)\Big)
:=\chi(r)v_p^a,\\
\tilde{p}_p^a(\theta,r):=&\chi^2(r)\Big(\sum_{i=1}^5\varepsilon^i p_p^{(i)}(\theta,Y)\Big)
:=\chi^2(r)p_p^a,
\end{align*}
and
\begin{align*}
u_e^a(\theta,r):=&u_e(r)+\sum_{i=1}^4\varepsilon^i \tilde{u}_e^{(i)}(\theta,r), \\[5pt]
 v_e^a(\theta,r):=&\sum_{i=1}^4\varepsilon^i \tilde{v}_e^{(i)}(\theta,r),\\[5pt]
 p_e^a(\theta,r):=&p_e(r)+\sum_{i=1}^4\varepsilon^i \tilde{p}_e^{(i)}(\theta,r).
\end{align*}
We construct an approximate solution
\begin{align}\label{approximate solution}
u^a(\theta,r):&=u_e^a(\theta,r)+\tilde{u}_p^a(\theta,r)+\varepsilon^5h(\theta,r),\nonumber\\[5pt]
v^a(\theta,r):&=v_e^a(\theta,r)+\tilde{v}_p^a(\theta,r),\nonumber\\[5pt]
p^a(\theta,r):&=p_e^a(\theta,r)+\tilde{p}_p^a(\theta,r),
\end{align}
where the corrector $h(\theta,r)$ will be given in Appendix B which satisfies
\beno
h(\theta, 1)=0, \ \|\partial_\theta^j\partial_r^kh\|_2\leq C(j,k)\varepsilon^{-k}
\eeno
and makes $(u^a, v^a)$ be divergence-free
\beno
u^a_\theta+rv^a_r+v^a=0.
\eeno
Moreover, $(u^a, v^a)$  satisfies the following boundary condition
\begin{align*}
 u^a(\theta+2\pi,r)&=u^a(\theta,r), \ v^a(\theta+2\pi,r)=v^a(\theta,r),\\[5pt]
u^a(\theta, 1)&=\alpha+\eta f(\theta), \ v^a(\theta,1)=0.
\end{align*}
And $v^a$ satisfies
\begin{align}\label{v^a null 0}
\int_{0}^{2\pi}v^ad\theta=0.
\end{align}

Notice that $\chi(r)=0$ and $A_i(r)=a_ir$ for $r\leq \frac12$, by collecting the estimate (\ref{Estimate of modified first linearized Euler equation}), (\ref{Estimate of modified second linearized Euler equation}), (\ref{Estimate of modified third linearized Euler equation}), (\ref{Estimate of fourth linearized Euler equation}) and (\ref{decay behavior-prandtl}), (\ref{decay behavior-prandtl-1}), (\ref{decay behavior-prandtl-2}), (\ref{decay behavior-prandtl-3}), (\ref{decay behavior-prandtl-4}),  we deduce that
\begin{align}\label{estimate on Euler parts}
\begin{aligned}
&\|u^a_e-ar\|_0+\|(\partial_r(u^a_e-ar),\partial_\theta u^a_e)\|_0\leq C\varepsilon\eta, \ \ \|v^a_e\|_0+\|(\partial_rv^a_e, \partial_\theta v^a_e)\|_0\leq C\varepsilon\eta, \\[5pt]
&|\partial_\theta u_e^a(\theta,r)+v_e^a(\theta,r)|\leq C\varepsilon \eta r, \ \ |\partial_\theta v_e^a(\theta,r)-u_e^a(\theta,r)+ar|\leq C\varepsilon \eta r, \ \forall (\theta, r)\in \Omega
\end{aligned}
\end{align}
and
\begin{align}\label{estimate on Prandtl parts}
\|(Y^j\partial_Y^ku_p^a, Y^j\partial_\theta^ku_p^a)\|_0\leq C \eta, \ \ \|(Y^j\partial_Y^kv_p^a,Y^j\partial_\theta^kv_p^a)\|_0\leq C \varepsilon\eta,  \ \ \forall j\leq 2, k\leq 1.
\end{align}

 Finally, set
 \begin{align*}
 R_u^a:&=u^au^a_\theta+v^aru^a_r+u^av^a+p^a_\theta-\varepsilon^2\Big( ru^a_{rr}
+\frac{u^a_{\theta\theta}}{r}+2\frac{v^a_{\theta}}{r}+u^a_r-\frac{u^a}{r}\Big),\\[5pt]
R_v^a:&=u^av^a_\theta+v^arv^a_r-(u^a)^2+rp^a_r-\varepsilon^2 \Big(rv^a_{rr}+\frac{v^a_{\theta\theta}}{r}-2\frac{u^a_{\theta}}{r}+v^a_r-\frac{v^a}{r}\Big),
 \end{align*}
 then there hold
 \begin{align}\label{remainder estimate}
\Big(\int_0^1\int_0^{2\pi}\frac{(R_u^a)^2(\theta,r)}{r}d\theta dr\Big)^{\frac12}\leq C\varepsilon^5, \ \Big(\int_0^1\int_0^{2\pi}\frac{(R_v^a)^2(\theta,r)}{r}d\theta dr\Big)^{\frac12}\leq C\varepsilon^5
 \end{align}
 and
{\small\begin{align*}
\left\{
\begin{array}{lll}
u^au^a_\theta+v^aru^a_r+u^av^a+p^a_{\theta}-\varepsilon^2\big( ru^a_{rr}
+\frac{u^a_{\theta\theta}}{r}+2\frac{v^a_{\theta}}{r}+u^a_r-\frac{u^a}{r}\big)=R_u^a,\ (\theta,r)\in \Omega\\[5pt]
u^av^a_\theta+v^arv^a_r-(u^a)^2+rp^a_r-\varepsilon^2\big( rv^a_{rr}+\frac{v^a_{\theta\theta}}{r}-2\frac{u^a_{\theta}}{r}+v^a_r-\frac{v^a}{r}\big)=R_v^a,\ (\theta,r)\in \Omega\\[5pt]
 u^a_\theta+(rv^a)_r=0, \ (\theta,r)\in \Omega \\[5pt]
 u^a(\theta+2\pi,r)=u^a(\theta,r), \ v^a(\theta+2\pi,r)=v^a(\theta,r), \ (\theta,r)\in \Omega\\[5pt]
u^a(\theta,1)=\alpha+f(\theta)\eta,\ v^a(\theta,1)=0, \ \theta\in [0,2\pi].
\end{array}
\right.
\end{align*}}

In fact, when $r\leq \frac12$, there hold
\begin{align*}
R_u^a=&u_e^a\partial_\theta u_e^a+v_e^ar\partial_ru_e^a+u^a_ev^a_e+\partial_\theta p_e^a-\varepsilon^2\Big( r\partial_{rr}u^a_e
+\frac{\partial_{\theta\theta}u^a_e}{r}+2\frac{\partial_\theta v^a_e}{r}+\partial_ru^a_e-\frac{u_e^a}{r}\Big)\\
=&\sum_{m=5}^8\varepsilon^m\Big( \sum_{i+j=m, i,j\leq 4} \tilde{u}_e^{(i)}[\partial_\theta \tilde{u}_e^{(j)}+\tilde{v}_e^{(j)}]+\sum_{i+j=m, i,j\leq 4} \tilde{v}_e^{(i)}r\partial_r \tilde{u}_e^{(j)}\Big),\\
R_v^a=&u_e^a\partial_\theta v_e^a+v_e^ar\partial_rv_e^a-(u^a_e)^2+r\partial_r p_e^a-\varepsilon^2\Big( r\partial_{rr}v^a_e
+\frac{\partial_{\theta\theta}v^a_e}{r}+2\frac{\partial_\theta u^a_e}{r}+\partial_rv^a_e-\frac{v_e^a}{r}\Big)\\
=&\sum_{m=5}^8\varepsilon^m\Big( \sum_{i+j=m, i,j\leq 4} \tilde{u}_e^{(i)}[\partial_\theta \tilde{v}_e^{(j)}-\tilde{u}_e^{(j)}]+\sum_{i+j=m, i,j\leq 4} \tilde{v}_e^{(i)}r\partial_r \tilde{v}_e^{(j)}\Big).
\end{align*}
Using (\ref{Estimate of modified first linearized Euler equation}), (\ref{Estimate of modified second linearized Euler equation}), (\ref{Estimate of modified third linearized Euler equation}) and (\ref{Estimate of fourth linearized Euler equation}), we deduce that
\begin{align*}
|R_u^a(\theta,r)|+|R_v^a(\theta,r)|\leq C\varepsilon^5 r, \  r<\frac12.
\end{align*}
Moreover, when $\frac12\leq r <1$, there holds $\|R_u^a\|_2+\|R_v^a\|_2\leq C\varepsilon^5$. Thus, we obtain (\ref{remainder estimate}).

\section{Linear stability estimates of error equations}

\indent
In this section, we derive the error equations and establish linear stability estimates.

\subsection{Error equations}
\indent

Set the error
\begin{align*}
u:=u^\varepsilon-u^a,\ v:=v^\varepsilon-v^a,\ p:=p^\varepsilon-p^a,
\end{align*}
then there hold
{\small \begin{align}
\left\{
\begin{array}{lll}
u^au_\theta+uu^a_\theta+uu_\theta+v^aru_r+vru^a_r+vru_r+v^au+vu^a+vu+p_\theta-\varepsilon^2\big( ru_{rr}
+\frac{u_{\theta\theta}}{r}+2\frac{v_{\theta}}{r}+u_r-\frac{u}{r}\big)=R_u^a,\\[5pt]
u^av_\theta+uv^a_\theta+uv_\theta+v^arv_r+vrv^a_r+vrv_r-(u^2+2u u^a)+rp_r-\varepsilon^2 \big( rv_{rr}+\frac{v_{\theta\theta}}{r}-2\frac{u_{\theta}}{r}+v_r-\frac{v}{r}\big)=R_v^a,\\[5pt]
u_\theta+(rv)_r=0,  \\[5pt]
u(\theta+2\pi,r)=u(\theta,r), \ v(\theta+2\pi,r)=v(\theta,r), \\[5pt]
u(\theta,1)=0,\ v(\theta,1)=0. \nonumber
\end{array}
\right.
\end{align}}
Let
\begin{align*}
S_u:&=u^au_\theta+v^aru_r+uu^a_\theta+vru^a_r+v^au+vu^a,\\[5pt]
S_v:&=u^av_\theta+v^arv_r+uv^a_\theta+vrv^a_r-2uu^a,
\end{align*}
then the error equations become
\begin{align}\label{e:error equation}
\left\{
\begin{array}{lll}
-\varepsilon^2\big(r u_{rr}
+\frac{u_{\theta\theta}}{r}+2\frac{v_{\theta}}{r}+u_r-\frac{u}{r}\big)+p_\theta+S_u=R_u,\\[5pt]
-\varepsilon^2\big( rv_{rr}+\frac{v_{\theta\theta}}{r}-2\frac{u_{\theta}}{r}+v_r-\frac{v}{r}\big)+rp_r+S_v=R_v,\\[5pt]
u_\theta+(rv)_r=0,  \\[5pt]
u(\theta,r)=u(\theta+2\pi,r),\  v(\theta,r)=v(\theta+2\pi,r), \\[5pt]
u(\theta,1)=0,\ v(\theta,1)=0,
 \end{array}
\right.
\end{align}
where
\begin{align*}
R_u:=R_u^a-uu_\theta-vru_r-vu,\quad R_v:=R_v^a-uv_\theta-vrv_r+u^2.
\end{align*}

\subsection{Linear stability estimate}
\indent

In this subsection, we consider the linear equations
\begin{align}\label{linear equation of error}
\left\{
\begin{array}{lll}
-\varepsilon^2\big(r u_{rr}
+\frac{u_{\theta\theta}}{r}+2\frac{v_{\theta}}{r}+u_r-\frac{u}{r}\big)+p_\theta+S_u=F_u,\\[5pt]
-\varepsilon^2\big( rv_{rr}+\frac{v_{\theta\theta}}{r}-2\frac{u_{\theta}}{r}+v_r-\frac{v}{r}\big)+rp_r+S_v=F_v,\\[5pt]
u_\theta+(rv)_r=0,  \\[5pt]
u(\theta,r)=u(\theta+2\pi,r),\  v(\theta,r)=v(\theta+2\pi,r), \\[5pt]
u(\theta,1)=0,\ v(\theta,1)=0,
 \end{array}
\right.
\end{align}
 and establish its stability estimate. Due to the divergence-free condition and the boundary condition of $v$, we deduce that
\begin{align*}
\int_0^{2\pi}v(\theta,r)d\theta=0,\ r\in (0,1].
\end{align*}
So the Poincar\'{e} inequality is valid:
\begin{align}\label{poin}
\int_0^{2\pi}v^2d\theta\leq\int_0^{2\pi}v_\theta^2d\theta.
\end{align}
We also give the following Hardy-type inequality which will be used frequently.
\begin{Lemma}
If $u(r)$ is a bounded function in $[0,1]$ and $u(1)=0$, then there hold
\begin{align}\label{Hardy type inequality}
&\int_0^1r^\alpha u^2dr\leq C(\alpha)\int_0^1 r^{\alpha+2}u_r^2dr,\ \forall \alpha >-1.\\\label{weight Hardy}
&\int_0^1\frac{ru^2}{(1-r)^2}dr\leq C\int_0^1ru_r^2dr.
\end{align}
\end{Lemma}
\begin{proof}
By integrating by parts, we deduce that
\begin{align*}
\int_0^1r^\alpha u^2dr=-\frac{2}{\alpha+1}\int_0^1 r^{\alpha+2}u u_rdr\leq C(\alpha)\Big(\int_0^1r^\alpha u^2 dr\Big)^{\frac12}\Big(\int_0^1r^{\alpha+2}u_r^2dr\Big)^{\frac12}.
\end{align*}
Thus, there hold
\begin{align*}
\int_0^1r^\alpha u^2dr\leq C(\alpha)\int_0^1 r^{\alpha+2}u_r^2dr,\ \forall \alpha >-1.
\end{align*}
By integrating by parts,
\begin{align*}
\int_0^1\frac{ru^2}{(1-r)^2}dr=\int_0^1ru^2d\frac{1}{1-r}=-\int_0^1\frac{1}{1-r}d(ru^2)=-\int_0^1\Big(\frac{u^2}{1-r}+\frac{ruu_r}{1-r}\Big)dr.
\end{align*}
So we deduce that
\begin{align*}
\int_0^1\frac{ru^2}{(1-r)^2}dr+\int_0^1\frac{u^2}{1-r}dr=-\int_0^1\frac{2ruu_r}{1-r}dr\leq C\Big(\int_0^1\frac{ru^2}{(1-r)^2} dr\Big)^{\frac12}\Big(\int_0^1ru_r^2dr\Big)^{\frac12}.
\end{align*}

Thus, there holds
\begin{align*}
&\int_0^1\frac{ru^2}{(1-r)^2}dr\leq C\int_0^1ru_r^2dr.
\end{align*}
\end{proof}
Since $\chi(r)$ is a smooth cut-off function satisfying $\chi|_{[0,\frac{1}{2}]}=0$, by inequality (\ref{weight Hardy}),
\begin{align}\label{cut Hardy}
\int_0^1\Big(\frac{\chi u^2}{(1-r)^2}+\frac{|\chi'|u^2}{(1-r)^2}\Big)dr\leq C\int_0^1\frac{ru^2}{(1-r)^2}dr\leq C\int_0^1ru_r^2dr.
\end{align}

\subsubsection{Basic energy estimate} 
\indent

In this subsection, we establish the following basic energy estimate of (\ref{linear equation of error}).

 \begin{lemma}\label{basic Energy estimate}
 Let $(u,v)$ be a bounded solution of (\ref{linear equation of error}), then there holds
\begin{align}\label{e:basic Energy estimate}
&\varepsilon^2\int_{0}^{1}\int_0^{2\pi}\frac{ (u_\theta+v)^2+(v_\theta-u)^2}{r}d\theta dr+\varepsilon^2\int_{0}^{1}\int_0^{2\pi}r\big(u_r^2+v_r^2\big)d\theta dr\nonumber\\
\leq& C\int_{0}^{1}\int_0^{2\pi}r(u^2_\theta+v^2_\theta)d\theta dr+C\int_{0}^{1}\int_0^{2\pi}\big[u F_u+vF_v \big]d\theta dr.
\end{align}
\end{lemma}
\begin{proof}
Multiplying the first equation in (\ref{e:error equation}) by $u$, the second equation  in (\ref{e:error equation}) by $v$, adding them  together and integrating in $\Omega$, we obtain that
\begin{align*}
&\underbrace{-\varepsilon^2\int_{0}^{1}\int_0^{2\pi}\Big[\big(ru_{rr}
+\frac{u_{\theta\theta}}{r}+2\frac{v_{\theta}}{r}+u_r-\frac{u}{r}\big)u+ \big(rv_{rr}+\frac{v_{\theta\theta}}{r}-2\frac{u_{\theta}}{r}+v_r-\frac{v}{r}\big)v\Big]d\theta dr}_{diffusion \ term}\\
&+\underbrace{\int_{0}^{1}\int_0^{2\pi}(p_\theta u+p_r rv)d\theta dr}_{pressure \ term}+\underbrace{\int_{0}^{1}\int_0^{2\pi}\big(u S_u +vS_v \big)d\theta dr}_{convective \ term}\\
=&\int_{0}^{1}\int_0^{2\pi}\big[u F_u +vF_v \big]d\theta dr.
\end{align*}
We deal with them term by terms. \\
{\bf The diffusion term:} Integrating by parts, we obtain
\begin{align}\label{e:diffusion estimate in basic energy estimate}
&-\varepsilon^2\int_{0}^{1}\int_0^{2\pi}\Big[\big(ru_{rr}
+\frac{u_{\theta\theta}}{r}+2\frac{v_{\theta}}{r}+u_r-\frac{u}{r}\big)u+ \big(rv_{rr}+\frac{v_{\theta\theta}}{r}-2\frac{u_{\theta}}{r}+v_r-\frac{v}{r}\big)v\Big]d\theta dr\nonumber\\
=&\varepsilon^2\int_{0}^{1}\int_0^{2\pi}\bigg(\frac{ u_\theta^2+v_\theta^2}{r}+\frac{ u^2+v^2}{r}+2\frac{u_\theta v-uv_\theta}{r}\bigg)d\theta dr
+\varepsilon^2\int_{0}^{1}\int_0^{2\pi}r\big(u_r^2+v_r^2\big)d\theta dr\nonumber\\
= &\varepsilon^2\int_{0}^{1}\int_0^{2\pi}\frac{ (u_\theta+v)^2+(v_\theta-u)^2}{r}d\theta dr+\varepsilon^2\int_{0}^{1}\int_0^{2\pi}r\big(u_r^2+v_r^2\big)d\theta dr.
\end{align}
{\bf Pressure term:} Integrating by parts and using the divergence-free condition, we deduce that
\begin{align}\label{e:pressure estimate in basic energy estimate}
\int_{0}^{1}\int_0^{2\pi}\big(p_\theta u+p_r (rv)\big)d\theta dr=0.
\end{align}
{\bf Convective term:} Integrating by part and using the divergence-free condition, we obtain
\begin{align}
&\int_{0}^{1}\int_0^{2\pi}\big(uS_u +vS_v \big)d\theta dr\nonumber\\
=&\int_{0}^{1}\int_0^{2\pi}\big[u^au_\theta u+rv^au_ru+uu^a_\theta u+vu^a_rru+(v^au+vu^a)u\big] d\theta dr\nonumber\\
&+\int_{0}^{1}\int_0^{2\pi}\big[u^av_\theta v+rv^av_rv+uv^a_\theta v+vv^a_r rv-2uu^av\big] d\theta dr\nonumber\\
=&\int_{0}^{1}\int_0^{2\pi}\big[u^2u^a_\theta +v^2 r v^a_r+v^au^2+uvv_\theta^a +uv(ru^a_r-u^a)\big] d\theta dr,\nonumber
\end{align}
where we used
$$\int_{0}^{1}\int_0^{2\pi}(u^au_\theta u+rv^au_ru)d\theta dr=\int_{0}^{1}\int_0^{2\pi}(u^av_\theta v+rv^av_rv)d\theta dr=0.$$
Thus,
\begin{align*}
&\int_{0}^{1}\int_0^{2\pi}\big(uS_u +vS_v \big)d\theta dr\\
=&\int_{0}^{1}\int_0^{2\pi}\big[u^2(u^a_\theta+v^a)+uv(v^a_\theta-u^a+ar)+v^2 r v^a_r+uv(ru^a_r-ar)\big] d\theta dr\\
=&\underbrace{\int_{0}^{1}\int_0^{2\pi}\big(v^2 r v^a_r+uv(ru^a_r-ar)\big) d\theta dr}_{I_1}+
\underbrace{\int_{0}^{1}\int_0^{2\pi}u^2(u^a_\theta+v^a) d\theta dr}_{I_2}
+\underbrace{\int_{0}^{1}\int_0^{2\pi}uv(v^a_\theta-u^a+ar)d\theta dr}_{I_3}.
\end{align*}
{\bf 1)Estimate of $I_1$:} We divide it into Euler part and Prandtl part.
By (\ref{estimate on Euler parts}), the Poincar\'{e} inequality (\ref{poin}) and the Hardy-type inequality (\ref{Hardy type inequality}), we obtain
\begin{align*}
&\Big|\int_{0}^{1}\int_0^{2\pi} (v^2 r \partial_rv^a_e+uv(r\partial_ru^a_e-ar)) d\theta dr\Big|\\
\leq& C\varepsilon \eta\int_{0}^{1}\int_0^{2\pi}(rv^2+r|uv|)d\theta dr\\
\leq& C\varepsilon \eta\int_{0}^{1}\int_0^{2\pi}rv^2d\theta dr+C\varepsilon \eta\Big(\int_{0}^{1}\int_0^{2\pi}rv^2d\theta dr\Big)^{\frac12}\Big(\int_{0}^{1}\int_0^{2\pi}ru^2d\theta dr\Big)^{\frac12}\\
\leq& C\varepsilon \eta\int_{0}^{1}\int_0^{2\pi}rv_\theta^2d\theta dr+C\varepsilon \eta\Big(\int_{0}^{1}\int_0^{2\pi}rv_\theta^2d\theta dr\Big)^{\frac12}\Big(\int_{0}^{1}\int_0^{2\pi}ru_r^2d\theta dr\Big)^{\frac12}\\
\leq& C\int_{0}^{1}\int_0^{2\pi}rv_\theta^2d\theta dr+\frac{\varepsilon^2}{20}\int_{0}^{1}\int_0^{2\pi}ru_r^2d\theta dr.
\end{align*}
Moreover, notice that $\chi(r)=0$ for $r\leq \frac12$, using (\ref{estimate on Prandtl parts}), the Poincar\'{e} inequality (\ref{poin}), and the Hardy inequality (\ref{cut Hardy}), we have
\begin{align*}
&\Big|\int_{0}^{1}\int_0^{2\pi} uvr\partial_r(\chi u^a_p) d\theta dr\Big|\\
\leq&\Big|\int_{0}^{1}\int_0^{2\pi} uvr\chi'(r) u^a_p d\theta dr\Big|+\Big|\int_{0}^{1}\int_0^{2\pi} \frac{ru v}{r-1}\chi(r) Y\partial_Yu^a_pd\theta dr\Big|\\
\leq&\varepsilon\Big|\int_{0}^{1}\int_0^{2\pi} \frac{uvr}{r-1}\chi'(r) Yu^a_p d\theta dr\Big|+\varepsilon\Big|\int_{0}^{1}\int_0^{2\pi} \frac{ru v}{(r-1)^2}\chi(r) Y^2\partial_Yu^a_pd\theta dr\Big|\\
\leq&C\e\Big(\int_{0}^{1}\int_0^{2\pi} rv^2d\theta dr\Big)^\frac12\Big(\int_{0}^{1}\int_0^{2\pi} \frac{ru^2}{(r-1)^2}d\theta dr\Big)^\frac12\\
&+C\e\Big(\int_{0}^{1}\int_0^{2\pi} \frac{r(rv)^2}{(r-1)^2}d\theta dr\Big)^\frac12\Big(\int_{0}^{1}\int_0^{2\pi} \frac{ru^2}{(r-1)^2}d\theta dr\Big)^\frac12\\
\leq&C\e\Big(\int_{0}^{1}\int_0^{2\pi} rv_\theta^2d\theta dr\Big)^\frac12\Big(\int_{0}^{1}\int_0^{2\pi} ru_r^2d\theta dr\Big)^\frac12\\
&+C\e\Big(\int_{0}^{1}\int_0^{2\pi} r[(rv)_r]^2d\theta dr\Big)^\frac12\Big(\int_{0}^{1}\int_0^{2\pi} ru_r^2d\theta dr\Big)^\frac12\\
\leq&C\int_{0}^{1}\int_0^{2\pi}r(u^2_\theta+v_\theta^2)d\theta dr+\frac{\varepsilon^2}{20}\int_{0}^{1}\int_0^{2\pi}ru_r^2d\theta dr.
\end{align*}
The last inequality above used the divergence-free condition $u_\theta=-(rv)_r$. By (\ref{estimate on Prandtl parts}) and  the Poincar\'{e} inequality (\ref{poin}), we deduce that
\begin{align*}
&\Big|\int_{0}^{1}\int_0^{2\pi} v^2r\partial_r( \chi(r)v^a_p) d\theta dr\Big|\leq C\int_{0}^{1}\int_0^{2\pi}rv_\theta^2d\theta dr.
\end{align*}
Thus, we obtain
\begin{align*}
|I_1|\leq C\int_{0}^{1}\int_0^{2\pi}r(u^2_\theta+v_\theta^2)d\theta dr+\frac{\varepsilon^2}{10}\int_{0}^{1}\int_0^{2\pi}ru_r^2d\theta dr.
\end{align*}
{\bf 2)Estimate of $I_2$:}
We decompose $u(\theta,r)=u_0(r)+\tilde{u}(\theta,r)$,  where $u_0(r)=\frac{1}{2\pi}\int_0^{2\pi}u(\theta,r)d\theta$, and notice that
$\int_0^{2\pi}v^a(r,\theta)d\theta=0, \ \forall r\in (0,1]$,  we obtain
\begin{align*}
\int_{0}^{1}\int_0^{2\pi}u^2(u^a_\theta+v^a) d\theta dr =&\int_{0}^{1}\int_0^{2\pi}(u^a_\theta+v^a)(2u_0\tilde{u}+\tilde{u}^2)d\theta dr\\
=&\underbrace{\int_{0}^{1}\int_0^{2\pi}(\partial_\theta u^a_e+v^a_e)(2u_0\tilde{u}+\tilde{u}^2)d\theta dr}_{I_{21}}\\
&+\underbrace{\int_{0}^{1}\int_0^{2\pi}(\partial_\theta \tilde{u}^a_p+\tilde{v}^a_p)(2u_0\tilde{u}+\tilde{u}^2)d\theta dr}_{I_{22}}.
\end{align*}
Moreover, using (\ref{estimate on Euler parts}),  the Poincar\'{e} inequality and Hardy inequality (\ref{Hardy type inequality}), we deduce that
\begin{align*}
|I_{21}|\leq & C\varepsilon \eta \int_{0}^{1}\int_0^{2\pi}r|2u_0\tilde{u}+\tilde{u}^2|d\theta dr\\
 \leq & C\varepsilon \eta \int_{0}^{1}\int_0^{2\pi}ru^2_\theta d\theta dr+C\varepsilon \eta \Big(\int_{0}^{1}\int_0^{2\pi}r\tilde{u}^2 d\theta dr\Big)^{\frac12}\Big(\int_{0}^{1}\int_0^{2\pi}ru^2_0(r) d\theta dr\Big)^{\frac12}\\
 \leq & C\varepsilon \eta \int_{0}^{1}\int_0^{2\pi}ru^2_\theta d\theta dr+C\varepsilon \eta \Big(\int_{0}^{1}\int_0^{2\pi}ru^2_\theta d\theta dr\Big)^{\frac12}\Big(\int_{0}^{1}\int_0^{2\pi}r( u'_0(r))^2 d\theta dr\Big)^{\frac12}\\
 \leq&C\int_{0}^{1}\int_0^{2\pi}ru_\theta^2d\theta dr+\frac{\varepsilon^2}{20}\int_{0}^{1}\int_0^{2\pi}ru_r^2d\theta dr.
\end{align*}
Moreover, notice that $\chi(r)=0$ for $r\leq \frac12$, by (\ref{estimate on Prandtl parts}) and the Hardy inequality (\ref{cut Hardy}), we deduce that
\begin{align}\label{estimate of prandtl part}
|I_{22}|\leq & 2\varepsilon\Big|\int_{\frac12}^{1}\int_0^{2\pi}Y\partial_\theta \tilde{u}^a_p\frac{u_0\tilde{u}}{r-1}d\theta dr\Big|+2\varepsilon\Big|\int_{\frac12}^{1}\int_0^{2\pi} u_0\tilde{u}d\theta dr\Big|+C\int_{\frac12}^{1}\int_0^{2\pi}\chi(r)\tilde{u}^2d\theta dr\nonumber\\
\leq & C\varepsilon \eta \Big(\int_{0}^{1}\int_0^{2\pi}r\tilde{u}^2 d\theta dr\Big)^{\frac12}\Big(\int_{0}^{1}\int_0^{2\pi}\frac{ru^2_0(r)}{(r-1)^2} d\theta dr\Big)^{\frac12}\nonumber\\
&+C\varepsilon  \Big(\int_{0}^{1}\int_0^{2\pi}r\tilde{u}^2 d\theta dr\Big)^{\frac12}\Big(\int_{0}^{1}\int_0^{2\pi}ru^2_0(r) d\theta dr\Big)^{\frac12}+C \int_{0}^{1}\int_0^{2\pi}ru^2_\theta d\theta dr\nonumber\\
 \leq & C\varepsilon \eta \Big(\int_{0}^{1}\int_0^{2\pi}ru^2_\theta d\theta dr\Big)^{\frac12}\Big(\int_{0}^{1}\int_0^{2\pi}r(u'_0(r))^2 d\theta dr\Big)^{\frac12}+C \int_{0}^{1}\int_0^{2\pi}ru^2_\theta d\theta dr\nonumber\\
 \leq&C\int_{0}^{1}\int_0^{2\pi}ru_\theta^2d\theta dr+\frac{\varepsilon^2}{20}\int_{0}^{1}\int_0^{2\pi}ru_r^2d\theta dr.
\end{align}
Thus, we obtain
\begin{align*}
|I_2|\leq C\int_{0}^{1}\int_0^{2\pi}ru_\theta^2d\theta dr+\frac{\varepsilon^2}{10}\int_{0}^{1}\int_0^{2\pi}ru_r^2d\theta dr.
\end{align*}
{\bf 3)Estimate of $I_3$:} First, there holds
\begin{align*}
I_3=\underbrace{\int_{0}^{1}\int_0^{2\pi}uv(\partial_\theta v_e^a-u_e^a+ar)d\theta dr}_{I_{31}}+\underbrace{\int_{0}^{1}\int_0^{2\pi}uv(\partial_\theta \tilde{v}^a_p-\tilde{u}_p^a)d\theta dr}_{I_{32}}.
\end{align*}
Using (\ref{estimate on Euler parts}) and (\ref{Hardy type inequality}), we deduce that
\begin{align*}
|I_{31}|\leq & C\varepsilon \eta \int_{0}^{1}\int_0^{2\pi}r|uv|d\theta dr\\
 \leq&C\int_{0}^{1}\int_0^{2\pi}rv_\theta^2d\theta dr+\frac{\varepsilon^2}{20}\int_{0}^{1}\int_0^{2\pi}ru_r^2d\theta dr.
\end{align*}
The similar estimate as (\ref{estimate of prandtl part}) gives
\begin{align*}
|I_{32}|\leq C\int_{0}^{1}\int_0^{2\pi}rv_\theta^2d\theta dr+\frac{\varepsilon^2}{20}\int_{0}^{1}\int_0^{2\pi}ru_r^2d\theta dr.
\end{align*}
Thus, there holds
\begin{align*}
|I_{3}|\leq C\int_{0}^{1}\int_0^{2\pi}rv_\theta^2d\theta dr+\frac{\varepsilon^2}{10}\int_{0}^{1}\int_0^{2\pi}ru_r^2d\theta dr.
\end{align*}
Finally, summing the estimate of $I_1-I_3$, we obtain
\begin{align}\label{e:linear term in energy estimate}
\Big|\int_{0}^{1}\int_0^{2\pi}\big(uS_u +vS_v \big)d\theta dr\Big|\leq \frac{\varepsilon^2}{3}\int_{0}^{1}\int_0^{2\pi}r(u_r^2+v_r^2)d\theta dr+C\int_{0}^{1}\int_0^{2\pi}r(v_\theta^2+u^2_\theta)d\theta dr.
\end{align}
Collecting the estimates (\ref{e:diffusion estimate in basic energy estimate}), (\ref{e:pressure estimate in basic energy estimate}) and (\ref{e:linear term in energy estimate}), we obtain (\ref{e:basic Energy estimate}).
\end{proof}

\subsubsection{Positivity estimate}
\indent

In this subsection, we establish the following positivity estimate.
\begin{lemma}\label{positivity estimate}
Let  $(u,v)$ be a bounded solution of (\ref{linear equation of error}), then there exists $\eta_0>0$ such that for any $\eta\in (0,\eta_0)$, there holds
\begin{align}\label{e:positivity estimate}
\int_0^1\int_0^{2\pi}r(u_\theta^2+v_\theta^2)d\theta dr
\leq& C\varepsilon^2\eta\Big(\int_{0}^{1}\int_0^{2\pi}\frac{ (u_\theta+v)^2+(v_\theta-u)^2}{r}d\theta dr+\int_{0}^{1}\int_0^{2\pi}r\big(u_r^2+v_r^2\big)d\theta dr\Big)\nonumber\\
&+C\int_{0}^{1}\int_0^{2\pi}\big[u_\theta F_u+v_\theta F_v\big]d\theta dr.
\end{align}
\end{lemma}
\begin{proof}
Multipling the first equation by $u_\theta$ and the second equation by $v_\theta$, integrating in $\Omega$ and summing two terms together, we arrive at
\begin{align}\label{p:positivity estimate}
&\int_{r_0}^{1}\int_0^{2\pi}\Big[-\varepsilon^2\big(r u_{rr}
+\frac{u_{\theta\theta}}{r}+2\frac{v_{\theta}}{r}+u_r-\frac{u}{r}\big)
+p_\theta+S_u\Big]u_\theta d\theta dr\nonumber\\
&+\int_{r_0}^{1}\int_0^{2\pi}\Big[-\varepsilon^2 \big(rv_{rr}+\frac{v_{\theta\theta}}{r}-2\frac{u_{\theta}}{r}+v_r
-\frac{v}{r}\big)+rp_r+S_v\Big]v_\theta d\theta dr\nonumber\\
=&\int_{0}^{1}\int_0^{2\pi}\Big[u_\theta F_u+v_\theta F_v\Big]d\theta dr.
\end{align}
{\bf Positivity term:} We first deal with these terms which is related to $S_u, S_v$.
For simplicity of notation, we set $\bar{u}=u^a-ar, \bar{v}=v^a$, thus we deduce that
\begin{align*}
&\int_{0}^{1}\int_0^{2\pi}[S_uu_\theta+S_v v_\theta] d\theta dr\\
=&\int_{0}^{1}\int_0^{2\pi}\big(aru_\theta+vra+var\big)u_\theta d\theta dr
+\int_{0}^{1}\int_0^{2\pi}\big(arv_\theta-2uar\big)v_\theta d\theta dr\\
&+\int_{0}^{1}\int_0^{2\pi}\big(\bar{u}u_\theta+\bar{v}ru_r+u\bar{u}_\theta+vr\bar{u}_r+\bar{v}u+v\bar{u}\big)u_\theta d\theta dr\\
&+\int_{0}^{1}\int_0^{2\pi}\big(\bar{u}v_\theta+\bar{v}rv_r+u\bar{v}_\theta+vr\bar{v}_r-2u\bar{u}\big)v_\theta d\theta dr\\
=&\underbrace{\int_{0}^{1}\int_0^{2\pi}\big(aru_\theta+2arv\big)u_\theta d\theta dr
+\int_{0}^{1}\int_0^{2\pi}\big(arv_\theta-2aru\big)v_\theta d\theta dr}_{I_1}\\
&+\underbrace{\int_{0}^{1}\int_0^{2\pi}\big(\bar{v}ru_ru_\theta+\bar{v}rv_rv_\theta)d\theta dr}_{I_2}+\underbrace{\int_0^1\int_0^{2\pi}(rv \bar{u}_ru_\theta+rv\bar{v}_rv_\theta) d\theta dr}_{I_3}\\
&+\underbrace{\int_{0}^{1}\int_0^{2\pi}(\bar{u}_\theta+\bar{v})uu_\theta d\theta dr}_{I_4}+\underbrace{\int_{0}^{1}\int_0^{2\pi}\bar{u}u_\theta(v+u_\theta)d\theta dr}_{I_5}\\
&+\underbrace{\int_{0}^{1}\int_0^{2\pi}(v_\theta-u)\bar{u}v_\theta d\theta dr}_{I_6}+\underbrace{\int_{0}^{1}\int_0^{2\pi}uv_\theta(\bar{v}_\theta-\bar{u})d\theta dr}_{I_7}.
\end{align*}
Direct computation gives that
\begin{align*}
I_1=\int_{0}^{1}\int_0^{2\pi}ar(u^2_\theta+v^2_\theta)d\theta dr.
\end{align*}
Next, we estimate $\{I_i\}_{i=2}^7$ term by term.

{\bf 1)Estimate of $I_2$.}
Using (\ref{estimate on Euler parts}) and (\ref{estimate on Prandtl parts}), we obtain
\begin{align*}
&\Big|\int_0^1\int_0^{2\pi}\bar{v}ru_ru_\theta d\theta dr\Big|\\
\leq& C\varepsilon \eta\Big(\int_{0}^{1}\int_0^{2\pi}ru_r^2d\theta dr\Big)^{\frac12}\Big(\int_{0}^{1}\int_0^{2\pi}ru_\theta^2d\theta dr\Big)^{\frac12}\\
\leq& \eta\int_{0}^{1}\int_0^{2\pi}ru_\theta^2d\theta dr+C\varepsilon^2\eta\int_{0}^{1}\int_0^{2\pi}ru_r^2d\theta dr.
\end{align*}
Similarly, there holds
\begin{align*}
&\Big|\int_0^1\int_0^{2\pi}\bar{v}rv_rv_\theta d\theta dr\Big|\\
\leq& \eta\int_{0}^{1}\int_0^{2\pi}rv_\theta^2d\theta dr+C\varepsilon^2\eta\int_{0}^{1}\int_0^{2\pi}rv_r^2d\theta dr.
\end{align*}
Thus, we obtain
\begin{align*}
I_2\leq\eta\int_{0}^{1}\int_0^{2\pi}r(u^2_\theta+v_\theta^2)d\theta dr+C\varepsilon^2\eta\int_{0}^{1}\int_0^{2\pi}r(u_r^2+v_r^2)d\theta dr.
\end{align*}
{\bf 2)Estimate of $I_3$.}
Firstly, by (\ref{estimate on Euler parts}) and  the Poincar\'{e} inequality (\ref{poin}), we obtain
\begin{align*}
&\Big|\int_0^1\int_0^{2\pi}rv \partial_r(u^a_e-ar)u_\theta d\theta dr\Big|\leq C\varepsilon\eta \Big(\int_0^1\int_0^{2\pi}ru_\theta^2d\theta dr\Big)^{\frac12}\Big(\int_0^1\int_0^{2\pi}rv_\theta^2d\theta dr\Big)^{\frac12},\\
&\Big|\int_0^1\int_0^{2\pi}rv \partial_rv^a_ev_\theta d\theta dr\Big|\leq C\varepsilon\eta \int_0^1\int_0^{2\pi}rv_\theta^2d\theta dr.
\end{align*}
Then, notice that $\chi(r)=0$ for $r\leq \frac12$, by (\ref{estimate on Prandtl parts}) and the Hardy inequality (\ref{cut Hardy})
\begin{align*}
&\Big|\int_0^1\int_0^{2\pi}rv \partial_r(\chi u^a_p)u_\theta d\theta dr\Big|\\
=&\Big|\int_0^1\int_0^{2\pi}rv \chi' u^a_pu_\theta d\theta dr+\int_0^1\int_0^{2\pi}rv \chi \partial_ru^a_pu_\theta d\theta dr\Big|\\
=&\Big|\int_0^1\int_0^{2\pi}rv \chi' u^a_pu_\theta d\theta dr+\int_0^1\int_0^{2\pi}\frac{rv}{r-1} \chi Y\partial_Yu^a_pu_\theta d\theta dr\Big|\\
\leq& C\eta \int_0^1\int_0^{2\pi}r(u_\theta^2+v_\theta^2)d\theta dr.
\end{align*}
By (\ref{estimate on Prandtl parts}) and  the Poincar\'{e} inequality (\ref{poin}), we deduce that
\begin{align*}
&\Big|\int_0^1\int_0^{2\pi}rv \partial_r(\chi v^a_p)v_\theta d\theta dr\Big|\\
=&\Big|\int_0^1\int_0^{2\pi}rv \chi' v^a_pv_\theta d\theta dr+\int_0^1\int_0^{2\pi}rv \chi \partial_rv^a_pv_\theta d\theta dr\Big|
\leq C\eta \int_0^1\int_0^{2\pi}rv_\theta^2d\theta dr.
\end{align*}
Thus, there holds
\begin{align*}
|I_3|\leq C\eta \int_0^1\int_0^{2\pi}r(u_\theta^2+v^2_\theta)d\theta dr.
\end{align*}
{\bf 3)Estimate of $I_4$ and $I_7$:} we first decompose $I_4$ into two parts
\begin{align*}
I_4=\underbrace{\int_{0}^{1}\int_0^{2\pi}(\partial_\theta u^a_e+v^a_e)uu_\theta d\theta dr}_{I_{41}}+\underbrace{\int_{0}^{1}\int_0^{2\pi}(\partial_\theta \tilde{u}^a_p+\tilde{v}^a_p)uu_\theta d\theta dr}_{I_{42}}
\end{align*}
 By (\ref{estimate on Euler parts}) and (\ref{Hardy type inequality}), we deduce that
\begin{align*}
|I_{41}|\leq& C\varepsilon\eta \Big|\int_0^1 \int_0^{2\pi}ruu_\theta drd\theta \Big|\\
\leq& C\varepsilon\eta
\Big(\int_0^1 \int_0^{2\pi}ru_r^2 drd\theta\Big)^{\frac12}\Big(\int_0^1 \int_0^{2\pi}ru_\theta^2 drd\theta\Big)^{\frac12}\\
\leq& \eta\int_{0}^{1}\int_0^{2\pi}ru_\theta^2d\theta dr+C\varepsilon^2\eta\int_{0}^{1}\int_0^{2\pi}ru_r^2d\theta dr.
\end{align*}
Moreover, notice that $\chi(r)=0$ for $r\leq \frac12$, by (\ref{estimate on Prandtl parts}) and the Hardy inequality (\ref{cut Hardy}), we deduce that
\begin{align*}
|I_{42}|=&\varepsilon\Big|\int_{0}^{1}\int_0^{2\pi}Y(\partial_\theta \tilde{u}^a_p+\tilde{v}^a_p)\frac{uu_\theta }{r-1}d\theta dr\Big|\\
\leq& \eta\int_{0}^{1}\int_0^{2\pi}ru_\theta^2d\theta dr+C\varepsilon^2\eta\int_{0}^{1}\int_0^{2\pi}ru_r^2d\theta dr.
\end{align*}
Thus, there holds
\begin{align*}
|I_{4}|
\leq\eta\int_{0}^{1}\int_0^{2\pi}ru_\theta^2d\theta dr+C\varepsilon^2\eta\int_{0}^{1}\int_0^{2\pi}ru_r^2d\theta dr.
\end{align*}
Same argument gives
\begin{align*}
|I_7|\leq \eta\int_{0}^{1}\int_0^{2\pi}rv_\theta^2d\theta dr+C\varepsilon^2\eta\int_{0}^{1}\int_0^{2\pi}ru_r^2d\theta dr.
\end{align*}

Thus, there holds
\begin{align*}
|I_4|+|I_7|\leq \eta\int_{0}^{1}\int_0^{2\pi}r(u^2_\theta+v_\theta^2)d\theta dr+C\varepsilon^2\eta\int_{0}^{1}\int_0^{2\pi}ru_r^2d\theta dr.
\end{align*}
{\bf 4)Estimate of $I_5$ and $I_6$:} By (\ref{estimate on Euler parts}) and (\ref{estimate on Prandtl parts}), we deduce that
\begin{align*}
|I_5|\leq& \underbrace{\Big|\int_0^1\int_0^{2\pi}\chi(r)u_p^{(0)}u_\theta(v+u_\theta)d\theta dr\Big|}_{I_{51}}\\
&+C\varepsilon \eta \Big(\int_0^1\int_0^{2\pi} ru^2_\theta d\theta dr\Big)^{\frac 12}
\Big(\int_0^1\int_0^{2\pi} \frac{(u_\theta+v)^2}{r} d\theta dr\Big)^{\frac 12}\\
\leq &I_{51}+\eta\int_{0}^{1}\int_0^{2\pi}ru_\theta^2d\theta dr+C\varepsilon^2\eta\int_0^1\int_0^{2\pi} \frac{(u_\theta+v)^2}{r} d\theta dr.
\end{align*}
Moreover, notice that $\chi(r)=0$ for $r<\frac12$, by (\ref{estimate on Prandtl parts}) and the Hardy inequality (\ref{cut Hardy}), we deduce
\begin{align*}
I_{51}\leq& C\eta \int_0^1\int_0^{2\pi}ru^2_\theta d\theta dr+\varepsilon\Big|\int_0^1\int_0^{2\pi}Yu_p^{(0)}u_\theta \frac{\chi(r)v}{r-1}d\theta dr\Big|\\
 \leq & C\eta \int_0^1\int_0^{2\pi}ru^2_\theta d\theta dr+C\varepsilon\eta\Big(\int_0^1\int_0^{2\pi}ru^2_\theta d\theta dr\Big)^{\frac12}\Big(\int_0^1\int_0^{2\pi}rv^2_r d\theta dr\Big)^{\frac12}.
\end{align*}
Thus, there holds
\begin{align*}
|I_5|\leq C\eta\int_{0}^{1}\int_0^{2\pi}ru^2_\theta d\theta dr
+C\varepsilon^2\eta\Big(\int_0^1\int_0^{2\pi} \frac{(u_\theta+v)^2}{r} d\theta dr+\int_0^1\int_0^{2\pi} rv^2_rd\theta dr\Big).
\end{align*}
Similar argument gives that
\begin{align*}
|I_6|\leq C\eta\int_{0}^{1}\int_0^{2\pi}rv_\theta^2d\theta dr+C\varepsilon^2\eta\Big(\int_0^1\int_0^{2\pi} \frac{(v_\theta-u)^2}{r} d\theta dr+\int_0^1\int_0^{2\pi}ru^2_r d\theta dr\Big).
\end{align*}

Thus, we deduce that
\begin{align*}
|I_5|+|I_6|\leq& C\eta\int_{0}^{1}\int_0^{2\pi}r(u^2_\theta+v_\theta^2)d\theta dr\\
&+C\varepsilon^2\eta\Big(\int_0^1\int_0^{2\pi} \frac{(u_\theta+v)^2+(v_\theta-u)^2}{r} d\theta dr+\int_0^1\int_0^{2\pi} r(u^2_r+v^2_r)d\theta dr\Big).
\end{align*}

Finally, collecting these estimates $I_1-I_7$, we can choose $\eta_0>0$ such that for any $\eta\in (0,\eta_0)$, there holds
\begin{align}\label{e:linear term in positivity estimate}
&\int_{0}^{1}\int_0^{2\pi}[S_uu_\theta+S_v v_\theta] d\theta dr\nonumber\\[5pt]
\quad \quad \geq &(a-C\eta)\int_{0}^{1}\int_0^{2\pi}r(u^2_\theta+v^2_\theta)drd\theta\nonumber\\
&-C\varepsilon^2\eta\Big(\int_{0}^{1}\int_0^{2\pi}r(u_r^2+v_r^2) drd\theta+\int_0^1\int_0^{2\pi} \frac{(u_\theta+v)^2+(v_\theta-u)^2}{r} d\theta dr\Big).
\end{align}
{\bf Pressure estimate:} Integrating by parts, we deduce that
\begin{align}\label{e:pressure estimate in positivity estimate}
\int_{0}^{1}\int_0^{2\pi}p_\theta u_\theta d\theta dr
+\int_{0}^{1}\int_0^{2\pi}rp_rv_\theta d\theta dr=0.
\end{align}
{\bf Diffusion term:}
Finally, we deal with the diffusion term:
\begin{align}\label{diffusive term}
&\int_{0}^{1}\int_0^{2\pi}\Big[-\varepsilon^2\big(r u_{rr}
+\frac{u_{\theta\theta}}{r}+2\frac{v_{\theta}}{r}+u_r-\frac{u}{r}\big)
\Big]u_\theta d\theta dr\nonumber\\
&+\int_{0}^{1}\int_0^{2\pi}\Big[-\varepsilon^2 \big(rv_{rr}+\frac{v_{\theta\theta}}{r}-2\frac{u_{\theta}}{r}+v_r
-\frac{v}{r}\big)\Big]v_\theta d\theta dr\nonumber\\
=&-\varepsilon^2\int_{0}^1\int_0^{2\pi} (r u_{rr}+ u_{r})u_\theta d\theta dr
-\varepsilon^2\int_{0}^1\int_0^{2\pi} (r v_{rr}+ v_{r})v_\theta d\theta dr=0.
\end{align}
Finally, collecting the estimate (\ref{e:linear term in positivity estimate}), (\ref{e:pressure estimate in positivity estimate}) and (\ref{diffusive term}) together, we obtain (\ref{e:positivity estimate}) which completes the proof of this lemma.
\end{proof}

\subsubsection{Linear stability estimate.}
\begin{Proposition}
 Let $(u,v)$ be a bounded solution of (\ref{linear equation of error}), then there exist $\varepsilon_0>0, \eta_0>0$ such that for any $\varepsilon\in (0,\varepsilon_0), \eta\in(0,\eta_0)$, there holds
\begin{align}\label{e:linear stability estimates for linear equation}
&\int_0^1\int_0^{2\pi}r(u_\theta^2+v_\theta^2)d\theta dr+\varepsilon^2\Big(\int_{0}^{1}\int_0^{2\pi}\frac{ (u_\theta+v)^2+(v_\theta-u)^2}{r}d\theta dr+\int_{0}^{1}\int_0^{2\pi}r\big(u_r^2+v_r^2\big)d\theta dr\Big)\nonumber\\
\leq& C\varepsilon^{-2}\int_{0}^{1}\int_0^{2\pi}\Big(\frac{F^2_u}{r}+\frac{F^2_v}{r} \Big)d\theta dr.
\end{align}
\end{Proposition}
\begin{proof}
By combining the estimates (\ref{e:basic Energy estimate}) and (\ref{p:positivity estimate}), we can choose $\varepsilon_0>0, \eta_0>0$ such that for any $\varepsilon\in (0,\varepsilon_0), \eta\in(0,\eta_0)$, there holds
 \begin{align}\label{first estimate of linear stability}
&\int_0^1\int_0^{2\pi}r(u_\theta^2+v_\theta^2)d\theta dr+\varepsilon^2\Big(\int_{0}^{1}\int_0^{2\pi}\frac{ (u_\theta+v)^2+(v_\theta-u)^2}{r}d\theta dr+\int_{0}^{1}\int_0^{2\pi}r\big(u_r^2+v_r^2\big)d\theta dr\Big)\nonumber\\
\leq& C\int_{0}^{1}\int_0^{2\pi}\big(uF_u +vF_v \big)d\theta dr+C\int_{0}^{1}\int_0^{2\pi}\Big[u_\theta F_u+v_\theta F_v\Big]d\theta dr.
\end{align}
By the H\"{o}lder inequality, we deduce that
\begin{align*}
\int_{0}^{1}\int_0^{2\pi}\big(u_\theta F_u +v_\theta F_v \big)d\theta dr\leq& \Big(\int_{0}^{1}\int_0^{2\pi}r\big(u_\theta^2+v_\theta^2\big)d\theta dr\Big)^{\frac12}\Big(\int_{0}^{1}\int_0^{2\pi}\Big(\frac{F^2_u}{r}+\frac{F^2_v}{r} \Big)d\theta dr\Big)^{\frac12},\\
\int_{0}^{1}\int_0^{2\pi}\big(u F_u +v F_v \big)d\theta dr\leq& \Big(\int_{0}^{1}\int_0^{2\pi}r\big(u_r^2+v_r^2\big)d\theta dr\Big)^{\frac12}\Big(\int_{0}^{1}\int_0^{2\pi}\big(F^2_u+F^2_v\big)d\theta dr\Big)^{\frac12},
\end{align*}
where we used (\ref{Hardy type inequality}) in the second inequality.  Putting this into (\ref{first estimate of linear stability}), we obtain (\ref{e:linear stability estimates for linear equation}).
\end{proof}
\section{Existence of error equations}
\indent

In this section we rewrite the error equations and the associated linear stability estimate in Euler coordinates, then obtain $H^2$ estimate by the Stokes estimate in a smooth bounded domain.

\subsection{Error equations in Euler coordinates}
\indent

Let
\begin{align*}
\vec{e}_\theta=
\left(\begin{array}{c}
-\sin \theta \\[5pt]
\cos \theta
\end{array}\right), \ \
\vec{e}_r=
\left(\begin{array}{c}
\cos \theta \\[5pt]
\sin \theta
\end{array}\right)
\end{align*}
and
\begin{align*}
\left(\begin{array}{c}
\tilde{u}(x,y) \\[5pt]
\tilde{v}(x,y)
\end{array}\right)
&=u(\theta, r)\vec{e}_\theta+v(\theta,r)\vec{e}_r, \ \
\left(\begin{array}{c}
\tilde{u}^a(x,y) \\[5pt]
\tilde{v}^a(x,y)
\end{array}\right)
=u^a(\theta, r)\vec{e}_\theta+v^a(\theta,r)\vec{e}_r,\\
\tilde{p}(x,y)&=p(\theta,r), \ \
\left(\begin{array}{c}
\tilde{R}^a_u(x,y) \\[5pt]
\tilde{R}^a_v(x,y)
\end{array}\right)
=\frac{R^a_u(\theta, r)}{r}\vec{e}_\theta+\frac{R^a_v(\theta, r)}{r}\vec{e}_r.
\end{align*}
Then the error equations (\ref{e:error equation}) can be written in Euler coordinates as follows
\begin{align}\label{error equation in Euler coordinates}
\left\{
\begin{array}{lll}
\tilde{u}^a\partial_x \tilde{u}+\tilde{v}^a\partial_y \tilde{u} + \tilde{u}\partial_x \tilde{u}^a +\tilde{v}\partial_y \tilde{u}^a+\partial_x\tilde{p}+\tilde{u}\partial_x \tilde{u}+\tilde{v}\partial_y \tilde{u}-\varepsilon^2 \Delta \tilde{u}= \tilde{R}^a_u,\\[5pt]
\tilde{u}^a\partial_x \tilde{v}+\tilde{v}^a\partial_y \tilde{v} + \tilde{u}\partial_x \tilde{v}^a +\tilde{v}\partial_y \tilde{v}^a+\partial_y\tilde{p}+\tilde{u}\partial_x \tilde{v}+\tilde{v}\partial_y \tilde{v}-\varepsilon^2 \Delta \tilde{v}= \tilde{R}^a_v,\\[5pt]
\partial_x\tilde{u}+\partial_y \tilde{v}=0,\\[5pt]
(\tilde{u},\tilde{v})|_{\partial B_1}=(0,0).
\end{array}
\right.
\end{align}

\subsection{Linear stability estimate in Euler coordinates}
\indent

Let
\begin{align*}
\left(\begin{array}{c}
\tilde{F}_u(x,y) \\[5pt]
\tilde{F}_v(x,y)
\end{array}\right)
=\frac{F_u(\theta, r)}{r}\vec{e}_\theta+\frac{F_v(\theta, r)}{r}\vec{e}_r.
\end{align*}
Then, the equations (\ref{linear equation of error}) become
\begin{align}\label{linear error equation in Euler coordinates}
\left\{
\begin{array}{lll}
\tilde{u}^a\partial_x \tilde{u}+\tilde{v}^a\partial_y \tilde{u} + \tilde{u}\partial_x \tilde{u}^a +\tilde{v}\partial_y \tilde{u}^a+\partial_x\tilde{p}-\varepsilon^2 \Delta \tilde{u}= \tilde{F}_u,\\[5pt]
\tilde{u}^a\partial_x \tilde{v}+\tilde{v}^a\partial_y \tilde{v} + \tilde{u}\partial_x \tilde{v}^a +\tilde{v}\partial_y \tilde{v}^a+\partial_y\tilde{p}-\varepsilon^2 \Delta \tilde{v}= \tilde{F}_v,\\[5pt]
\partial_x\tilde{u}+\partial_y \tilde{v}=0,\\[5pt]
(\tilde{u},\tilde{v})|_{\partial B_1}=(0,0).
\end{array}
\right.
\end{align}
Moreover, from the linear stability estimate (\ref{e:linear stability estimates for linear equation}), we can deduce that there exist $\varepsilon_0>0, \eta_0>0$ such that for any $\varepsilon\in (0,\varepsilon_0), \eta\in(0,\eta_0)$ and any solution $(\tilde{u},\tilde{v})$ of  (\ref{linear error equation in Euler coordinates}), there holds
\begin{align}\label{linear stability estimates in Euler coordinates}
\varepsilon\|\nabla(\tilde{u},\tilde{v})\|_2\leq C\varepsilon^{-1}\|(\tilde{F}_u, \tilde{F}_v)\|_2.
\end{align}

\subsection{Existence of error equations}
\indent

We apply the contraction mapping theorem to prove the existence of the error equations (\ref{e:error equation}).

\begin{Proposition}\label{existence and error estimate of error equation}
There exist $\varepsilon_0>0,\eta_0>0$ such that for any $\varepsilon\in (0,\varepsilon_0), \eta\in (0,\eta_0)$, the error equations (\ref{e:error equation}) have a unique solution $(u,v)$ which satisfies
$$\|(u,v)\|_\infty\leq C\varepsilon.$$
\end{Proposition}
\begin{proof}
For each smooth function $(\tilde{u},\tilde{v})$ which satisfies
\begin{align}\label{iterative conditon}
\left\{
\begin{array}{ll}
 \partial_x\tilde{u}+\partial_y \tilde{v}=0,\\[5pt]
(\tilde{u},\tilde{v})|_{\partial B_1}=(0,0),
\end{array}
\right.
\end{align}
we consider the following linear system
\begin{align}
\left\{
\begin{array}{lll}
\tilde{u}^a\partial_x \bar{u}+\tilde{v}^a\partial_y \bar{u} + \bar{u}\partial_x \tilde{u}^a +\bar{v}\partial_y \tilde{u}^a+\partial_x\bar{p}-\varepsilon^2 \Delta \bar{u}= \tilde{R}^a_u-\tilde{u}\partial_x \tilde{u}-\tilde{v}\partial_y \tilde{u},\\[5pt]
\tilde{u}^a\partial_x \bar{v}+\tilde{v}^a\partial_y \bar{v} + \bar{u}\partial_x \tilde{v}^a +\bar{v}\partial_y \tilde{v}^a+\partial_y\bar{p}-\varepsilon^2 \Delta \bar{v}= \tilde{R}^a_v-\tilde{u}\partial_x \tilde{v}-\tilde{v}\partial_y \tilde{v},\\[5pt]
\partial_x\bar{u}+\partial_y \bar{v}=0,\\[5pt]
(\bar{u},\bar{v})|_{\partial B_1}=(0,0).
\end{array}
\right.
\end{align}

By linear stability estimate (\ref{linear stability estimates in Euler coordinates}), we deduce that there exist $\varepsilon_0>0, \eta_0>0$ such that for any $\varepsilon\in (0,\varepsilon_0), \eta\in(0,\eta_0)$, there holds
\begin{align}\label{linear stability estimate}
\varepsilon\|\nabla(\bar{u},\bar{v})\|_2\leq& C\varepsilon^{-1}\|( \tilde{R}^a_u-\tilde{u}\partial_x \tilde{u}-\tilde{v}\partial_y \tilde{u}, \tilde{R}^a_v-\tilde{u}\partial_x \tilde{v}-\tilde{v}\partial_y \tilde{v})\|_2\nonumber\\[3pt]
\leq & C\varepsilon^{-1}\|( \tilde{R}^a_u, \tilde{R}^a_v)\|_2+ C\varepsilon^{-1}\|(\tilde{u},\tilde{v})\|_\infty \|\nabla(\tilde{u},\tilde{v})\|_2.
\end{align}
Then, by Stokes estimates in smooth domain, we deduce that
\begin{align*}
\varepsilon^2\|\nabla^2(\bar{u},\bar{v})\|_2\leq& C\|\tilde{R}^a_u-\tilde{u}\partial_x \tilde{u}-\tilde{v}\partial_y \tilde{u}-(\tilde{u}^a\partial_x \bar{u}+\tilde{v}^a\partial_y \bar{u} + \bar{u}\partial_x \tilde{u}^a +\bar{v}\partial_y \tilde{u}^a)\|_2\\[3pt]
&+C\|\tilde{R}^a_v-\tilde{u}\partial_x \tilde{v}-\tilde{v}\partial_y \tilde{v}-(\tilde{u}^a\partial_x \bar{v}+\tilde{v}^a\partial_y \bar{v} + \bar{u}\partial_x \tilde{v}^a +\bar{v}\partial_y \tilde{v}^a)\|_2\\[3pt]
\leq &  C\|( \tilde{R}^a_u, \tilde{R}^a_v)\|_2+ C\|(\tilde{u},\tilde{v})\|_\infty \|\nabla(\tilde{u},\tilde{v})\|_2\\[3pt]
&+C\|\tilde{u}^a\partial_x \bar{u}+\tilde{v}^a\partial_y \bar{u} + \bar{u}\partial_x \tilde{u}^a +\bar{v}\partial_y \tilde{u}^a\|_2\\[3pt]
&+C\|\tilde{u}^a\partial_x \bar{v}+\tilde{v}^a\partial_y \bar{v} + \bar{u}\partial_x \tilde{v}^a +\bar{v}\partial_y \tilde{v}^a\|_2.
\end{align*}
Using $\|(\tilde{u}^a, \tilde{v}^a)\|_\infty\leq C$, it is easy to obtain
\begin{align*}
\|\tilde{u}^a\partial_x \bar{u}+\tilde{v}^a\partial_y \bar{u} \|_2
+\|\tilde{u}^a\partial_x \bar{v}+\tilde{v}^a\partial_y \bar{v}\|_2\leq C\|\nabla(\bar{u},\bar{v})\|_2.
\end{align*}
Moreover, let
\begin{align*}
\left(\begin{array}{c}
\bar{u}(x,y)\\[5pt]
\bar{v}(x,y)
\end{array}
\right)=
\hat{u}(\theta, r)\vec{e}_\theta+\hat{v}(\theta,r)\vec{e}_r,
\end{align*}
then there holds
\begin{align*}
&\|\bar{u}\partial_x \tilde{u}^a +\bar{v}\partial_y \tilde{u}^a\|_2^2+\|\bar{u}\partial_x \tilde{v}^a +\bar{v}\partial_y \tilde{v}^a\|_2^2\\[5pt]
=&\int_0^1\int_0^{2\pi}\frac{[\hat{u}(\partial_\theta u^a+v^a)+\hat{v}r\partial_r u^a]^2}{r}d\theta dr+\int_0^1\int_0^{2\pi}\frac{[\hat{u}(\partial_\theta v^a-u^a)+\hat{v}r\partial_r v^a]^2}{r}d\theta dr\\
\leq& C\int_0^1\int_0^{2\pi}r[\hat{u}^2(\theta,r)+\hat{v}^2(\theta,r)]d\theta dr +C\int_0^1\int_0^{2\pi}\frac{r\hat{v}^2(\theta,r)}{(r-1)^2}d\theta dr.
\end{align*}
Obviously, there holds
\begin{align*}
\int_0^1\int_0^{2\pi}r[\hat{u}^2(\theta,r)+\hat{v}^2(\theta,r)]d\theta dr=\|(\bar{u},\bar{v})\|_2^2.
\end{align*}
By the Hardy inequality (\ref{weight Hardy}), there holds
\begin{align*}
\int_0^1\int_0^{2\pi}\frac{r\hat{v}^2(\theta,r)}{(r-1)^2}d\theta dr\leq C\int_0^1\int_0^{2\pi}r(\partial_r\hat{v})^2d\theta dr\leq C \|\nabla(\bar{u},\bar{v})\|_2^2.
\end{align*}
Hence, we have
\begin{align}\label{Stokes estimates}
\varepsilon^2\|\nabla^2(\bar{u},\bar{v})\|_2\leq
 C\|( \tilde{R}^a_u, \tilde{R}^a_v)\|_2+ C\|(\tilde{u},\tilde{v})\|_\infty \|\nabla(\tilde{u},\tilde{v})\|_2+C \|\nabla(\bar{u},\bar{v})\|_2.
\end{align}

 Set
\begin{align*}
\|(u,v)\|_Y:=\varepsilon\|\nabla(u,v)\|_2+\varepsilon^3\|\nabla^2(u,v)\|_2.
\end{align*}
 By (\ref{linear stability estimate}) and (\ref{Stokes estimates}), we deduce that for any $\varepsilon\in (0,\varepsilon_0), \eta\in(0,\eta_0)$, there holds
\begin{align}\label{inductive estimate}
\|(\bar{u},\bar{v})\|_Y
\leq C\varepsilon^{-1}\|( \tilde{R}^a_u, \tilde{R}^a_v)\|_2+ C\varepsilon^{-1}\|(\tilde{u},\tilde{v})\|_\infty \|\nabla(\tilde{u},\tilde{v})\|_2.
\end{align}
Moreover, by the Sobolev embedding, we have
\begin{align*}
\|(\tilde{u},\tilde{v})\|_\infty\leq C \|\nabla(\tilde{u},\tilde{v})\|_2^{\frac34}\|\nabla^2(\tilde{u},\tilde{v})\|_2^{\frac14}.
\end{align*}
Thus, we obtain
\begin{align}\label{iterative estimate}
\|(\bar{u},\bar{v})\|_Y\leq C\varepsilon^{-1}\|( \tilde{R}^a_u, \tilde{R}^a_v)\|_2+C\varepsilon^{-\frac{7}{2}}\|(\tilde{u},\tilde{v})\|_Y^2.
\end{align}
Let $Y=\{(u,v)\in C^\infty: (u,v) \ satisfes \ (\ref{iterative conditon})\ and \ \|(u,v)\|_Y<+\infty\}$.
Thus, due to
\begin{align*}
\|(\tilde{R}_u^a, \tilde{R}_v^a)\|_2\leq C \varepsilon^5,
\end{align*}
there exist $\varepsilon_0>0$ such that for any $\varepsilon\in (0,\varepsilon_0)$, the operator
\begin{align*}
(\tilde{u},\tilde{v})\mapsto (\bar{u},\bar{v})
\end{align*}
maps the ball $\{(u,v): \|(u,v)\|_Y\leq 2C^2\varepsilon^4\}$  in $Y$ into itself. Moreover, for every two pairs
$(\tilde{u}_1, \tilde{v}_1)$ and $(\tilde{u}_2, \tilde{v}_2)$ in this ball, we have
\begin{align}\label{contraction estimate}
\|(\bar{u}_1-\bar{u}_2, \bar{v}_1-\bar{v}_2)\|_Y\leq C\varepsilon^{-\frac{7}{2}}(\|(\tilde{u}_1,\tilde{v}_1)\|_Y+\|(\tilde{u}_2,\tilde{v}_2)\|_Y)\|(\tilde{u}_1-\tilde{u}_2, \tilde{v}_1-\tilde{v}_2)\|_Y.
\end{align}
In fact, set
\begin{align*}
\bar{U}:=\bar{u}_1-\bar{u}_2,\ \bar{V}:=\bar{v}_1-\bar{v}_2, \ \bar{P}=\bar{p}_1-\bar{p}_2,
\end{align*}
then there holds
\begin{align}
\left\{
\begin{array}{lll}
-\varepsilon^2\Delta\bar{U}+\partial_x\bar{P}+
\tilde{u}^a\partial_x\bar{U}+\tilde{v}^a\partial_y \bar{U}+\bar{U}\partial_x \tilde{u}^a+\bar{V}\partial_y \tilde{u}^a=\tilde{R}_U,\\[7pt]
-\varepsilon^2\triangle \bar{V}
+\partial_y\bar{P}+\tilde{u}^a\partial_x\bar{V}+\tilde{v}^a\partial_y \bar{V}+\bar{U}\partial_x \tilde{v}^a+\bar{V}\partial_y \tilde{v}^a=\tilde{R}_V,\\[7pt]
\partial_x \bar{U}+\partial_y\bar{V}=0,  \\[7pt]
 (\bar{U}, \bar{V})|_{\partial B_1}=(0,0),
\end{array}
\right.\nonumber
\end{align}
where
\begin{align*}
\tilde{R}_U=&-\tilde{u}_1\partial_x(\tilde{u}_1-\tilde{u}_2)-\tilde{v}_1\partial_y(\tilde{u}_1-\tilde{u}_2)
-(\tilde{u}_1-\tilde{u}_2)\partial_x\tilde{u}_2-(\tilde{v}_1-\tilde{v}_2)\partial_y\tilde{u}_2,\\[5pt]
\tilde{R}_V=&-\tilde{u}_1\partial_x(\tilde{v}_1-\tilde{v}_2)-\tilde{v}_1\partial_y(\tilde{v}_1-\tilde{v}_2)
-(\tilde{u}_1-\tilde{u}_2)\partial_x\tilde{v}_2-(\tilde{v}_1-\tilde{v}_2)\partial_y\tilde{v}_2.
\end{align*}
Thus, following the estimate (\ref{iterative estimate}) line by line, we obtain (\ref{contraction estimate}).
Hence, there exist $\varepsilon_0>0,\eta_0>0$ such that for any $\varepsilon\in (0,\varepsilon_0), \eta\in (0,\eta_0)$, the operator
\begin{align*}
(u,v)\mapsto (\bar{u},\bar{v})
\end{align*}
maps the ball $\{(u,v): \|(u,v)\|_Y\leq 2C^2\varepsilon^4\}$ in $Y$ into itself and is a contraction mapping. Thus, for any $\varepsilon\in (0,\varepsilon_0), \eta\in (0,\eta_0)$, the error equations (\ref{error equation in Euler coordinates}) have a unique solution $(\tilde{u},\tilde{v})$ which satisfies
\begin{align*}
\|(\tilde{u},\tilde{v})\|_{Y}\leq C\varepsilon^4.
\end{align*}
Hence, there holds $\|(\tilde{u},\tilde{v})\|_\infty\leq C\varepsilon.$  Notice that
\begin{align*}
u(\theta,r)=
\left(\begin{array}{c}
\tilde{u}(x,y)\\[5pt]
\tilde{v}(x,y)
\end{array}
\right)\cdot \vec{e}_\theta, \ \
v(\theta,r)=
\left(\begin{array}{c}
\tilde{u}(x,y)\\[5pt]
\tilde{v}(x,y)
\end{array}
\right)\cdot \vec{e}_r,
\end{align*}
we deduce $\|(u,v)\|_\infty\leq C\varepsilon$. This complete the proof of this proposition.
\end{proof}

\section{Proof of Theorem \ref{main theorem}}
Finally, we give the proof of Theorem \ref{main theorem}.
\begin{proof}
Combining the Proposition \ref{existence and error estimate of error equation} and the approximate solution (\ref{approximate solution}) of the Navier-Stokes equations (\ref{NS-curvilnear}), we easily obtain Theorem \ref{main theorem}.
\end{proof}

\section{Appendix}

{\bf Appendix A:  Constant coefficient periodic PDE}

In this appendix we give a brief argument to solve the following problem
\begin{eqnarray}\label{periodic heat equation}
\left \{
\begin {array}{ll}
(Q_0)_\theta=u_{e}(1)(Q_0)_{\psi\psi},\\[5pt]
Q_0(\theta,\psi)=Q_0(\theta+2\pi,\psi),\\[5pt]
Q_0\big|_{\psi=0}=g(\theta),\ \ Q_0\big|_{\psi\rightarrow-\infty}=0,\label{Q0-0}
\end{array}
\right.
\end{eqnarray}
where
\begin{align*}
g(\theta)=\alpha^2+2\alpha\eta f(\theta)+\eta^2 f^2(\theta)-u_{e}^2(1)=2\alpha\eta f(\theta)+\eta^2 f^2(\theta)-\frac{\eta^2}{2\pi} \int_0^{2\pi}f^2(\theta)d\theta.
\end{align*}

Let  $Q_0(\theta,\psi)=\sum\limits_{k\in\mathbb{Z}}e^{ik\theta}Q_{0k}(\psi)$  and substitute it into (\ref{periodic heat equation}), we obtain
\begin{eqnarray*}
\left \{
\begin {array}{ll}
ikQ_{0k}=u_{e}(1)Q_{0k}'',\\[5pt]
Q_{0k}\big|_{\psi=0}=\widehat{g}(k)=\frac{1}{2\pi}\int_0^{2\pi}e^{-ik\theta}g(\theta) d\theta,
\\[5pt]
Q_{0k}\big|_{\psi\rightarrow-\infty}=0.
\end{array}
\right.
\end{eqnarray*}

It is easy to get
\begin{align*}
Q_{0k}(\psi)=\widehat{g}(k)e^{\alpha_k\psi}
\end{align*}
with
$\alpha_k=\sqrt{\frac{|k|}{2u_e(1)}}(1+\text{sgn}k\cdot i)$. Then
\begin{align*}
Q_0(\theta,\psi)=\sum\limits_{k\in\mathbb{Z}}e^{ik\theta}\widehat{g}(k)e^{\alpha_k\psi}\in X
\end{align*}
and
\begin{align}
\|Q_0\|_X\leq C\eta.\label{norm of q0}
\end{align}

{\bf Appendix B: Construction of corrector $h(\theta,r)$}

In this section, we give a construction of corrector $h(\theta,r)$. Firstly, we give a simple lemma.

\begin{Lemma}\label{corector equation}
Assume that $K(\theta,r)$ is a $2\pi$-periodic smooth function which satisfies
\beno
\int_0^{2\pi}K(\theta,r)d\theta=0, \ \forall r\in (0,1]; \ K(\theta, 1)=0,
\eeno
then there exists a $2\pi$-periodic function $h(\theta,r)$ such that
\begin{align}\label{corrector $h$}
&\partial_\theta h(\theta,r)=K(\theta,r); \quad h(\theta, 1)=0;\nonumber\\
&\int_0^{2\pi}h(\theta,r)d\theta=0, \quad \|\partial_\theta^j\partial_r^kh\|_2\leq C\|\partial_\theta^j\partial_r^kK\|_2.
\end{align}
\end{Lemma}
\begin{proof}
Let
\beno
K(\theta,r)=\sum_{n\neq 0}K_n(r)e^{in\theta}, \ \ K_n(1)=0.
\eeno
Set
\beno
h(\theta,r)=\sum_{n\neq 0}\frac{K_n(r)}{in}e^{in\theta}.
\eeno
It's easy to justify that $h(\theta,r)$ satisfies (\ref{corrector $h$}) which completes the proof.
\end{proof}

Next, we construct the corrector $h(\theta,r)$ by the above lemma.
Direct computation gives
\begin{align*}
u^a_\theta+rv^a_r+v^a=&\varepsilon^5\partial_\theta h(\theta,r)+K(\theta,r),
\end{align*}
where
\begin{align*}
K(\theta,r)=&\varepsilon^5\chi(r)[Y\partial_Yv_p^{(5)}(\theta,Y)+v_p^{(5)}(\theta,Y)]
+r\chi'(r)\Big(\sum_{i=1}^5\varepsilon^i v_p^{(i)}(\theta,Y)\Big).
\end{align*}
Notice that $\chi'(r)=0,\ r\in [0,\frac12]\cup [\frac34,1]$ and the property of $v_p^{(i)}$, we deduce that $K(\theta,r)=O(\varepsilon^5)$ and
\begin{align*}
K(\theta,1)=0.
\end{align*}

Moreover, notice that
\begin{align*}
\int_0^{2\pi}v_p^{i}(\theta,Y)d\theta=0, \ \forall \ Y\leq 0,\quad i=1,\cdot\cdot\cdot,5,
\end{align*}
we deduce that
\begin{align*}
\int_0^{2\pi}K(\theta,r)d\theta=0, \ \forall r\in (0,1].
\end{align*}
Thus, we can choose $h(\theta,r)$ by Lemma \ref{corector equation} such that
\begin{align*}
\varepsilon^5\partial_\theta h(\theta,r)+K(\theta,r)=0, \ h(\theta, 1)=0, \ \|\partial_\theta^j\partial_r^kh\|_2\leq C\varepsilon^{-k}.\\
\end{align*}

{\bf Appendix C:  Prandtl-Batchlor theory in disk.}

For the convenience of readers, we give an introduction of the Prandtl-Batchlor theory, one can see \cite{B, K1998, K2000, OK, W} for more details.

\begin{Theorem}
We consider the steady Navier-Stokes equations in two dimensional simply-connected domain $\Omega$
\begin{eqnarray}\label{NSE in simply-connected domain}
\left \{
\begin {array}{ll}
\mathbf{u}^\varepsilon\cdot\nabla\mathbf{u}^\varepsilon+\nabla p^\varepsilon-\varepsilon^2\Delta\mathbf{u}^\varepsilon=0,\\[5pt]
\nabla\cdot \mathbf{u}^\varepsilon=0,\\[5pt]
\mathbf{u}^\varepsilon\cdot \mathbf{n}\big|_{\partial \Omega}=0, \ \ \mathbf{u}^\varepsilon\cdot \mathbf{t}\big|_{\partial \Omega}=g,
\end{array}
\right.
\end{eqnarray}
where $\mathbf{n}$ is the unit normal vector to $\partial \Omega$ and $\mathbf{t}$ is the unit tangential vector to $\partial \Omega$, $g$ is a smooth function. We assume that (i)the stream function $\psi^\varepsilon$ of equation (\ref{NSE in simply-connected domain}) has no hyperbolic critical point(i.e. nested closed streamlines and single eddy);
(ii)for any $\Omega_1\subset\subset \Omega$, there exist $\varepsilon_0>0$ such that for any $0<\varepsilon<\varepsilon_0$, $\Omega_1$ is away from the boundary layer of equation (\ref{NSE in simply-connected domain}) and $\mathbf{u}^\varepsilon\rightarrow \mathbf{u}^e$ in $C^2(\Omega_1)$, where $\mathbf{u}^e$ is a solution of steady Euler equations in $\Omega$. Then the vorticity $w^e=\nabla\times \mathbf{u}^e$ is a constant in $\Omega$.
\end{Theorem}
\begin{proof}
Let $\mathbf{u}^\varepsilon=(u^\varepsilon,v^\varepsilon)$ and $w^\varepsilon=\partial_yu^\varepsilon-\partial_xv^\varepsilon$ be the vorticity, then it's easy to obtain that
\begin{align*}
\mathbf{u}^\varepsilon\cdot\nabla w^\varepsilon-\varepsilon\Delta w^\varepsilon=0.
\end{align*}
 The boundary is taken to be defined by $\psi^\varepsilon=0$ and $0<\psi^\varepsilon<c_1$ throughout the interior of the eddy. For any $0<c<c_1$, integrating the Navier-Stokes equations over the domain which is surrounded by the closed streamline $\{(x,y)|\psi^\varepsilon(x,y)=c\}$ and using the divergence theorem,
we obtain
\begin{align}\label{necessary condition of vorticity}
\int_{\{\psi^\varepsilon=c\}}\frac{\partial w_\varepsilon}{\partial n}dl=0, \ \ \forall c\in (0,c_1).
\end{align}
Moreover, due to $\mathbf{u}^\varepsilon\rightarrow \mathbf{u}^e$ in $C^2$, then there holds
$w^\varepsilon\rightarrow w^e$ in $C^1$.

Let $\psi^e$ be the associated stream function of Euler equations, then $w^e=F(\psi^e).$  In fact, we introduce
the action-angle transform $(x,y)\rightarrow(r^{\varepsilon},\theta^{\varepsilon})$, where
\begin{align}
r^{\varepsilon}=\psi^{\varepsilon}(x,y)  ,\quad \ \frac{2\pi
}{v^{\varepsilon}(r)}=\oint_{\left\{  \psi^{\varepsilon
}=r^{\varepsilon}\right\}  }\frac{1}{\left\vert \nabla\psi^{\varepsilon
}\right\vert },\quad \theta^{\varepsilon}=v^{\varepsilon}(r)
\int_{0}^{l}\frac{dl^{\prime}}{\left\vert \nabla\psi^{\varepsilon}\right\vert
},\nonumber
\end{align}
and $l$ is the arc-length variable on the curve $\left\{  \psi^{\varepsilon
}=r^{\varepsilon}\right\}  $. Then the transformation $(x,y)
\rightarrow(r^{\varepsilon},\theta^{\varepsilon})$ has Jacobian
$1$ and the operator $u^{\varepsilon}\cdot\nabla$ becomes $v^{\varepsilon
}(r)\partial_{\theta^{\varepsilon}}$. In $\Omega_{1}$, $(
r^{\varepsilon},\theta^{\varepsilon})  \rightarrow(  r^{0}%
,\theta^{0})$ in $C^{1}$ and $r^{0}=\psi^{e}$. Then $(\psi
^{e},\theta^{0})  $ is a coordinate system and the steady vorticity
$\omega^{e}$ is a single-valued function $F(\psi^{e})$.

Taking $\varepsilon\rightarrow 0$ in (\ref{necessary condition of vorticity}), we obtain
\begin{align*}
F'(c)=0, \ \forall c\in (0,c_1).
\end{align*}
Thus, $w_e$ is a constant in $\Omega$.
\end{proof}

\begin{Remark}
In the disk $B_1(0)$: due to $\vec{u}_e=u_e(\theta,r)\vec{e}_\theta+v_e(\theta,r)\vec{e}_r$, we deduce
\begin{align*}
w_e=\frac{1}{r}(\partial_r(ru_e)-\partial_\theta v_e)=a, \quad v_e|_{\partial\Omega}=0.
\end{align*}
Solving this equation, we obtain
$$(u_e,v_e)=\Big(\frac{a}{2}r, 0\Big),$$
where $a$ is any constant.
\end{Remark}

\section*{Acknowledgments}

M.Fei is partially supported by NSF of China under Grant No.11871075 and 11971357. Z.Lin is partially supported by the NSF grants DMS-1715201 and DMS-2007457. T.Tao is partially supported by the NSF of China under Grant 11901349 and the NSF of Shandong Province grant ZR2019QA001. C.Gao and T.Tao are deeply grateful to professor L.Zhang for very valuable suggestion.

\end{document}